\documentclass[11pt,leqno]{article}

\usepackage{mathrsfs,amssymb,amsmath,amsthm,amsfonts, color}
\usepackage[dvips]{graphicx}

\makeatletter
	
	\@addtoreset{equation}{section}
\makeatother

\begin{document}

\renewcommand{\theenumi}{\rm (\roman{enumi})}
\renewcommand{\labelenumi}{\rm \theenumi}

\setcounter{footnote}{1}

\newtheorem{thm}{Theorem}[section]
\newtheorem{defi}[thm]{Definition}
\newtheorem{lem}[thm]{Lemma}
\newtheorem{prop}[thm]{Proposition}
\newtheorem{cor}[thm]{Corollary}
\newtheorem{exam}[thm]{Example}
\newtheorem{conj}[thm]{Conjecture}
\newtheorem{rem}[thm]{Remark}
\allowdisplaybreaks

\title{The invariant measure and the flow associated to the $\Phi ^4_3$-quantum field model}

\author{Sergio Albeverio$^\dagger$ and Seiichiro Kusuoka$^\ddagger$
\vspace{5mm}\\
\normalsize $^\dagger$Institute for Applied Mathematics and HCM, University of Bonn\\
\normalsize Endenicher Allee 60, 53115 Bonn, Germany \\
\normalsize e-mail address: albeverio@iam.uni-bonn.de\vspace{5mm}\\
\small $^\ddagger$ Department of Mathematics, Graduate School of Science, Kyoto University,\\
\small Kitashirakawa-Oiwakecho, Sakyo-ku, Kyoto 606-8502, Japan\\
\small e-mail address: {\tt{kusuoka@math.kyoto-u.ac.jp}}}
\maketitle

\begin{abstract}
We give a direct construction of invariant measures and global flows for the stochastic quantization equation to the quantum field theoretical $\Phi ^4_3$-model on the $3$-dimensional torus.
This stochastic equation belongs to a class of singular stochastic partial differential equations (SPDEs) presently intensively studied, especially after Hairer's groundbreaking work on regularity structures.
Our direct construction exhibits invariant measures and flows as limits of the (unique) invariant measures for corresponding finite-dimensional approximation equations.
Our work is done in the setting of distributional Besov spaces, adapting semigroup techniques for solving nonlinear dissipative parabolic equations on such spaces and using methods that originated from work by Gubinelli et al on paracontrolled distributions for singular SPDEs.
\end{abstract}

{\bf AMS Classification Numbers:}
81T08, 81S20, 60H15, 35Q40, 35R60, 35K58

 \vskip0.2cm


\section{Introduction}\label{sec:intro}

The present paper undertakes a new and direct construction of global solutions with general initial conditions and the invariant measure for a nonlinear stochastic partial differential equation (stochastic quantization equation) associated with the $\Phi ^4_3$-model of quantum field theory on a torus.
To understand the origins of the problem and present some motivation for the study of the $\Phi ^4_3$-model, let us shortly recall the origins of quantum field theory and the motivations for the construction of quantum field models.

The origins of quantum field theory have to be found already at the beginning of quantum theory.
In fact the considerations which lead M. Planck at the beginning of last century to the introduction of the basic ``quantum of action'' (expressed by Planck's constant $\hbar$) were based on a phenomenon (``black body radiation'') involving the electromagnetic field (described by Maxwell's equations).
Quantum theory evolved first (1924-25) as a physical theory, different from classical mechanics, for the description of phenomena characterized by a dependence on $\hbar$, typical of the world of atoms and molecules.
Later it found its well-known mathematical formulation in terms of operators acting in Hilbert spaces (see e.g. \cite{ReSi1, ReSi2, ReSi3, ReSi4}).
Already in 1927, M. Born, W. Heisenberg and P. Jordan considered an analogue of quantum mechanics where the particles are replaced by fields.
This was quite natural since a field (e.g. the classical electromagnetic field) in any bounded space-time domain after a decomposition in Fourier components can be looked upon as an infinite system of oscillators, susceptible to be quantized as mechanical particles performing oscillations.
In the same year P. Dirac gave the first physical discussion of a quantized electromagnetic field in interaction with quantized particles (see e.g. \cite{Jo1, Jo2}).
Soon it was realized that divergences arise in trying to compute quantities of physical interest.
This is largely due to the fact of having to do with an infinite-dimensional quantum system which evolves according to the laws of relativity theory.
Despite the fact that quantum mechanics of finite systems of non-relativistic particles found a mathematical formulation quite early, the extension to the case of quantum fields took a lot of time and in some sense is still an open problem.
However, in the case without interaction (``free field case'') a suitable setting was found through the Fock space representation (since the 30s) and (since the 60s) the isomorphic Friedrichs-Segal representation of Fock spaces as an $L^2$-space with respect to a suitable Gaussian measure on the space of real maps from the space variables to the real numbers.
The singularities of this measure coupled with the nonlinearity of the interaction makes difficult the treatment of the inclusion of interactions.
These difficulties lead in the 50s on one hand to the physical theory of renormalization, on the other hand to the development of ``axiomatic settings'', trying to fix a minimal set of requirements for a theory or a model to be acceptable.
Up to the present no model satisfying all requirements has been found for the case where the dimension $d$ of space-time is $4$.
In the case where $d\leq 3$ some nontrivial models satisfying all requirements have been constructed, as part of the area of research developed in the 60s-70s known under the name of ``constructive quantum field theory'' (see e.g. \cite{AHFL}, \cite{BSZ}, \cite{GlJa}, \cite{HKPS}, \cite{Si}).
The $\Phi ^4_3$-model, which we discuss in the present paper, belongs to this area, more precisely to the class of models which can be looked upon as quantized versions of a classical nonlinear partial differential equation of the form
\begin{equation}\label{eq:intro1}
\frac{\partial ^2}{\partial t^2} \phi (t,\vec x) = \left( \triangle _{\vec x} -m_0^2\right) \phi (t,\vec x) - V' (\phi (t,\vec x)) .
\end{equation}
Here $m_0>0$ is a constant, $t$ and $\vec x$ are time and space variables, respectively, $t\in {\mathbb R}$, $\vec x \in {\mathbb R}^\sigma$, $\sigma \in {\mathbb N} \cup \{ 0\}$, $\phi$ takes real values, $V$ is a real-valued differentiable map on ${\mathbb R}$ expressing the nonlinearity of the equation.
More precisely, $-V'(\phi )$ is the nonlinearity for the equation.
In the $\Phi ^4_3$ case we have $\sigma =2$, $V(y)= \lambda y^4$ for some $\lambda >0$ (more general models have been discussed for $\sigma \leq 1$, where $V$ can be of the lower bounded polynomial, trigonometric on exponential-type (see e.g. \cite{AGH}, \cite{Al}, \cite{Al16}, \cite{Di}, \cite{GlJa} and \cite{ReSi3})).
(\ref{eq:intro1}) is called Klein-Gordon equation (with mass $m_0$ and nonlinearity given by $V$).
For the study of (\ref{eq:intro1}) and similar classical  nonlinear PDEs see, e.g. \cite{BeF} and \cite{VeWi77}.
It is a prototype of relativistic local equations, in as much as it can be looked upon as a local perturbation by the $V$-term of the relativistic linear equation expressed by (\ref{eq:intro1}) for $V\equiv 0$ (the linear Klein-Gordon equation, which is obviously relativistic covariant, since it only involves the relativistic operator $\Box = \frac{\partial ^2}{\partial t^2} - \triangle _{\vec x}$).

The quantum field $\Phi _{\rm qu}$ corresponding to the classical field $\phi $ satisfying (\ref{eq:intro1}) has been realized in the models mentioned above as an operator-valued distribution, satisfying all requirements of a relativistic quantum field theory in space-time dimension $d:=\sigma +1 \leq 3$ (as mentioned above, the most interesting case where $d=4$ is still out of reach, despite several partial results, see, e.g \cite{AlSe}, \cite{FMRS}, \cite{FeFrSo}, \cite{GRS76} and \cite{So}).

A common construction of $\Phi _{\rm qu}$ for all $d\leq 3$ (within the above mentioned ``constructive quantum field theoretical approach) is by probabilistic methods, where one first constructs a generalized random field $\Phi _{\rm Eu}$ (where $\rm Eu$ stands for ``Euclidean") defined as the coordinate process to a probability measure $\mu _{\rm Eu}$ (depending on $m_0$ and $V$) on the probability space $\Omega = {\mathcal S}'({\mathbb R}^d)$, with its Borel $\sigma$-algebra.
The measure $\mu _{\rm Eu}$ is invariant under the (full) Euclidean group $E_d$ acting on ${\mathcal S}' ({\mathbb R}^d)$.
$\Phi _{\rm Eu}$ is thus $E_d$-homogeneous (stationary).
All axioms of Euclidean field theory are satisfied, and from the moments functions of $\mu _{\rm Eu}$ (which have been shown to exist) one can find, by a suitable analytic continuation, a set of functions, called Wightman functions, which characterize the relativistic quantum field $\Phi _{\rm qu}$ corresponding to $\Phi _{\rm Eu}$.
These $\Phi _{\rm qu}$ are ``nontrivial'' in the sense that they differ both physically and mathematically from the corresponding quantities for $V \equiv 0$ (see e.g. \cite{BLOT}, \cite{GlJa} and \cite{Si}).

Let us indicate briefly how the structure $(\Phi _{\rm Eu},\mu _{\rm Eu})$ is constructed in the cases $\sigma =1,2$ (for the more elementary but also instructive case $\sigma =0$ (nonlinear quantum oscillator) see \cite{CoRe}, \cite{PrYo} and \cite{Str}).
$\mu _{\rm Eu}$ is obtained by a double limit, introducing both a space-time cut off (also called ``infrared cut-off'') and a regularization cut-off (``ultraviolet cut-off'').
The first is realized either by considering the interaction-term only for $(t,\overrightarrow{x})$ in a bounded region $\Lambda$ of ${\mathbb R}^d$, and putting appropriate boundary conditions on $\partial \Lambda$ for the space-time Laplacian in ${\mathbb R}^d$ (with Euclidean metric), or by replacing ${\mathbb R}^d$ itself by a $d$-dimensional torus ${\mathbb T}^d$.
The ultraviolet cut-off is realized in two steps: first by plainly replacing in the interaction term the coordinate variable by a regularized version of it (e.g. through convolution with a mollifier, depending on a parameter $\varepsilon >0$); the second step consists in introducing appropriate renormalization counterterms as we shall see.
As a result of the first step one has then a family of probability measures $\mu _{\Lambda ,\varepsilon}$ on ${\mathcal S}'({\mathbb R}^{d})$ of the form
\begin{equation}\label{eq:intro2}
\mu _{\Lambda ,\varepsilon} (d\omega ) = Z_{\Lambda ,\varepsilon}^{-1} e^{-\int _{\Lambda}V(\omega _{\varepsilon}(t,\vec x)) dt d\vec x} \mu _0(d\omega )
\end{equation}
where $\omega \in \Omega = {\mathcal S}'({\mathbb R}^d)$, $\Omega$ denoting the probability space ($\omega$ plays the role of $\Phi _{\rm Eu}$); $\mu _0$ is the probability measure corresponding to the case where in (\ref{eq:intro1}) we have $V\equiv 0$, which is Nelson's free field measure $\mu _0$ on ${\mathcal S}'({\mathbb R}^d)$, i.e. the Gaussian measure with mean $0$ and covariance operator $(-\triangle +m_0^2)^{-1}$ in $L^2({\mathbb R}^d)$, respectively when $\Lambda$ is the torus ${\mathbb T}^d$, in $L^2({\mathbb T}^d)$ (and then $\mu _0$ can be seen as a measure on ${\mathcal S}'(\Lambda )$) (see, e.g. \cite{GlJa}, \cite{GRS75}, \cite{GRS76}, \cite{Ne73a}, \cite{Ne73b} and \cite{Si}).
To keep in touch with suggestive notations of the physical literature $\mu_0$ is heuristically given by a normalization times
\[
\exp \left( -\frac 12 \int [ \dot \omega ^2(t,\vec x) + |\nabla \omega (t,\vec x) | ^2 + m_0^2 \omega ^2(t,\vec x) ] dt d\vec x \right) \prod _{t,\vec x}d\omega (t,\vec x).
\]
$Z_{\Lambda , \varepsilon}$ in (\ref{eq:intro2}) is a normalization constant.
Note that $(t,\vec x)$ are meant to run over ${\mathbb R}^d$ resp. ${\mathbb T}^d$, in the former case it is the $\mu _{\rm Eu}$ which corresponds to $V\equiv 0$ in (\ref{eq:intro1}).
As it stands the limit of $\mu _{\Lambda, \varepsilon}$ for $\varepsilon \downarrow 0$ (removal of the regularization given by $\varepsilon >0$) does not exist even when $V$ has a simple form, e.g. $V(y)= \lambda y^4 /4$ for $y\in {\mathbb R}$ with a constant $\lambda >0$ (this model is called $\Phi ^4_d$-model).
For $d=2$ a replacement of $\omega _{\varepsilon}^4(t,\vec x)$ by the Wick ordered power $:\omega_{\varepsilon}^4(t,\vec x):$ to $\omega _{\varepsilon}(t,\vec x)$ (renormalization by Wick ordering; for Wick ordered powers see, e.g. \cite{DPTu}, \cite{Ne73b} and \cite{Si}) suffices, in the sense that the moments of the measure
\[
Z _{\Lambda , \varepsilon} ^{-1} \exp \left( -\frac{\lambda}{4}\int _{\Lambda} :\omega_{\varepsilon}^4(t,\vec x): dt d\vec x \right) \mu _0 (d\omega )
\]
($Z_{\Lambda ,\varepsilon}$ being again a suitable normalization constant), converge as $\varepsilon \downarrow 0$ to the moments of a probability measure $\mu _{\Lambda}$ on ${\mathcal S}'({\mathbb R}^2)$.
Moreover, (for $\lambda /m_0^2$ small enough, ``weak coupling case") the latter moments converge as $\Lambda \uparrow {\mathbb R}^2$ to the moments of a probability measure $\mu _{\rm Eu}$ on ${\mathcal S}'({\mathbb R}^2)$.
$\mu _{\rm Eu}$ is singular with respect to $\mu _0$, whereas $\mu _{\Lambda}$ was still absolutely continuous with respect to $\mu _0$.
For these and other results on the $\Phi ^4_2$-model, including its relevance as yielding a model of relativistic quantum fields, see e.g. \cite{AKR}, \cite{Fr}, \cite{GlJa}, \cite{GRS75} and \cite{Si}.

\begin{rem}\label{rem:intro1}
Let us make two side remarks:
\begin{enumerate}
\item \label{rem:intro1-1} for $d\leq 2$ other interesting models have been constructed, e.g. for $V$ a lower bounded polynomial (see e.g. \cite{GlJa} and \cite{Si}), or $V$ of exponential or trigonometric type (see e.g. \cite{AlHo}, \cite{HKPS} and references therein),

\item the $\Phi ^4_2$-model and related ones are also relevant for other areas of research, like condensed matter physics (Allen-Cahn model of phase separation), image analysis, hydrodynamics, or nonlinear phenomena (see e.g. \cite{AlFe}, \cite{LeNR}, \cite{FKSS}, \cite{Rou}).
\end{enumerate}
\end{rem}

The construction of a corresponding $\Phi ^4_3$-model is more complicated and less detailed results have been established.
The main difference in the construction with respect to the one for the $\Phi ^4_2$-model is that the renormalization needed to obtain $\mu _{\Lambda}$ (from $\mu _{\Lambda ,\varepsilon}$) involves, besides Wick ordering, the insertion of a divergent second order ``mass renormalization'' term and to perform the limit $\varepsilon \downarrow 0$ more detailed estimates had to be established.
Basic steps for this were made by J. Glimm \cite{Gl68} and J. Glimm and A. Jaffe \cite{GlJa73}, who developed a Hamiltonian approach (see also \cite{GlJa}, \cite{GlJa1} and \cite{GlJa2}).
J. Feldman constructed the moments of a measure corresponding to $\mu_\Lambda$, $\Lambda$ being now a bounded subset of ${\mathbb R}^3$ (see \cite{Fe}).

The proof of convergence of the moments of $\mu _{\Lambda}$ as $\Lambda \uparrow {\mathbb R}^3$ to the moments of a Euclidean $\Phi ^4_3$-measure $\mu _{\rm Eu}$ is also more indirect, but it has been achieved in \cite{FeOs}, \cite{MaSe} and \cite{SeSi}, for the model defined by replacing in the expression for $\mu_{\Lambda ,\varepsilon}$ in (\ref{eq:intro2}) the term
\[
\frac{\lambda}{4} \int_{\Lambda} : \omega_\varepsilon^4 (t,\vec x): dt d\vec x
\]
by
\[
\frac{\lambda}{4} \int_\Lambda [\omega_\varepsilon^4(t,\vec x) + a(\varepsilon, \lambda)\omega_\varepsilon^2(t,\vec x)] dt dx,
\]
with $a(\varepsilon, \lambda) := -\alpha\lambda\varepsilon ^{-1}+ \beta\lambda^2\ln (\varepsilon ^{-1}) + \sigma$, with suitable constants $\alpha, \beta$ and for $\lambda > 0, \sigma \in \mathbb R$ (cf. \cite{GlJa}). In these references it is then shown that for $\sigma$ sufficiently large compared to $\lambda$ (``weak coupling'') the moments of $\mu_{\Lambda,\varepsilon}$ converge as $\varepsilon \downarrow 0$, $\Lambda \uparrow \mathbb R^3$ to the moments of a unique probability measure $\mu_{\rm E}$.
The limit satisfies the axioms of a Euclidean model and by analytic continuation a relativistic model is obtained. 
$\mu_E$ is non-Gaussian, its moments have an asymptotic expansion in powers of $\lambda$ to all orders \cite{BoFe}, its Borel summability is also proven \cite{MaSe}.
On the other hand, non-uniqueness of the limit for sufficiently small $\sigma$ is shown in \cite{FrSiSp}. 

\begin{rem}
Another approach was developed in \cite{Pa2} for the case where $\Lambda$ is the $3$-dimensional torus and $\mu_0$ is looked upon as a probability measure on the corresponding ${\mathcal S}'(\Lambda)$ space.
On the basis of estimates in \cite{Pa2} (Theorem 1.1(c), Theorem 3.5) and \cite{Fe} (Theorem 1d) it is argued to be unique.
The coincidence of the limits when $\Lambda=[-L,L]^d$, $L \uparrow \infty$, (extending functions on $\Lambda$ periodically with period $(2L)^d$) of the moments of $\mu_\Lambda$ defined in \cite{Pa2} with the moments of the Euclidean invariant measure $\mu _{\rm Eu}$ discussed in \cite{FeOs} and \cite{MaSe} in the ``weak coupling case'' is only hinted to in \cite{Pa2}.
Another result on the $\Phi_3^4$-model on the $3$-dimensional torus is in \cite{Pa1}, where the homogeneous term $\frac{\lambda}{4}\varphi^4$ is replaced by $\frac{\lambda}{4}\varphi^4-\sigma\varphi^2-\mu\varphi$, with $\sigma > 0, \mu \in {\mathbb R}$.
Here a corresponding $\mu _\Lambda$ is constructed by first replacing the $\Lambda$ by a lattice $\Lambda _\delta$ of mesh $\delta > 0$, then letting $\delta \to 0$, and showing (Corollary 4.3) the convergence of the moments of $\mu_{\Lambda_\delta}$ to the moments of a unique limit measure $\tilde \mu_\Lambda$. 
$\tilde \mu_\Lambda$ is then studied in the limit $\Lambda \uparrow {\mathbb R}^3$ and brought in contact with the above Euclidean measure $\mu_{\rm Eu}$ (on ${\mathcal S}'(\mathbb R^3)$, as discussed in \cite{FeOs} and \cite{MaSe}), in the case where $\sigma$ is sufficiently large compared to $\lambda$ (which corresponds to the above weak coupling case).
Further results on the $\Phi_3^4$-model are presented, e.g. in \cite{AlLi2}, \cite{Bat}, \cite{BCGNOPS}, \cite{BDH}, \cite{BrFrSo}, \cite{MWX} and \cite{So}.
\end{rem}

Recent important developments initiated by M. Hairer \cite{Ha} are concerned with the construction of an SPDE of the heuristic form (\ref{eq:intro3}) below, and as such being related, in the case where $V(y)= \lambda y^4/4$ ($y\in {\mathbb R}$ and $\lambda >0$) with the heuristic expansion for the probability measure $\mu _{\rm Eu}$ of the Euclidean approach to the $\Phi ^4_d$-model, in the sense that $\mu _{\rm Eu}$ is a candidate for an invariant measure for the solution of (\ref{eq:intro3}) for such a $V$.
The general idea of considering an SDE having a measure of interest as an invariant measure goes back to work by G. Parisi and Y. S. Wu \cite{PaWu}.
In the context of quantum field theory this has taken the name of ``stochastic quantization method''.
For the case of the structure $(\Phi _{\rm Eu}, \mu _{\rm Eu})$ associated to the classical equation (\ref{eq:intro1}) the stochastic quantization method yields the equation
\begin{equation}\label{eq:intro3}
dX_\tau = [(\triangle -m_0^2)X_\tau - V'(X_\tau )] dt + dW_\tau
\end{equation}
where $dW_\tau$ is a Gaussian white noise in the new $\tau \in [0,\infty )$-variable and in the old space-time variables $(t,\vec x)\in {\mathbb R}\times {\mathbb R}^\sigma = {\mathbb R}^d$, relative to which $\triangle$ is taken.
Thus $X_{\tau}(t,\vec x)$ is for any given $\tau$ thought as a random field in the Euclidean space-time variables $(t,\vec x)$.
$\tau$ is thought as a ``computer time''.
Heuristically, assuming that the solution flow to (\ref{eq:intro3}) exists and is ergodic one can compute $\mu _{\rm Eu}$-averages like $\int F d\mu _{\rm Eu}$, for suitable integrable $F$, from limits of $\tau$-averages $\frac 1T \int _0^T F(X_\tau ) d\tau$ as $T\rightarrow \infty$.
This program has been implemented mathematically for $d=1$ and $V$, e.g. of the type of those in the $\Phi ^4_d$-model, in \cite{Iwa} where existence and uniqueness of solutions of (\ref{eq:intro3}) and their properties have been discussed (see also \cite{KaRo}).

In the case $d=2$ correspondingly as for the construction mentioned above of a Euclidean measure for models over ${\mathbb R}^2$, one achieves the construction of solutions of (\ref{eq:intro3}) for $V$, e.g. of the form $V(y)= \lambda y^4/4$ for $y\in {\mathbb R}$ with $\lambda >0$ (or more generally for the class mentioned in Remark \ref{rem:intro1}\ref{rem:intro1-1}), by suitably modifying the nonlinear term $V'$ in (\ref{eq:intro3}).
E.g. for the above quartic $V$ one replaces $-\lambda X_\tau ^3$ in (\ref{eq:intro3}) by a Wick ordered version $-\lambda :X_\tau ^3:$ of it.
The first solution by a Wick ordered version of the so modified (\ref{eq:intro3}), both on a $2$-dimensional torus and on ${\mathbb R}^2$ has been realized in \cite{AR} by the method of Dirichlet forms (see also \cite{BChM} and \cite{Mi}) (solutions are here in the weak probabilistic sense), for quasi-every initial conditions.
Solutions in a strong sense have been obtained by other essentially analytic methods in \cite{DPTu}.
In \cite{DPDe} G. Da Prato and A. Debussche introduced the method of exploiting the Ornstein-Uhlenbeck process $Z_\tau$ associated with the linear part in (\ref{eq:intro3}) and replacing the process of $Z_{\tau}$ arising from the nonlinear term in $X_{\tau}= (X_{\tau} -Z_{\tau}) + Z_{\tau}$ by corresponding Wick powers; this method has been extended to more singular SPDEs by Hairer and Gubinelli, see below.
In \cite{DPDe} ergodicity results for the solution process have been obtained.
See also \cite{AMR} for a survey of results on the stochastic quantization equation for the $\Phi ^4_2$-model and discussion of uniqueness problems.
For a proof of restricted Markov uniqueness of dynamics associated with the $\Phi ^4_2$-model see \cite{RZZ}, which uses also results of \cite{MW2}, \cite{MW2} providing also a new construction of strong solutions in certain negative index Besov spaces for this stochastic quantization equation.
For a derivation of the stochastic quantization equation from Kac-Ising models see \cite{FrRu}, \cite{GLP}, \cite{HaIb} and \cite{MW17}.
For work on Gaussian white noise driven PDEs related to other models of quantum fields in $2$-dimensional space-time see \cite{AKaMR}, \cite{AHaRu} and \cite{HaSh}. 
Let us also add that much work has been done on related SPDEs with more regular noise and having as a common invariant measure the $\Phi ^4_2$-measure $\mu$ (see \cite{BChM}, \cite{JoMi} and the references in \cite{AMR}).

The situation with the stochastic quantization of the $\Phi ^4_3$-model remained open for a long time, except for a partial result in \cite{ALZ} until the ground breaking work by M. Hairer \cite{Ha, Ha2}.
Hairer's methods are essentially PDE's ones in spaces of generalized functions ($C^\alpha$ with $\alpha$ negative) and are rooted in Gubinelli's extension of T. Lyons' rough path methods to the case of multidimensional time \cite{Gu1, Gu2}.
Hairer's break through in producing solutions of the stochastic quantization equation to the $\Phi ^4_3$-model (which following his work is also named (equation of the) dynamical $\Phi ^4_3$-model) generated an intensive activity in the area of singular SPDE, also for other SPDEs, in particular using Gubinelli's adaption of the method of paracontrolled distributions for SPDEs (see e.g. \cite{CaCh}, \cite{FuG}, \cite{GuHo} and \cite{GIP}).

We shall limit ourselves here to mention work specifically related to the $\Phi ^4_3$-model.
The original work by Hairer proved the existence of local (in time) solutions of (\ref{eq:intro3}) on the $3$-dimensional torus ${\mathbb T}^3$, after a renormalization procedure inspired by the one used for the construction of the $\Phi ^4_3$-measure $\mu _{\rm Eu}$, in the weak coupling case.
The space on which the solutions are located is a $C^\alpha$-space, for any $\alpha \in (-2/3, -1/2)$, of generalized functions, for initial conditions which are also in the same $C^\alpha$.
Various approximation results for the solutions have been derived subsequently, see \cite{HaX} (from other interaction terms) and \cite{HaMa}, \cite{ZZ1} (from a lattice approximation).
The local well-posedness of (\ref{eq:intro3}) on ${\mathbb T}^3$ has also been proved successively by other methods (see \cite{Kup}).
Existence and uniqueness of local solutions on ${\mathbb T}^3$ have been obtained in \cite{CaCh} by the method of paracontrolled distributions.
The extension to local solutions of (\ref{eq:intro3}) on ${\mathbb R}^3$ (the case associated with the original resp. Euclidean model) was discussed in \cite{HaLa} and \cite{HaMa} by introducing suitable weights.
The extension from local to global solutions in the case of ${\mathbb T}^3$ is discussed in \cite{Ha}, and in \cite{MW3} by an interplay of the paracontrolled approach in \cite{GIP} with Bourgain's method, exploiting the presence of the candidate for an invariant measure, namely the weak coupling case $\Phi^4_3$-measure as discussed in \cite{BrFrSo}.
It is asserted in the abstract of \cite{MW3} that the existence of invariant measures follows from the proven uniform bounds on solutions ``via the Krylov-Bogoliubov method'' (details are not given in the paper).
For the relation of such invariant measures with ``the $\Phi ^4_3$-measure" of quantum field theory see \cite{HaMa} and \cite{MW3}.

\begin{rem}
\begin{enumerate}

\item In \cite{HaMatt} a method for establishing the strong Feller property of processes associated with SPDEs of the form (\ref{eq:intro3}) is presented.
In particular, the strong Feller property of the process of the stochastic quantization equation on ${\mathbb T}^3$ constructed in \cite{HaMa} and \cite{MW3} is established, for initial data of suitable regularity.

\item To the best of our knowledge, all papers discussing invariant measures for the stochastic quantization equation (over ${\mathbb R}^3$ and ${\mathbb T}^3$) use a `` $\Phi ^4_3$-measure" as presented in constructive quantum field theory, rather than constructing them directly; one exception being \cite{MW3}, in which as we already commented above an invariant measure is considered to follow from the proven uniform bounds on global solutions in the relevant Besov spaces.
One main aim of the present paper -and of the partly related paper \cite{GuHo} - is precisely to provide a direct construction of invariant measures, see below.
The uniqueness of invariant measures remains open in all approaches.

\item \cite{ZZ2} introduces a Dirichlet form associated with the solution process of the dynamical $\Phi ^4_3$-model over ${\mathbb T}^3$ discussed in \cite{MW3}, whereas \cite{AlYo} relates to work in \cite{ALZ} by associating a family of positive bilinear forms to the weak coupling $\Phi ^4_3$-measure on ${\mathbb R}^3$.

\item Results of the type of those of \cite{RZZ} established for the restricted Markov uniqueness of the dynamical $\Phi ^4_2$-model, seem however, to remain open for the $\Phi ^4_3$-model, both on ${\mathbb T}^3$ and on ${\mathbb R}^3$.

\end{enumerate}
\end{rem}

In the present paper, we consider the stochastic quantization (\ref{eq:intro3}) with $(t, \vec x) \in \Lambda = {\mathbb T}^3$, $V(y)= \lambda y^4/4$ for $y\in {\mathbb R}$ and $\lambda >0$ ($\Phi ^4_3$-model on ${\mathbb T}^3$).
Differently from other approaches, we do not consider pointwise initial conditions for the regularized equation, but rather a family of finite-dimensional SDEs approximations with their invariant measures as initial condition.
More precisely, we consider the well-defined finite-sum approximation $\{ \mu _N; N\in {\mathbb N}\}$ (defined at the beginning of Section \ref{sec:main}) of the Fourier expansion of the (heuristic) $\Phi ^4_3$-measure on ${\mathbb T}^3$, and discuss the nonlinear stochastic partial differential equations given by the stochastic quantization of the approximation measures $\mu _N$.
Denote by $X^N$ the solution to the finite-dimensional approximation equations with the initial distribution $\mu _N$.
The difference from all other approaches to the study of the stochastic quantization equation of the $\Phi ^4_3$-model mentioned above is that in our case the initial distribution of $X^N$ is given by $\mu _N$.
In our setting we have then the advantage of being able to exploit the stationarity of $X^N$. 
To construct a limit process we will prove a uniform estimate for $\{ X^N\}$, which implies the tightness of its laws.
For this we use Hairer's reconstruction method of singular stochastic partial differential equations.
The renormalization will appear in the reconstruction.
The tightness yields a limit process for a suitable subsequence of $\{ X^N\}$.
In particular, we obtain the convergence of the marginal distribution of the subsequence, which is the limit of the subsequence of $\{ \mu _N\}$ in view of the stationarity of $X^N$.
This is the strategy for our direct construction of an invariant measure and a flow associated with the $\Phi ^4_3$-stochastic quantization equation on ${\mathbb T}^3$.
The strategy seems natural being much in the spirit of the treatment of stochastic differential equations based on It\^o calculus, and in this sense it is a natural extension of it.
This seems to be a natural method also in relation to the variational approach to SPDEs (for a related approach see \cite{GuHo}). 
It is expected that our method can be extended to other singular semilinear SPDEs with Gaussian white noise, having finite-dimensional approximations with invariant measures.
The extension will be model-dependent and will however require separate estimates.

The organization of the present paper is as follows.
The material in Section \ref{sec:Besov} and \ref{sec:OU} is introductory.
Although it is related to \cite{BCD}, \cite{CaCh}, \cite{GuHo}, \cite{GIP}, \cite{Ha}, \cite{Ha2}, \cite{MW17}, \cite{MW3}, \cite{MW2} and \cite{MWX}, many detailed estimates needed for our main results are not to be found in these references.
In Section \ref{sec:Besov} we give the definition of Besov spaces and the notation of paraproducts.
Paraproducts appear when we consider the partial differential equation reconstructed from (\ref{eq:intro3}), and we solve the reconstructed equation in Besov spaces that are useful for our later deductions.
We also prepare some function inequalities, which are applied for obtaining estimates of each term in the reconstructed partial differential equations.
In Section \ref{sec:OU} we introduce the infinite-dimensional Ornstein-Uhlenbeck process solving the linear part of the stochastic quantization equation (\ref{eq:intro3}) and the polynomials associated with this process.
The polynomials of the Ornstein-Uhlenbeck process also were used in Hairer's reconstruction method and related works, and their renormalization is required for proving the convergence in the Besov spaces which we need in the rest of the paper. 
In Section \ref{sec:main} we consider the stochastic quantization equations associated to the measures which approximate the (candidate for a) $\Phi ^4_3$-measure on ${\mathbb T}^3$.
This constitutes the main part of the present paper.
We first apply Hairer's reconstruction method, and obtain a solvable partial differential equation with random coefficients.
Next we prove many estimates for each term in the partial differential equation and an associated energy functional, which appears in the typical approach to dissipative nonlinear partial differential equations and enables us to control the nonlinear terms.
In the estimate for the energy functional new terms appear.
So, we reiterate the procedure to be able to estimate the new terms which appeared, and then keep repeating the procedure until finally obtain a uniform estimate, which yields the tightness of the solutions to the approximation equations.
From this our main results follow in a natural way.

\section{Besov spaces and estimates of functions}\label{sec:Besov}

In this section, we introduce the Besov spaces relevant for our work, as well as the paraproducts and functional inequalities that we shall use.
Let $\Lambda$ be the $3$-dimensional torus, i.e. $({\mathbb R}/2\pi {\mathbb Z})^3$ with the natural Lebesgue measure $dx$ induced from the one on ${\mathbb R}^3$.
Let $L^p$ and $W^{s,p}$ be the corresponding $p$th-order integrable function space and the Sobolev space on $\Lambda$, for $s\in {\mathbb R}$ and $p\in [1,\infty ]$, respectively.
Let $\chi$ and $\varphi$ be functions in $C^\infty ([0,\infty );[0,1])$ such that the supports of $\chi$ and $\varphi$ are included by $[0,4/3)$ and $[3/4, 8/3]$ respectively,
\begin{align*}
&\chi (r ) + \sum _{j=0}^\infty \varphi (2^{-j}r ) =1, \quad r\in [0,\infty ),\\
&\varphi (2^{-j}r )  \varphi (2^{-k}r ) =0, \quad r \in [0,\infty ),\ j,k \in {\mathbb N}\cup \{ 0\} \ \mbox{such that}\ |j-k| \geq 2,\\
&\chi (r ) \varphi (2^{-j}r ) =0, \quad r \in [0,\infty),\ j \in {\mathbb N}.
\end{align*}
For the existence of $\chi$ and $\varphi$, see Proposition 2.10 in \cite{BCD}.
Throughout this paper, we fix $\chi$ and $\varphi$.
Moreover, even if the constants that will appear in the estimates below depend on $\chi$ and $\varphi$, we do not mention explicitly this dependence.

Let ${\mathcal S}({\mathbb R}^3)$ and ${\mathcal S}'({\mathbb R}^3)$ be the Schwartz space and the space of tempered distributions on ${\mathbb R}^3$, respectively.
For $f\in {\mathcal D}'(\Lambda )$ where ${\mathcal D}'(\Lambda )$ is the topological dual of $C^\infty (\Lambda )$, we can define the periodic extension $\widetilde{f} \in {\mathcal S}'({\mathbb R}^3)$ (see Section 3.2 in \cite{ScTr}).
By means of this extension, we define the (Littlewood-Paley) nonhomogeneous dyadic blocks $\{ \Delta _j; j\in {\mathbb N} \cup \{ -1, 0\}\}$ by setting
\[
\begin{array}{lll}
\Delta _{-1} f (x)&= \left[ {\mathcal F}^{-1} \left( \chi (|\cdot |) {\mathcal F} \widetilde{f} \right) \right] (x), & x\in \Lambda \\
\Delta _j f (x)&= \left[ {\mathcal F}^{-1} \left( \varphi (2^{-j} |\cdot |){\mathcal F} \widetilde{f} \right) \right] (x), & x\in \Lambda, \ j \in {\mathbb N}\cup \{ 0\} ,
\end{array}
\]
where ${\mathcal F}$ and ${\mathcal F}^{-1}$ are the Fourier transform and inverse Fourier transform operators, i,e. $\mathcal F$ is the automorphism of ${\mathcal S}'({\mathbb R}^3)$ given by the extension of the map
\[
g \mapsto \widehat g (\xi )= \frac{1}{(2\pi)^{3/2}}\int _{{\mathbb R}^3} g(x) e^{-\sqrt{-1} x\cdot \xi} dx , \quad g\in {\mathcal S}({\mathbb R}^3),
\]
where $x\cdot \xi := \sum _{j=1}^3 x_j \xi _j$ for $x=(x_1,x_2,x_3), \xi =(\xi _1, \xi _2, \xi _3) \in {\mathbb R}^3$, and ${\mathcal F}^{-1}$ is the inverse operator of ${\mathcal F}$, respectively (see Section 1.2 in \cite{BCD}).
As a family of pseudo-differential operator, $\{ \Delta _j; j\in {\mathbb N} \cup \{ -1, 0\}\}$ is given by
\[
\Delta _{-1} f = \chi \left( \sqrt{-\triangle} \right) f, \quad \Delta _j f = \varphi \left( 2^{-j} \sqrt{-\triangle} \right) f \quad j \in {\mathbb N}\cup \{ 0\}
\]
where $\triangle$ is the Laplace operator for the functions on $\Lambda$, i.e.
\[
\triangle f (x) := \left( \frac{\partial ^2}{\partial x_1^2} + \frac{\partial ^2}{\partial x_2^2} + \frac{\partial ^2}{\partial x_3^2} \right) f(x), \quad x =(x_1,x_2,x_3) \in {\Lambda },\ f\in C^\infty (\Lambda ).
\]

We define the Besov norm $\| \cdot \| _{B_{p,r}^s}$ and the Besov space $B_{p,r}^s$ on $\Lambda$ with $s \in {\mathbb R}$ and $p,r \in [1,\infty]$ by
\begin{align*}
\| f \| _{B_{p,r}^s} &:= \left\{ \begin{array}{ll}
\displaystyle \left( \sum _{j=-1}^\infty 2^{jsr} \| \Delta _j f \| _{L^p}^r \right) ^{1/r} , & r\in [1,\infty) ,\\
\displaystyle \sup _{j \in {\mathbb N}\cup \{ -1,0\}} 2^{js} \| \Delta _j f \| _{L^p} , & r= \infty ,
\end{array} \right. \\
B_{p,r}^s& := \{ f\in {\mathcal D}'(\Lambda ); \| f \| _{B_{p,r}^s} <\infty \} .
\end{align*}
It is easy to see that $B_{p_1,r_1}^{s_1} \subset B_{p_2,r_2}^{s_2}$ for $s_1,s_2 \in {\mathbb R}$ and $p_1,p_2, r_1,r_2\in [1,\infty ]$ such that $s_1\geq s_2$, $p_1\geq p_2$ and $r_1\leq r_2$.
It is known that $B_{p,p}^s =W^{s,p}$ for $s\in {\mathbb R}\setminus {\mathbb Z}$ and $p\in [1,\infty ]$.
It is also known that $B_{p,\infty}^s \subset B_{p,1}^{s'}$ for $p\in [0,\infty ]$ and $s,s' \in {\mathbb R}$ such that $s'<s$
(see Corollary 2.96 in \cite{BCD}).
For simplicity of notation, we denote $B_{p,\infty }^s$ by $B_p^s$ for $s \in {\mathbb R}$ and $p \in [1,\infty]$.

Next we prepare the notation and estimates of paraproducts by following Chapter 2 in \cite{BCD}.
Let
\[
S_j f := \sum _{k =-1}^{j-1} \Delta _{k} f, \quad j\in {\mathbb N}\cup \{ 0\} .
\]
For simplicity of notation, let $\Delta _{-2}f :=0$ and $S_{-1}f :=0$.
We define
\begin{align*}
f \mbox{\textcircled{\scriptsize$<$}} g &:= \sum _{j=0}^\infty (S_j f) \Delta _{j+1} g ,\quad f \mbox{\textcircled{\scriptsize$>$}} g := g \mbox{\textcircled{\scriptsize$<$}} f, \\
f \mbox{\textcircled{\scriptsize$=$}} g &:= \sum _{j=-1}^\infty \Delta _{j} f  \left( \Delta _{j-1} g + \Delta _{j} g + \Delta _{j+1} g\right) 
\end{align*}
By the definitions of $\{ \Delta _j\}$, $\{ S_j\}$, $\mbox{\textcircled{\scriptsize$<$}}$, $\mbox{\textcircled{\scriptsize$=$}}$, and $\mbox{\textcircled{\scriptsize$>$}}$, we have
\[
fg = f \mbox{\textcircled{\scriptsize$<$}} g + f \mbox{\textcircled{\scriptsize$=$}} g + f \mbox{\textcircled{\scriptsize$>$}} g
\]
Let $f \mbox{\textcircled{\scriptsize$\leqslant$}} g := f \mbox{\textcircled{\scriptsize$<$}} g + f \mbox{\textcircled{\scriptsize$=$}} g$ and $f \mbox{\textcircled{\scriptsize$\geqslant$}} g := f \mbox{\textcircled{\scriptsize$>$}} g + f \mbox{\textcircled{\scriptsize$=$}} g$.

We summarize the fundamental estimates of Besov norms and paraproducts in the following proposition.
Note that here and in the whole paper constants $C$ appearing on the right-hand side of estimates are always meant as positive, without mention it.

\begin{prop}\label{prop:paraproduct}
\begin{enumerate}
\item \label{prop:paraproduct1} For $s\in {\mathbb R}$ and $p_1, p_2 ,r_1, r_2 \in [1,\infty ]$ such that $p_1\leq p_2$ and $r_1\leq r_2$,
\[
\| f \| _{B_{p_2,r_2}^{s-3(1/p_1-1/p_2)}} \leq C \| f\| _{B_{p_1,r_1}^s} , \quad f \in B_{p_1,r_1}^{s} ,
\]
where $C$ is a constant depending on $s$.

\item \label{prop:paraproduct2} For $s\in {\mathbb R}$ and $p, p_1, p_2 ,r \in [1,\infty ]$ such that $1/p = 1/p_1 + 1/p_2$,
\[
\| f \mbox{\textcircled{\scriptsize$<$}} g \| _{B_{p,r}^s} \leq C\| f\| _{L^{p_1}} \| g\| _{B_{p_2,r}^s}, \quad f \in L^{p_1} ,\ g \in B_{p_2,r}^s ,
\]
where $C$ is a constant depending on $s$.

\item \label{prop:paraproduct3} For $s\in {\mathbb R}$, $t\in (-\infty ,0)$, and $p, p_1, p_2 , r, r_1,r_2 \in [1,\infty ]$, such that
\[
\frac 1p = \frac{1}{p_1} + \frac{1}{p_2} \quad \mbox{and}\quad \frac 1r = \min \left\{ 1, \frac{1}{r_1} + \frac{1}{r_2} \right\},
\]
it holds that
\[
\| f \mbox{\textcircled{\scriptsize$<$}} g \| _{B_{p,r}^{s+t}} \leq C\| f\| _{B_{p_1,r_1}^t} \| g\| _{B_{p_2,r_2}^s}, \quad f \in B_{p_1 ,r_1}^t ,\ g \in B_{p_2,r_2}^s ,
\]
where $C$ is a constant depending on $s$ and $t$.

\item \label{prop:paraproduct4} For $s_1 ,s_2\in {\mathbb R}$ such that $s_1+s_2>0$, and $p, p_1, p_2, r, r_1,r_2 \in [1,\infty ]$, such that
\[
\frac 1p = \frac{1}{p_1} + \frac{1}{p_2} \quad \mbox{and}\quad \frac 1r = \frac{1}{r_1} + \frac{1}{r_2},
\]
it holds that
\[
\| f \mbox{\textcircled{\scriptsize$=$}} g \| _{B_{p,r}^{s_1+s_2}} \leq C\| f\| _{B_{p_1,r_1}^{s_1}} \| g\| _{B_{p_2,r_2}^{s_2}}, \quad f \in B_{p_1,r_1}^{s_1} ,\ g \in B_{p_2,r_2}^{s_2} ,
\]
where $C$ is a constant depending on $s_1$ and $s_2$.

\item \label{prop:paraproduct5}
For $s\in (0,\infty )$, $\varepsilon \in (0,1)$ and $p, p_1, p_2 ,r \in [1,\infty ]$ such that $1/p = 1/p_1 + 1/p_2$,
\[
\| f^2\| _{B_{p,r}^s} \leq C \| f\| _{L^{p_1}} \| f\| _{B_{p_2,r}^{s+\varepsilon }} , \quad f\in L^{p_1}\cap B_{p_2,r}^{s+\varepsilon}
\]
where $C$ is a constant depending on $s$, $\varepsilon$, $p_1$, $p_2$ and $r$.

\item \label{prop:paraproduct5.5}
For $s\in (0,\infty )$, $\varepsilon \in (0,1)$ and $p,r \in [1,\infty ]$,
\[
\| f^3\| _{B_{p,r}^s} \leq C \| f\| _{L^{4p}}^2 \| f\| _{B_{2p,r}^{s+\varepsilon }} , \quad f\in L^{4p}\cap B_{2p,r}^{s+\varepsilon}
\]
where $C$ is a constant depending on $s$, $\varepsilon$, $p$ and $r$.

\item \label{prop:paraproduct6} For $s_1, s_2 \in {\mathbb R}$, $p, p_1, p_2 , r, r_1,r_2 \in [1,\infty ]$ such that $1/p = 1/p_1+1/p_2$ and $1/r = 1/r_1 + 1/r_2$ and $\theta \in (0,1)$,
\[
\| f\| _{B_{p,r}^{\theta s_1 + (1-\theta )s_2}} \leq \| f\| _{B_{p_1,r_1}^{s_1}}^{\theta } \| f\| _{B_{p_2,r_2}^{s_2}}^{1-\theta}  , \quad f\in B_{p_1,r_1}^{s_1} \cap B_{p_2,r_2}^{s_2} .
\]

\item \label{prop:paraproduct7}
 For $s\in {\mathbb R}$ and $p_1, p_2 ,r_1,r_2 \in [1,\infty ]$ such that $1= 1/p_1+1/p_2$ and $1= 1/r_1 + 1/r_2$, there exists a constant $C$ depending on $s$, satisfying
\[
\left| \int _\Lambda f(x) g(x) dx \right| \leq C \| f\| _{B_{p_1,r_1}^{-s}} \| g\| _{B_{p_2,r_2}^{s}} , \quad f\in B_{p_1,r_1}^{-s} ,\ g \in B_{p_2,r_2}^{s} .
\]

\item \label{prop:paraproduct8}
For $\alpha \in {\mathbb R}$, $\beta \in [0,\infty )$, and $p,r\in [1,\infty ]$, there exists a constant $C$ depending on $\alpha$ and $\beta$, such that
\[
\| e^{t\triangle} f\| _{B_{p,r}^{\alpha}} \leq C (1+t^{-\beta}) \| f\| _{B_{p,r}^{\alpha -2 \beta}} ,\quad f\in B_{p,r}^{\alpha - 2\beta},
\]
where $\{ e^{t\triangle }; t\geq 0\}$ is the heat semigroup generated by $\triangle $ on $L^2(\Lambda ;{\mathbb C})$.
\end{enumerate}
\end{prop}

\begin{proof}
The proofs of  \ref{prop:paraproduct1}, \ref{prop:paraproduct4}, \ref{prop:paraproduct6} and \ref{prop:paraproduct7} are similar to the case of the functions on the whole space ${\mathbb R}^3$.
See Proposition 2.71, Theorem 2.85, Theorem 2.80 and Proposition 2.76 in \cite{BCD} for \ref{prop:paraproduct1}, \ref{prop:paraproduct4}, \ref{prop:paraproduct6} and \ref{prop:paraproduct7}, respectively.
For \ref{prop:paraproduct2}, \ref{prop:paraproduct3} and \ref{prop:paraproduct8} , see Proposition A.7 and A.13 in \cite{MW3}.
By \ref{prop:paraproduct3} and \ref{prop:paraproduct4}, we have
\begin{align*}
\| f^2\| _{B_{p,r}^s} &\leq \| f\mbox{\textcircled{\scriptsize$=$}} f\| _{B_{p,r}^s} + 2 \| f\mbox{\textcircled{\scriptsize$<$}} f\| _{B_{p,r}^s} \\
&\leq C \| f\| _{B_{p_1,\infty }^{-\varepsilon}} \| f\| _{B_{p_2,r}^{s+\varepsilon }}.
\end{align*}
Hence, the fact that $L^{p_1} \subset W^{-\varepsilon ,p_1} \subset B_{p_1,\infty }^{-\varepsilon}$ yields \ref{prop:paraproduct5}.
To prove \ref{prop:paraproduct5.5}, applying \ref{prop:paraproduct2}, \ref{prop:paraproduct3} and \ref{prop:paraproduct4} again, we have
\begin{align*}
\| f^3\| _{B_{p,r}^s} &\leq \| f^2 \mbox{\textcircled{\scriptsize$<$}} f\| _{B_{p,r}^s} + \| f^2 \mbox{\textcircled{\scriptsize$=$}} f\| _{B_{p,r}^s} + \| f\mbox{\textcircled{\scriptsize$<$}} f^2\| _{B_{p,r}^s} \\
&\leq C \| f^2\| _{B_{2p,\infty }^{-\varepsilon}} \| f\| _{B_{2p,r}^{s+\varepsilon }} + C \| f\| _{L^{4p}} \| f^2\| _{B_{4p/3,r}^{s}}.
\end{align*}
Hence, \ref{prop:paraproduct5} and the fact that $L^{2p} \subset W^{-\varepsilon ,2p} \subset B_{2p,\infty }^{-\varepsilon}$ yield
\[
\| f^3\| _{B_{p,r}^s} \leq C \| f^2\| _{L^{2p}} \| f\| _{B_{2p,r}^{s+\varepsilon }} + C \| f\| _{L^{4p}}^2 \| f\| _{B_{2p,r}^{s}}.
\]
This inequality implies \ref{prop:paraproduct5.5}.
\end{proof}

Now we prepare some functional inequalities which we shall apply in the proof of our main theorem.

\begin{lem}\label{lem:Lpestimates1}
\begin{enumerate}
\item \label{lem:Lpestimates1-1} For $\theta \in (0,9/16)$ there exists a constant $C$ depending on $\theta$ such that for $\delta \in (0,1]$ and $f\in L^4\cap B_2^{15/16}$
\[
\| f\| _{B_2^{\theta }} \leq \delta \left( \| f\| _{L^4}^4 + \| f\| _{B_2^{15/16}}^2\right) ^{7/8} + C\delta ^{-16/19} .
\]

\item \label{lem:Lpestimates1-2}For $\theta \in (0,9/16)$ there exists a constant $C$ depending on $\theta$ such that for $\delta \in (0,1]$ and $f\in L^4\cap B_2^{15/16}$
\[
\| f^2\| _{B_{4/3}^{\theta }} \leq \delta \left( \| f\| _{L^4}^4 + \| f\| _{B_2^{15/16}}^2\right) ^{7/8} + C\delta ^{-26/9} . 
\]

\item \label{lem:Lpestimates1-3} For $\theta \in (0,9/16)$ there exists a constant $C$ depending on $\theta$ such that for $\delta \in (0,1]$ and $f\in L^4\cap B_2^{15/16}$
\[
\| f^3\| _{B_1^{\theta}} \leq \delta \left( \| f\| _{L^4}^4 + \| f\| _{B_2^{15/16}}^2\right) + C \delta ^{-10} .
\]
\end{enumerate}
\end{lem}

\begin{proof}
Let $f\in L^4\cap B_2^\theta$.
By Proposition \ref{prop:paraproduct}\ref{prop:paraproduct6} we have
\[
\| f\| _{B_2^{\theta }} \leq C \| f\|_{L^2}^{2/5} \| f\| _{B_2^{5\theta /3}}^{3/5} \leq C \| f\|_{L^2}^{2/5} \| f\| _{B_2^{15/16}}^{3/5}
\]
where $C$ is a positive constant depending on $\theta $.
By applying the fact that
\begin{equation}\label{eq:lemLpestimates1-1}
ab \leq \theta a^{1/ \tilde \theta} + (1- \tilde \theta) b^{1/(1- \tilde \theta )} \leq a^{1/\tilde \theta} + b^{1/(1- \tilde \theta )},  \quad a,b\in [0,\infty ),\ \tilde \theta \in (0,1)
\end{equation}
we have
\[
\| f\| _{B_2^{\theta }} \leq C \left( \| f\|_{L^2}^{8/5} + \| f\| _{B_2^{15/16}}^{4/5} \right) \leq C \left( \| f\|_{L^4}^4 + \| f\| _{B_2^{15/16}}^2 \right) ^{2/5} + C.
\]
Hence, applying (\ref{eq:lemLpestimates1-1}) again, we obtain \ref{lem:Lpestimates1-1}.

Next, we show that \ref{lem:Lpestimates1-2} holds.
By Proposition \ref{prop:paraproduct}\ref{prop:paraproduct5} and \ref{prop:paraproduct6}, and (\ref{eq:lemLpestimates1-1}) we have
\begin{align*}
\| f^2\| _{B_{4/3}^{\theta}} &\leq C \| f^2\| _{L^{4/3}}^{2/5} \| f^2\| _{B_{4/3}^{5\theta/3}} ^{3/5} \leq C \| f\| _{L^{8/3}}^{4/5} \| f\| _{L^4}^{3/5}  \| f\| _{B_{2}^{15/16}}^{3/5} \\
&\leq C \| f\| _{L^4}^{7/5} \| f\| _{B_{2}^{15/16}}^{3/5} \leq C \left( \| f\| _{L^4}^{13/5} + \| f\| _{B_{2}^{15/16}}^{13/10} \right) \\
&\leq C \left( \| f\| _{L^4}^4 + \| f\| _{B_{2}^{15/16}}^{2} \right) ^{13/20} + C.
\end{align*}
Hence, applying (\ref{eq:lemLpestimates1-1}) again, we obtain \ref{lem:Lpestimates1-2}.

Finally, we show that \ref{lem:Lpestimates1-3} holds.
By Proposition \ref{prop:paraproduct}\ref{prop:paraproduct5.5} and \ref{prop:paraproduct6}, and (\ref{eq:lemLpestimates1-1}) we have
\begin{align*}
\| f^3\| _{B_1^\theta} &\leq C \| f^3\| _{L^1}^{2/5} \| f^3\| _{B_1^{5\theta /3}}^{3/5} \leq C \| f\| _{L^3}^{6/5} \| f\| _{L^4}^{6/5} \| f\| _{B_{2}^{15/16}} ^{3/5}\\
&\leq C \| f\| _{L^4}^{12/5} \| f\| _{B_{2}^{15/16}} ^{3/5} \leq C\left( \| f\| _{L^4}^{18/5} + \| f\| _{B_{2}^{15/16}}^{9/5} \right) \\
&\leq C\left( \| f\| _{L^4}^4 + \| f\| _{B_{2}^{15/16}}^2 \right) ^{9/10} + C.
\end{align*}
Hence, applying (\ref{eq:lemLpestimates1-1}) again, we obtain \ref{lem:Lpestimates1-3}.
\end{proof}

\begin{lem}\label{lem:timeint}
Let $\alpha , \beta \in [0,1)$.
Then, for any nongegative measurable function $f$ and $s,t \in [0,\infty )$ such that $s<t$, we have
\[
\int _s^t (t-v)^{-\alpha} \left( \int _s^v (v-u)^{-\beta } f(u) du \right) dv = B(\alpha , \beta ) \int _s^t (t-u)^{-(\alpha + \beta ) +1} f(u) du
\]
where $B(\alpha , \beta)$ is the beta function with indices $\alpha$ and $\beta$.
\end{lem}

\begin{proof}
The assertion is obtained by a simple application of Fubini-Tonelli's theorem on changing the order of integration.
\end{proof}

The following Propositions \ref{prop:com2} and \ref{prop:com3} are about the estimate of commutators of paraproducts and the heat semigroup, respectively.

\begin{prop}\label{prop:com2}
Let $\alpha \in (0,1)$, and $\beta, \gamma \in {\mathbb R}$ and $p, p_1,p_2,p_3 \in [1,\infty ]$.
Assume that
\[
\beta + \gamma <0, \quad \alpha + \beta + \gamma >0, \quad \frac{1}{p} = \frac{1}{p_1} + \frac{1}{p_2} + \frac{1}{p_3}.
\]
Then,
\[
\| (f \mbox{\textcircled{\scriptsize$<$}} g) \mbox{\textcircled{\scriptsize$=$}} h - f (g \mbox{\textcircled{\scriptsize$=$}} h)\| _{B_{p}^{\alpha + \beta + \gamma}} \leq C \| f\| _{B_{p_1}^{\alpha}} \| g\| _{B_{p_2}^{\beta}} \| h\| _{B_{p_3}^{\gamma}}
\]
for $f\in B_{p_1}^{\alpha}$, $g\in B_{p_2}^{\beta}$ and $h\in B_{p_3}^{\gamma}$ where $C$ is a constant depending on $\alpha$, $\beta$, $\gamma$, $p_1$, $p_2$, $p_3$.
\end{prop}

\begin{proof}
See Proposition A.9 in \cite{MW3}.
\end{proof}

In the following we shall need the Fourier expansion of functions in $L^2(\Lambda; {\mathbb C})$.
Let
\[
\langle f,g \rangle := \int _{\Lambda} f(x) \overline{g(x)} dx , \quad f,g \in L^2(\Lambda ;{\mathbb C})
\]
where $\bar z$ is the complex conjugate of $z\in {\mathbb C}$.
For $k \in {\mathbb Z}^3$, define $e_k \in L^2(\Lambda ; {\mathbb C})$ by $e_k(x) := e^{\sqrt{-1} k\cdot x}/(2\pi )^{3/2}$.
Then, $\{ e_k; k\in {\mathbb Z}^3\}$ is a complete, orthogonal and normal system of $L^2(\Lambda ;{\mathbb C})$, and $e_k$ is an eigenfunction of $-\triangle + m_0^2$ acting as a self-adjoint operator ($\geq m_0^2$) in $L^2(\Lambda ,{\mathbb C})$ with pure discrete spectrum consisting of the eigenvalues $k^2+m_0^2$ where $k^2 := \sum _{j=1}^3 k_j^2$ (with $m_0>0$ as before, $k\in {\mathbb Z}^3$).
Let $\psi ^{(1)}$ be a nonincreasing $C^\infty $-function on $[0,\infty)$ such that $\psi ^{(1)}(r) =1$ for $r\in [0,1]$ and $\psi ^{(1)}(r)=0$ for $r\in [2,\infty )$, and let $\psi ^{(2)}$ be a nonincreasing function on $[0,\infty)$ such that $\psi ^{(2)}(r) =1$ for $r\in [0,2]$ and $\psi ^{(2)}(r)=0$ for $r\in [4,\infty )$.
We remark that $\psi ^{(2)}$ is not necessary continuous.

For $i=1,2$ and $N\in {\mathbb N}$ define $P_N^{(i)}$ by the mapping from ${\mathcal D}'(\Lambda)$ to $L^2(\Lambda ;{\mathbb C})$ given by
\[
P_N ^{(i)} \phi := \sum _{k\in {\mathbb Z}^3} \psi ^{(i)}(2^{-N}|k_1|) \psi ^{(i)} (2^{-N}|k_2|) \psi ^{(i)} (2^{-N}|k_3|) \langle \phi , e_k \rangle e_k .
\]
For $n\in {\mathbb N}$, denote $\{ j\in {\mathbb Z}; |j|\leq 2^n\}$ by ${\mathbb Z}_n$.
We remark that the terms in the sum are equal to $0$ unless $k\in {\mathbb Z}_{N+i}^3$ for $i=1,2$, and hence, $P_N ^{(i)}\phi$ is a real-valued and smooth function for any $\phi \in {\cal D}'(\Lambda)$.
Moreover, it holds that
\[
P_N ^{(1)}P_N ^{(2)} = P_N ^{(2)} P_N ^{(1)} = P_N ^{(1)} .
\]
This property is very important in arguments in Section \ref{sec:main}.
The theory of Fourier transforms of periodic distributions (see Section 3.2.3 in \cite{ScTr}) implies that for $f\in {\mathcal D}'(\Lambda)$
\[
h(\nabla ) f = \sum _{k\in {\mathbb Z}^3} h(k) \langle f, e_k\rangle e_k
\]
for any continuous function $h$ such that the right-hand side is an $L^2(\Lambda ;{\mathbb C})$ function.
In particular, 
\[
P_N ^{(i)} f (x) = \psi ^{(i)} (2^{-N}|\cdot |) ^{\otimes 3}(\nabla ) f (x)= {\mathcal F}^{-1} \left[ \psi ^{(i)} (2^{-N}|\cdot |) ^{\otimes 3} {\mathcal F} \widetilde{f} \right] (x), \quad x\in \Lambda
\]
with the periodic extension $\widetilde{f}$ of $f$, where
\[
\psi ^{(i)} (2^{-N}|\cdot |) ^{\otimes 3}(\xi ):= \psi ^{(i)} (2^{-N}|\xi _1|) \psi ^{(i)} (2^{-N}|\xi _2|) \psi ^{(i)} (2^{-N}|\xi _3|) , \quad \xi =(\xi_1, \xi_2, \xi_3) \in {\mathbb R}^3.
\]
In the rest of the paper, we fix $\psi ^{(1)}$ and $\psi ^{(2)}$ and do not mention explicitly dependence on them, even if the constants that will appear in the estimates below are depending on them. 

\begin{prop}\label{prop:bddp}
For $p\in (1,\infty)$ and $s\in {\mathbb R}$, there exists a constant $C_p$ depending on $p$ such that
\[
\| P_N ^{(i)} f \| _{B_p^s} \leq C_p \| f \|_{B_p^s}
\]
for $f\in B_p^s$, $N\in {\mathbb N}$ and $i=1,2$.
\end{prop}

\begin{proof}
Let $f\in B_p^s$ and $i=1,2$.
From the definition of Besov spaces and the commutativity of $P_N^{(i)}$ and $\Delta _j$, we have
\begin{equation}\label{eq:propbddp}
\| P_N ^{(i)} f \| _{B_p^s} = \sup _{j \in {\mathbb N}\cup \{ -1,0\}} 2^{js} \| P_N ^{(i)} \Delta _j f \| _{L^p}.
\end{equation}
Note that for $N\in {\mathbb N}$, the total variation of the function $\xi \mapsto \psi ^{(i)}(2^{-N}|\xi |)$ on ${\mathbb R}$ equals $2$.
In particular, the total variation is uniformly bounded for $N\in {\mathbb N}$.
In view of Theorem 3 of Section 3.4.3 in \cite{ScTr}, there exists a constant $C_p$ depending on $p$, such that
\[
\| P_N ^{(i)} \Delta _j f \| _{L^p} \leq C_p \| \Delta _j f\| _{L^p}, \quad j \in {\mathbb N}\cup \{ -1,0\}, \ N\in {\mathbb N},\ f\in {B_p^s}.
\]
By this inequality and (\ref{eq:propbddp}) we obtain the assertion.
\end{proof}

Let $\{ e^{t\triangle }; t\geq 0\}$ be the heat semigroup generated by $\triangle $ on $L^2(\Lambda ;{\mathbb C})$.
Then, we have the following. 

\begin{prop}\label{prop:com3}
For $\alpha \in (-\infty ,1)$, $\beta , \gamma \in {\mathbb R}$ such that $\gamma \geq \alpha + \beta$, $\varepsilon \in (0, \max\{ 1,1-\alpha \}]$, and $p, p_1, p_2 \in [1,\infty]$ such that $1/p = 1/p_1 + 1/p_2$, there exists a constant $C_{\psi ^{(1)}, \gamma ,\varepsilon}$ depending on $\psi ^{(1)}$, $\gamma$ and $\varepsilon$ such that
\[
\| e^{t\triangle }(P_N^{(1)} )^2 (f \mbox{\textcircled{\scriptsize$<$}} g) - f \mbox{\textcircled{\scriptsize$<$}} (e^{t\triangle } (P_N^{(1)} )^2 g) \| _{B_p^{\gamma }} \leq C_{\psi ^{(1)}, \gamma ,\varepsilon} t^{-(\gamma -\alpha -\beta )/2} \| f \|_{B_{p_1}^{\alpha +\varepsilon}} \| g \|_{B_{p_2}^{\beta }} 
\]
for $f,g \in C^\infty (\Lambda )$, $t\in (0,\infty )$ and $N\in {\mathbb N}$, where $\psi ^{(1)}$ is defined before, between Proposition \ref{prop:com2} and Proposition \ref{prop:bddp}.
\end{prop}

\begin{proof}
We prove the proposition by following the proofs of Lemmas 2.97 and 2.99 in \cite{BCD}.
We have, recalling the definition of paraproducts,
\begin{align*}
&e^{t\triangle }(P_N^{(1)} )^2 (f \mbox{\textcircled{\scriptsize$<$}} g)- f \mbox{\textcircled{\scriptsize$<$}} (e^{t\triangle } (P_N^{(1)} )^2 g)\\
&= \sum _{j=0}^\infty \left( e^{t\triangle } (P_N^{(1)} )^2 \left[ (S_j f) \Delta _{j+1} g \right] - (S_j f) \Delta _{j+1} e^{t\triangle } (P_N^{(1)} )^2 g \right) .
\end{align*}
We shall now use the notation $\widetilde{f}$ as defined as above.
Since the supports of ${\mathcal F}(\widetilde{S_j f})$ and ${\mathcal F}\widetilde{\Delta _{j+1}g}$ are included by $\{ \xi \in {\mathbb R}^3; |\xi |\leq 2^{j-1}(8/3)\}$ and $\{ \xi \in {\mathbb R}^3; 2^{j+1}(3/4)\leq |\xi|\leq 2^{j+1}(8/3)\}$ for $j\in {\mathbb N}\cup \{ 0\}$ respectively, the support of ${\mathcal F}\left[ (\widetilde{S_j f}) \widetilde{\Delta _{j+1} g} \right]$ is included by $\{ \xi \in {\mathbb R}^3; 1/6 \leq 2^{-j}|\xi|\leq 20/3\}$ for $j\in {\mathbb N}\cup \{ 0\}$.
Hence, in view of Lemma 2.69 in \cite{BCD}, it is sufficient to show that
\begin{equation}\label{eq:propcom3-1}\begin{array}{l}
\displaystyle \sup _{j\in {\mathbb N}\cup \{ 0\}} 2^{j \gamma}\left\| e^{t\triangle } (P_N^{(1)} )^2 \left[ (S_j f) \Delta _{j+1} g \right] - (S_j f) \Delta _{j+1} e^{t\triangle } (P_N^{(1)} )^2 g \right\| _{L^p} \\
\displaystyle \leq C_{\psi ^{(1)}, \gamma ,\varepsilon} t^{-(\gamma -\alpha -\beta )/2} \| f \|_{B_{p_1}^{\alpha +\varepsilon}} \| g \|_{B_{p_2}^{\beta }} .
\end{array}\end{equation}
Let $\rho \in C^\infty ([0,\infty );[0,1])$ such that $\rho (r) =1$ for $r \in  [1/6, 20/3]$ and $\rho (r)=0$ for $r \not \in [1/12, 40/3]$.
Define 
\[
h_{t,N,j}(\xi ):= e^{-t|\xi |^2} \left[ \psi ^{(1)} (2^{-N}|\cdot |) ^{\otimes 3}(\xi ) \right] ^2 \rho (2^{-j} |\xi| ) , \quad \xi \in {\mathbb R}^3.
\]
Noting that 
\[
e^{t\triangle } (P_N^{(1)} )^2 \phi = {\mathcal F}^{-1} \left( e^{-t|\cdot |^2} \left[ \psi ^{(1)} (2^{-N}|\cdot |) ^{\otimes 3}\right]^2 {\mathcal F} \widetilde{\phi} \right) , \quad \phi \in C(\Lambda ),
\]
and that $e^{t\triangle } P_N$ and $\Delta _{j+1}$ commute, we have for any $j\in {\mathbb N}\cup \{ 0\}$, $N\in {\mathbb N}$ and $x\in \Lambda$;
\begin{align*}
&e^{t\triangle } (P_N^{(1)} )^2 \left[ (S_j f) \Delta _{j+1} g \right] (x) - (S_j f) (x) \Delta _{j+1} e^{t\triangle } (P_N^{(1)} )^2 g (x)\\
&= {\mathcal F}^{-1} \left( h_{t,N,j} {\mathcal F} \left[ (\widetilde{S_j f}) \widetilde{\Delta _{j+1} g} \right]\right) (x) - (S_j f) (x) {\mathcal F}^{-1} \left( h_{t,N,j} {\mathcal F} \widetilde{\Delta _{j+1} g} \right) (x) \\
&= \int _{{\mathbb R}^3} {\mathcal F}^{-1}(h_{t,N,j})(x-y) \left( \widetilde{S_j f}(y) - \widetilde{S_j f}(x) \right) \widetilde{\Delta _{j+1} g}(y) dy .
\end{align*}
Since $\widetilde{S_j f}(\cdot + 2\pi k) = \widetilde{S_j f}$ and $\widetilde{\Delta _{j+1} g}(\cdot + 2\pi k) = \widetilde{\Delta _{j+1} g}$ for $k\in {\mathbb Z}^3$, we get then
\begin{align*}
&\left| e^{t\triangle } (P_N^{(1)} )^2 \left[ (S_j f) \Delta _{j+1} g \right] (x) - (S_j f) (x) \Delta _{j+1} e^{t\triangle } (P_N^{(1)} )^2 g (x) \right|\\
&= \left| \sum _{k\in {\mathbb Z}^3} \int _{[0,2\pi ]^3} {\mathcal F}^{-1}(h_{t,N,j})(x-y-2\pi k) \left( \widetilde{S_j f}(y) - \widetilde{S_j f}(x) \right) \widetilde{\Delta _{j+1} g}(y) dy \right| \\
&\leq \sum _{k\in {\mathbb Z}^3} \int _{[0,2\pi ]^3} \int _0^1 |x-y| \, \left| {\mathcal F}^{-1}(h_{t,N,j})(x-y-2\pi k)\right| \left| \widetilde{\nabla S_j f}((1-\tau )x + \tau y) \right| \, \left| \widetilde{\Delta _{j+1} g}(y) \right| d\tau dy\\
&=  \sum _{k\in {\mathbb Z}^3} \int _{[0,4\pi ]^3} \int _0^1 |z| \, \left|{\mathcal F}^{-1}(h_{t,N,j})(z-2\pi k)\right| \, \left| \widetilde{\nabla S_j f}(x - \tau z)\right| \, \left| \widetilde{\Delta _{j+1} g}(x-z)\right| {\mathbb I}_{\Lambda} (x-z)d\tau dz.
\end{align*}
Hence, H\"older's inequality applied to the right-hand side yields
\begin{equation}\label{eq:propcom3-2}
\begin{array}{l}
\displaystyle \left\| e^{t\triangle } (P_N^{(1)} )^2 \left[ (S_j f) \Delta _{j+1} g \right] - (S_j f) (x) \Delta _{j+1} e^{t\triangle } (P_N^{(1)} )^2 g \right\| _{L^p} \\
\displaystyle \leq \sum _{k\in {\mathbb Z}^3} \int _{[0,4\pi ]^3} \int _0^1 \left| z{\mathcal F}^{-1}(h_{t,N,j})(z-2\pi k)\right| \left\| \widetilde{\nabla S_j f}(\cdot -\tau z) \right\| _{L^{p_1}} \left\| \widetilde{\Delta _{j+1} g}(\cdot -z) \right\| _{L^{p_2}} d\tau dz.
\end{array}
\end{equation}
The periodicity of $\widetilde{\nabla S_j f}$ and $\widetilde{\Delta _{j+1} g}$ implies
\[
\left\| \widetilde{\nabla S_j f}(\cdot -\tau z) \right\| _{L^{p_1}} \leq \left\| \nabla S_j f \right\| _{L^{p_1}},\quad \left\| \widetilde{\Delta _{j+1} g}(\cdot -z) \right\| _{L^{p_2}} \leq \left\| \Delta _{j+1} g \right\| _{L^{p_2}}
\]
for $\tau \in [0,1]$ and $z\in [0,4\pi ]^3$.
These equalities and (\ref{eq:propcom3-2}) imply
\begin{equation}\label{eq:propcom3-3}
\begin{array}{l}
\displaystyle \left\| e^{t\triangle } (P_N^{(1)} )^2 \left[ (S_j f) \Delta _{j+1} g \right] - (S_j f) (x) \Delta _{j+1} e^{t\triangle } (P_N^{(1)} )^2 g \right\| _{L^p} \\
\displaystyle \leq \sum _{k\in {\mathbb Z}^3} \int _{[0,4\pi ]^3} \left| z{\mathcal F}^{-1}(h_{t,N,j})(z-2\pi k)\right| dz \left\| \nabla S_j f \right\| _{L^{p_1}} \left\| \Delta _{j+1} g \right\| _{L^{p_2}}.
\end{array}
\end{equation}
On the other hand, by the integration by parts formula we have
\begin{align*}
&\sum _{k\in {\mathbb Z}^3} \int _{[0,4\pi ]^3} \left| z{\mathcal F}^{-1}(h_{t,N,j})(z-2\pi k)\right|  dz\\
&\leq 8 \int _{{\mathbb R}^3} \min\{ |z|,2\pi \} \left| {\mathcal F}^{-1}(h_{t,N,j})(z)\right|  dz\\
&= \frac{2\sqrt 2}{\pi ^{3/2}} \int _{{\mathbb R}^3} |z|^{-4} \min\{ |z|,2\pi \} \left| \int _{{\mathbb R}^3}e^{-t|\xi |^2} \left[ \psi ^{(1)}(2^{-N}|\cdot |) ^{\otimes 3}(\xi )\right]^2 \rho (2^{-j} |\xi| ) |z|^4 e^{\sqrt{-1} \xi \cdot z} d\xi \right| dz\\
&= \frac{2\sqrt 2}{\pi ^{3/2}} \int _{{\mathbb R}^3} |z|^{-4} \min\{ |z|,2\pi \} \\
&\quad \hspace{2cm} \times  \left| \int _{{\mathbb R}^3}e^{-t|\xi |^2} \left[ \psi ^{(1)}(2^{-N}|\cdot |) ^{\otimes 3}(\xi )\right]^2 \rho (2^{-j} |\xi| ) \triangle _{\xi}^2 [\cos (\xi \cdot z) -1] d\xi \right| dz\\
&= \frac{2\sqrt 2}{\pi ^{3/2}} \int _{{\mathbb R}^3} |z|^{-4} \min\{ |z|,2\pi \} \\
&\quad \hspace{1cm} \times \left| \int _{{\mathbb R}^3} \left[ \triangle _{\xi}^2 \left( e^{-t|\xi |^2} \left[ \psi ^{(1)} (2^{-N}|\cdot |) ^{\otimes 3}(\xi )\right]^2 \rho (2^{-j} |\xi| )\right) \right] \left( \cos ( \xi \cdot z) -1\right) d\xi \right| dz\\
&\leq \frac{2\sqrt 2}{\pi ^{3/2}} \int _{{\mathbb R}^3} \left| \triangle _{\xi}^2 \left( e^{-t|\xi |^2} \left[ \psi ^{(1)} (2^{-N}|\cdot |) ^{\otimes 3}(\xi )\right]^2 \rho (2^{-j} |\xi| )\right) \right| \\
&\quad \hspace{5cm} \times \int _{{\mathbb R}^3} |z|^{-4} \min\{ |z|,2\pi \} \left| \cos (\xi \cdot z) -1\right| dz d\xi .
\end{align*}
Noting that for $\xi \in {\mathbb R}^3$,
\begin{align*}
&\int _{{\mathbb R}^3} |z|^{-4} \min\{ |z|,2\pi \} \left| \cos (\xi \cdot z) -1\right| dz \\
&\leq 2^{1-\varepsilon /3} \int _{|z|\leq 2\pi } |z|^{-3}  \left| \cos (\xi \cdot z) -1\right| ^{\varepsilon /3} dz + 4\pi \int _{|z|>2\pi } |z|^{-4} dz\\
&\leq C_{\varepsilon} (1+|\xi |^{2\varepsilon /3})
\end{align*}
where $C_{\varepsilon}$ is a constant depending on $\varepsilon$, we have
\begin{equation}\label{eq:propcom3-4}\begin{array}{l}
\displaystyle \sum _{k\in {\mathbb Z}^3} \int _{[0,4\pi ]^3} |z{\mathcal F}^{-1}(h_{t,N,j})(z-2\pi k)| dz \\
\displaystyle \leq  C_{\varepsilon} \int _{{\mathbb R}^3} (1+|\xi |^{2\varepsilon /3}) \left| \triangle _{\xi}^2 \left( e^{-t|\xi |^2} \left[ \psi ^{(1)} (2^{-N}|\cdot |) ^{\otimes 3}(\xi )\right]^2 \rho (2^{-j} |\xi| )\right) \right| d\xi .
\end{array}\end{equation}
When $2^{N} < 2^{j-3}/3$, the supports of $\psi ^{(1)}$ and $\rho$ imply that $\left[ \psi ^{(1)} (2^{-N}|\cdot |) ^{\otimes 3} \right]^2 \rho (2^{-j} |\cdot | ) =0$.
When $2^{N} \geq 2^{j-3}/3$, an explicit calculation implies that for $\xi \in {\mathbb R}^3$ and $t\in [0,\infty )$
\begin{align*}
& \left| \triangle _{\xi}^2 \left( e^{-t|\xi |^2} \left[ \psi ^{(1)} (2^{-N}|\cdot |) ^{\otimes 3}(\xi )\right]^2 \rho (2^{-j} |\xi| )\right) \right| \\
& \leq C \left( t^4|\xi |^4 + |\xi |^{-4} + 2^{-4j}\right) e^{-t|\xi |^2} \sum _{k=0}^4 (\partial ^k \rho ) (2^{-j} |\xi |)
\end{align*}
where $C$ is a constant depending on the bounds of derivatives of $\psi ^{(1)}$ up to order $4$.
Hence, in view of the support of $\rho$ we have
\begin{align*}
&\int _{{\mathbb R}^3} (1 + |\xi |^{2\varepsilon /3}) \left| \triangle _{\xi}^2 \left( e^{-t|\xi |^2} \left[ \psi ^{(1)} (2^{-N}|\cdot |) ^{\otimes 3}(\xi )\right]^2 \rho (2^{-j} |\xi| )\right) \right| d\xi \\
&\leq C_1 (t^2+ 2^{-4j}) \int _{\{ \xi \in {\mathbb R}^3; 2^{-j} |\xi |\in [1/12, 40/3]\} } (1+|\xi |^{2\varepsilon /3}) e^{-t|\xi |^2 /2} d\xi \\
&\leq C_2 (t^2+ 2^{-4j}) 2^{j(3 + 2\varepsilon /3)} e^{-t2^{2j}/24}
\end{align*}
where $C_1$ and $C_2$ are absolute constants.
This inequality and (\ref{eq:propcom3-4}) imply
\begin{equation}\label{eq:propcom3-5}
\sum _{k\in {\mathbb Z}^3} \int _{[0,4\pi ]^3} |z{\mathcal F}^{-1}(h_{t,N,j})(z-2\pi k)| dz  \leq C_{\psi ^{(1)}, \varepsilon} (t^2+ 2^{-4j}) 2^{j(3+2\varepsilon /3)} e^{-t2^{2j}/24} .
\end{equation}
In view of Proposition 2.78 in \cite{BCD} we have, on the other hand,
\begin{align*}
\left\| \nabla S_j f \right\| _{L^{p_1}} &\leq C_{\varepsilon } 2^{j(1-\alpha - 2\varepsilon /3 )} \| \nabla f \| _{B_{p_1}^{\alpha + 2\varepsilon /3 -1}} \leq C_{\varepsilon}' 2^{j(1-\alpha - 2\varepsilon /3 )} \| \nabla f \| _{W^{\alpha + 2\varepsilon /3 -1, p_1}}\\
&\leq C_{\varepsilon}' 2^{j(1-\alpha - 2\varepsilon /3 )} \| f \| _{W^{\alpha + 2\varepsilon /3 , p_1}} \leq C_{\varepsilon}'' 2^{j(1-\alpha - 2\varepsilon /3 )} \| f \| _{B_{p_1}^{\alpha + \varepsilon}}
\end{align*}
where $C_{\varepsilon}$, $C_{\varepsilon}'$ and $C_{\varepsilon}''$ are constants depending on $\varepsilon$.
From this inequality, (\ref{eq:propcom3-3}) and (\ref{eq:propcom3-5}) we obtain
\begin{align*}
&2^{j \gamma } \left\| e^{t\triangle } (P_N^{(1)} )^2 \left[ (S_j f) \Delta _{j+1} g \right] - (S_j f) (x) \Delta _{j+1} e^{t\triangle } (P_N^{(1)} )^2 g \right\| _{L^p} \\
&\leq C_{\psi ^{(1)}, \varepsilon} (t^2+ 2^{-4j}) 2^{j(4 +\gamma -\alpha -\beta )} e^{-t2^{2j}/24} \| f \| _{B_{p_1}^{\alpha +\varepsilon}} \cdot 2^{j\beta }\left\| \Delta _{j+1} g \right\| _{L^{p_2}} \\
&\leq C_{\psi ^{(1)}, \varepsilon} t^{(\alpha + \beta -\gamma)/2} (t2^{2j})^{(\gamma -\alpha -\beta )/2} \left( 1+(t2^{2j})^2 \right) e^{-t2^{2j}/24} \| f \| _{B_{p_1}^{\alpha +\varepsilon}} \cdot 2^{j\beta }\left\| \Delta _{j+1} g \right\| _{L^{p_2}} \\
&\leq C_\gamma C_{\psi ^{(1)}, \varepsilon} t^{(\alpha + \beta -\gamma)/2} \| f \| _{B_{p_1}^{\alpha +\varepsilon}} \cdot 2^{j\beta }\left\| \Delta _{j+1} g \right\| _{L^{p_2}}
\end{align*}
where $C_{\psi ^{(1)}, \varepsilon}$ is a constant depending on $\psi ^{(1)}$ and $\varepsilon$, and $C_\gamma$ is a constant depending on $\gamma$.
Thus, we have proven (\ref{eq:propcom3-1}) and this, as we mentioned in relation with (\ref{eq:propcom3-1}) suffices for the proof of Proposition \ref{prop:com3}.
\end{proof}

\section{Infinite-dimensional Ornstein-Uhlenbeck process}\label{sec:OU}

In this section we introduce the relevant infinite-dimensional Ornstein-Uhlenbeck process and its polynomials.
The polynomials of the Ornstein-Uhlenbeck process appear in the renormalization and the reconstruction of the stochastic partial differential equation associated to the $\Phi ^4_3$-model.
 
Let $\tilde \mu _0$ be the centered Gaussian measure on ${\cal D}'(\Lambda )$ with the covariance operator $[2(-\triangle + m_0^2)]^{-1}$ where $m_0>0$ as before, and let $Z_t$ be the solution to the stochastic partial differential equation on the $3$-dimensional torus $\Lambda$:
\begin{equation}\label{eq:SDEZ}\left\{ \begin{array}{rll}
d Z_t(x) &= d W_t(x) - (-\triangle +m_0^2)Z_t(x) dt, & (t,x)\in (-\infty ,\infty) \times \Lambda \\
Z_0(x)& =\zeta (x), &x\in \Lambda
\end{array}\right.\end{equation}
where $dW_t(x)$ is a Gaussian white noise with parameter $(t,x)\in  (-\infty ,\infty) \times \Lambda$ and $\zeta$ is a random variable which has $\tilde \mu _0$ as its law and is independent of $W_t$ (see Remark \ref{rem:free} below for this notation and the relation with the $\mu _0$ of Section \ref{sec:intro}).
We remark that (\ref{eq:SDEZ}) is an equation not only for positive $t$, but also for negative $t$.
$W_t$ can be looked upon as a $C((-\infty ,\infty) ; {\mathcal S}'(\Lambda ))$-Brownian motion.
Then, $\langle Z_t ,e_k \rangle$ satisfies the one-dimensional stochastic differential equation of the Ornstein-Uhlenbeck type 
\begin{equation}\label{eq:SDEZ2}
d \langle Z_t ,e_k \rangle = d \langle W_t , e_k\rangle - (k^2 +m_0^2) \langle Z_t ,e_k \rangle dt ,
\end{equation}
for each $k \in {\mathbb Z}^3$, and hence we obtain the solution as
\begin{equation}\label{eq:solZ}
\langle Z_t ,e_k \rangle = e^{-t(k^2 +m_0^2)} \langle Z_0 ,e_k \rangle + \int _0^t e^{(s-t)(k^2 +m_0^2)} d\langle W_s , e_k\rangle,
\end{equation}
for each $k \in {\mathbb Z}^3$.
We remark that $(\langle W_t ,e_k \rangle , t\in (-\infty ,\infty) )$ is a $1$-dimensional standard Brownian motion, $(\langle Z_t ,e_k\rangle ; t\in (-\infty ,\infty) , k \in{\mathbb Z}^d )$ is a Gaussian system and the law of $Z_t$ coincides with $\tilde \mu _0$ for all $t\in (-\infty ,\infty) $.

\begin{rem}\label{rem:free}
If we replace $dW_t(x)$ by $\sqrt 2 dW_t(x)$ in (\ref{eq:SDEZ}), then the solution $Z_t$ will be the Ornstein-Uhlenbeck process which has the Nelson's Euclidean free field measure $\mu _0$ (of Section \ref{sec:intro}, with covariance operator $(-\triangle +m_0^2)^{-1}$) with mass $m_0$ as the stationary measure.
Some authors define the Ornstein-Uhlenbeck process to be the solution to the equation (\ref{eq:SDEZ}) with this replacement.
However, in the present paper, in order to adjust (also for reader's convenience) our setting to those of other recent papers (e.g. \cite{CaCh}, \cite{Ha}, \cite{MW3} and \cite{MWX}), we define $Z_t$ as above.
We remark that even if we replace $dW_t(x)$ by $\sqrt 2 dW_t(x)$ in (\ref{eq:SDEZ}), the arguments below run almost in the same way, with some powers of constants entering in estimates, that do not change the conclusions. 
\end{rem}

For square-integrable complex-valued random variables $\xi_1$, $\xi_2$, we define ${\rm Cov}(\xi _1, \xi_2)$ by
\[
{\rm Cov}(\xi_1, \xi _2) := E\left[ (\xi _1-E[\xi _1]) (\xi _2-E[\xi _2]) \right]
\]
($E$ denoting as usual the expectation).
Then, it is easy to see that
\begin{align}
\label{eq:expZ} E[\langle Z_t , e_k \rangle ] &= 0,\\
\label{eq:covZ} {\rm Cov}(\langle Z_{t}, e_k \rangle , \langle Z_{s} , e_l \rangle ) &= \frac{e^{-|t-s|(k^2+m_0^2)}}{2(k^2+m_0^2)} {\mathbb I}_{k+l=0} ,
\end{align}
for $s,t \in (-\infty ,\infty)$ and $k,l\in {\mathbb Z}^3$.
Let 
\begin{align*}
C_1^{(N)} &:= \frac{1}{2(2\pi )^3} \sum _{k\in {\mathbb Z}^3} \frac{\left[ \psi ^{(1)} (2^{-N}|\cdot |)^{\otimes 3}(k) \right]^2}{k^2+m_0^2} \\
C_2^{(N)} &:= \frac{1}{2(2\pi )^6} \sum _{l_1,l_2 \in {\mathbb Z}^3} \frac{\left[ \psi ^{(1)}(2^{-N}|\cdot |)^{\otimes 3}(l_1) \right]^2 \left[ \psi ^{(1)} (2^{-N}|\cdot |)^{\otimes 3}(l_2) \right]^2 \left[ \psi ^{(1)}(2^{-N}|\cdot |)^{\otimes 3}(l_1+l_2) \right]^2}{(l_1^2+m_0^2)(l_2^2+m_0^2)(l_1^2 + l_2^2 + (l_1+l_2)^2 +3m_0^2)}
\end{align*}
and define
\begin{align*}
{\mathcal Z}^{(1,N)}_t &:= P_N ^{(1)} Z_t ,\\
{\mathcal Z}^{(2,N)}_t &:= (P_N ^{(1)} Z_t) ^2 - C_1^{(N)} , \\
{\mathcal Z}^{(3,N)}_t &:= (P_N ^{(1)} Z_t) ^3 - 3 C_1^{(N)} P_N ^{(1)} Z_t , \\
{\mathcal Z}^{(0,2,N)}_t &:= \int _{-\infty}^t e^{(t-s)(\triangle -m_0^2)} P_N ^{(1)} {\mathcal Z}^{(2,N)}_s ds , \\
{\mathcal Z}^{(0,3,N)}_t &:= \int _{-\infty}^t e^{(t-s)(\triangle -m_0^2)} P_N ^{(1)} {\mathcal Z}^{(3,N)}_s ds , \\
{\mathcal Z}^{(2,2,N)}_t &:= {\mathcal Z}^{(2,N)}_t \mbox{\textcircled{\scriptsize$=$}} P_N^{(1)} {\mathcal Z}^{(0,2,N)}_t -C_2^{(N)}, \\
{\mathcal Z}^{(2,3,N)}_t &:= {\mathcal Z}^{(2,N)}_t \mbox{\textcircled{\scriptsize$=$}} P_N^{(1)} {\mathcal Z}^{(0,3,N)}_t - 3C_2^{(N)}{\mathcal Z}^{(1,N)}_t ,
\end{align*}
for $t\in (-\infty , \infty)$ and $N\in {\mathbb N}$.

\begin{rem}\label{rem:rconst}
For $t \in (-\infty ,\infty )$ and $N\in {\mathbb N}$ it holds that
\begin{align}
\label{eq:remrconst1} &E\left[ ( P_N ^{(1)} Z_t ) ^2\right] -C_1^{(N)} =0 ,\\
\label{eq:remrconst2} &E\left[ {\mathcal Z}^{(2,N)}_t \mbox{\textcircled{\scriptsize$=$}} P_N^{(1)} {\mathcal Z}^{(0,2,N)}_t \right] -C_2^{(N)} =0 .
\end{align}
The proofs of (\ref{eq:remrconst1}) and (\ref{eq:remrconst2}) are mentioned at the beginning of the proof of Proposition \ref{prop:Z} below.
The definition of $C_2^{(N)}$ is a little different from other known results (e.g. \cite{CaCh}, \cite{MW3}).
However, the asymptotics are same and it can be replaced by the one in other known results. 
\end{rem}

The following proposition holds.

\begin{prop}\label{prop:Z}
Let $\varepsilon \in (0,1]$ and $p\in [1,\infty )$.
Then for all $T\in (0,\infty )$ the following properties hold:
\begin{enumerate}
\item \label{prop:Z1} $\{ {\mathcal Z}^{(1,N)}; N\in {\mathbb N}\}$ and $\{ P_N^{(2)} Z; N\in {\mathbb N}\}$ converge almost surely in $C([0,\infty ); B_{\infty}^{-1/2-\varepsilon })$ and satisfies
\[
\sup _{N\in {\mathbb N}} E\left[ \sup_{t\in [0,T]}\| {\mathcal Z}_t^{(1,N)}\| _{B_{\infty}^{-1/2-\varepsilon }}^p\right] < \infty ,\quad \sup _{N\in {\mathbb N}} E\left[ \sup_{t\in [0,T]}\| P_N^{(2)} Z_t \| _{B_{\infty}^{-1/2-\varepsilon }}^p \right] <\infty ;
\]

\item \label{prop:Z2} $\{ {\mathcal Z}^{(2,N)}; N\in {\mathbb N}\}$ converges almost surely in $C([0,\infty ); B_{\infty}^{-1-\varepsilon })$ and satisfies
\[
\sup _{N\in {\mathbb N}} E\left[ \sup _{t\in [0,T]} \left\| {\mathcal Z}^{(2,N)}_t \right\| _{B_{\infty}^{-1-\varepsilon}}^p \right] < \infty ;
\]

\item \label{prop:Z02} $\{ {\mathcal Z}^{(0,2,N)}; N\in {\mathbb N}\}$ converges almost surely in $C([0,\infty ); B_{\infty}^{1-\varepsilon })$ and satisfies
\[
\sup _{N\in {\mathbb N}} E\left[ \sup _{t\in [0,T]} \left\| {\mathcal Z}^{(0,2,N)}_t \right\| _{B_{\infty}^{1-\varepsilon}}^p \right] < \infty ;
\]

\item \label{prop:Z03} $\{ {\mathcal Z}^{(0,3,N)}; N\in {\mathbb N}\}$ converges almost surely in $C([0,\infty ); B_{\infty}^{1/2-\varepsilon })$ and satisfies
\[
\sup _{N\in {\mathbb N}} E\left[  \sup _{t\in [0,T]} \left\| {\mathcal Z}^{(0,3,N)}_t \right\| _{B_{\infty }^{1/2-\varepsilon}}^p\right] < \infty .
\]
Moreover, for $\gamma \in (0,1/4)$, ${\mathcal Z}^{(0,3,N)}$ is $\gamma$-H\"older continuous in time almost surely for $N\in {\mathbb N}$ and
\[
\sup _{N\in {\mathbb N}} E\left[ \left( \sup _{s,t \in [0,T]} \frac{\left\| {\mathcal Z}^{(0,3,N)}_t - {\mathcal Z}^{(0,3,N)}_s \right\| _{L^2}^2}{(t-s)^{\gamma}} \right) ^p \right] < \infty ;
\]

\item \label{prop:Z22} $\{ {\mathcal Z}^{(2,2,N)}; N\in {\mathbb N}\}$ converges almost surely in $C([0,\infty ); B_{\infty}^{-\varepsilon })$ and satisfies
\[
\sup _{N\in {\mathbb N}} E\left[  \sup _{t\in [0,T]}  \left\| {\mathcal Z}^{(2,2,N)}_t \right\| _{B_{\infty }^{-\varepsilon}}^p\right] < \infty ;
\]

\item \label{prop:Z23} $\{ {\mathcal Z}^{(2,3,N)}; N\in {\mathbb N}\}$ converges almost surely in $C([0,\infty ); B_{\infty}^{-1/2-\varepsilon })$ and satisfies
\[
\sup _{N\in {\mathbb N}} E\left[ \sup _{t\in [0,T]} \left\| {\mathcal Z}^{(2,3,N)}_t \right\| _{B_{\infty }^{-1/2-\varepsilon}}^p\right] < \infty ;
\]

\item \label{prop:Z13} $\{ {\mathcal Z}^{(1,N)} {\mathcal Z}^{(0,3,N)}; N\in {\mathbb N}\}$ satisfies
\[
\sup _{N\in {\mathbb N}} E\left[ \sup _{t\in [0,T]} \left\| {\mathcal Z}_t^{(1,N)} P_N^{(1)} {\mathcal Z}^{(0,3,N)}_t \right\| _{B_{\infty }^{-1/2-\varepsilon}}^p\right] < \infty ;
\]

\item \label{prop:Z133} $\{ {\mathcal Z}^{(1,N)} \left( {\mathcal Z}^{(0,3,N)} \right) ^2; N\in {\mathbb N}\}$ satisfies
\[
\sup _{N\in {\mathbb N}} E\left[ \sup _{t\in [0,T]} \left\| {\mathcal Z}_t^{(1,N)} \left( P_N^{(1)} {\mathcal Z}^{(0,3,N)}_t \right) ^2 \right\| _{B_{\infty }^{-1/2-\varepsilon}}^p \right] < \infty .
\]
\end{enumerate}
\end{prop}

For the proof of Proposition \ref{prop:Z}, we shall need a lot of explicit calculations.
Since, on the other hand, the results have been essentially derived before we shall only present here a sketch of the proof.
The proof uses methods of \cite{GIP}.
For more details on the explicit calculation used for establishing the estimates in Proposition \ref{prop:Z} see \cite{CaCh}, \cite{MW3} and \cite{MWX}, noticing also that corresponding calculations in the setting of regularity structures can be found in Section 7 of \cite{Ha} and \cite{Ha2}.

\vspace{3mm}

\noindent {\it Sketch of Proof of Proposition \ref{prop:Z}.}
The proofs of \ref{prop:Z1}, \ref{prop:Z2}, \ref{prop:Z02} and \ref{prop:Z03} are done by explicit calculation by the Fourier expansion.
We consider \ref{prop:Z2}.
First we prove (\ref{eq:remrconst1}).
Note that
\[
E\left[ \left( P_N ^{(1)} Z_t \right) ^2\right] = \frac{1}{(2\pi )^{3/2}}\sum _{k, l\in {\mathbb Z}^3} \psi ^{(1)} (2^{-N}|\cdot |)^{\otimes 3}(k) \psi ^{(1)} (2^{-N}|\cdot |)^{\otimes 3}(l) E\left[ \langle Z_t , e_k \rangle \langle Z_t , e_l \rangle \right] e_{k+l}.
\]
We calculate this sum by using (\ref{eq:expZ}), (\ref{eq:covZ}) and Theorem \ref{thm:AG}, and then easily obtain (\ref{eq:remrconst1}).
Let $\varepsilon , \varepsilon ' \in (0,1]$ such that $\varepsilon ' < \varepsilon$, and define $\tilde \varepsilon = (\varepsilon + \varepsilon')/2$.
In view of (\ref{eq:remrconst1}), to prove \ref{prop:Z2}, we first calculate
\begin{equation}\label{eq:propZ2-1}
\begin{array}{l}
\displaystyle E\left[ \left\| {\mathcal Z}^{(2,N)}_t- {\mathcal Z}^{(2,N)}_s\right\| _{W^{-1-\tilde \varepsilon, 2}}^2\right] \\
\displaystyle = E\left[ \left| \langle (P_N ^{(1)} Z_t)^2, 1 \rangle - \langle (P_N ^{(1)} Z_s)^2, 1 \rangle \right| ^2 \right] \\
\displaystyle \quad + \sum _{k\in {\mathbb Z}^3 \setminus \{ 0\}}\frac{1}{(1+k^2)^{1+\tilde \varepsilon }}E\left[ \left| \langle (P_N ^{(1)} Z_t)^2, e_k \rangle - \langle (P_N ^{(1)} Z_s)^2, e_k \rangle\right| ^2 \right]
\end{array}
\end{equation}
for $s,t\in [0,T]$.
Using the Fourier expansion of $P_N ^{(1)} Z_t$ we can express the right-hand side of (\ref{eq:propZ2-1}) by the expectation of a fourth order polynomial of complex-valued Gaussian random variables.
Hence, by (\ref{eq:expZ}), (\ref{eq:covZ}) and Theorem \ref{thm:AG} we are able to calculate it explicitly, and as a result we have the bound
\[
\sup _{N\in {\mathbb N}} E\left[ \left\| {\mathcal Z}^{(2,N)}_t- {\mathcal Z}^{(2,N)}_s\right\| _{W^{-1-\tilde \varepsilon, 2}}^2\right] \leq C|t-s|^{\varepsilon'}, \quad s,t \in [0,T],
\]
where $C$ is a constant depending on $\varepsilon$ and $\varepsilon'$.
Applying the hypercontractivity of polynomials of Gaussian random variables (see Proposition 2.14 in \cite{Shi} or Theorem 2.7.2 in \cite{NoPe}), for $p\in (1,\infty)$ we have 
\[
\sup _{N\in {\mathbb N}} E\left[ \left\| {\mathcal Z}^{(2,N)}_t- {\mathcal Z}^{(2,N)}_s\right\| _{W^{-1-\tilde \varepsilon, p}}^p\right] \leq C|t-s|^{\varepsilon' p/2}, \quad s,t \in [0,T]
\]
where $C$ is a constant depending on $p$, $\varepsilon$ and $\varepsilon'$.
The Besov embedding theorem (see Proposition 2.71 in \cite{BCD}) implies that for sufficiently large $p$, $W^{-1-\tilde \varepsilon, p}$ is embedded in $B_{\infty}^{-1-\varepsilon }$.
Hence, for sufficiently large $p$
\begin{equation}\label{eq:propZ2-2}
\sup _{N\in {\mathbb N}} E\left[ \left\| {\mathcal Z}^{(2,N)}_t- {\mathcal Z}^{(2,N)}_s\right\| _{B_{\infty}^{-1-\varepsilon }}^p\right] \leq C|t-s|^{\varepsilon' p/2}, \quad s,t \in [0,T]
\end{equation}
where $C$ is a constant depending on $p$, $\varepsilon$ and $\varepsilon'$.
On the other hand, by a similar calculation as above, we have
\[
\sum_{N\in {\mathbb N}} E\left[ \left\| {\mathcal Z}^{(2,N+1)}_t- {\mathcal Z}^{(2,N)}_t\right\| _{B_{\infty}^{-1-\varepsilon }}^p\right] ^{1/p}< \infty
\] 
for $t\in [0,T]$ and $p\in [1,\infty )$.
This implies that ${\mathcal Z}^{(2,N)}_t$ converges to a random variable ${\mathcal Z}^{(2,\infty )}_t$ almost surely for $t\in [0,T]$.
This convergence and (\ref{eq:propZ2-2}) yield the tightness of the laws of $\{ {\mathcal Z}^{(2,N)}; N\in {\mathbb N}\}$ as probability measures on $C([0,\infty ); B_{\infty}^{-1-\varepsilon })$ (see Theorem 4.3 in Chapter I of \cite{IW}).
In view of (\ref{eq:propZ2-2}) and the Kolmogorov criterion, ${\mathcal Z}^{(2,\infty )}$ has a modification $\tilde{\mathcal Z}^{(2,\infty )}$ which is continuous in time almost surely.
Therefore, by applying Proposition \ref{prop:AC} to $\{ {\mathcal Z}^{(2,N)}; N\in {\mathbb N}\}$ and $\tilde{\mathcal Z}^{(2,\infty )}$, we obtain \ref{prop:Z2}.

Similarly we prove \ref{prop:Z1}, \ref{prop:Z02} and \ref{prop:Z03}.
The proof of \ref{prop:Z1} is simpler. On the other hand, the proof of \ref{prop:Z03} is more complicated, because the order is higher and we also need to calculate the action of the semigroup and the integral in time, in order to get the result.

To prove \ref{prop:Z22} and \ref{prop:Z23}, we need to calculate paraproducts.
Since $e_k$ is an eigenfunction of $-\triangle$ with eigenvalue $|k|^2$, the expression of $\varphi \left( 2^{-j}\sqrt{-\triangle }\right) e_k$ by the spectral decomposition of $-\triangle$ and the definition of $\Delta _j$ imply
\[
\Delta _j e_k = \varphi \left( 2^{-j}\sqrt{-\triangle }\right) e_k = \varphi (2^{-j}|k|) e_k, \quad k\in {\mathbb Z}^3. 
\]
Similarly $\Delta _{-1} e_k = \chi (|k|)e_k$ for $k\in {\mathbb Z}^3$.
From this, the Fourier expansions of ${\mathcal Z}^{(2,2,N)}_t$ and ${\mathcal Z}^{(2,3,N)}_t$ can then be calculated explicitly.
Hence, \ref{prop:Z22} and \ref{prop:Z23} are proved similarly as we did for \ref{prop:Z2}.

The proofs of \ref{prop:Z13} and \ref{prop:Z133} are done also by explicit calculation as above.
See Section 1.2 of \cite{MW3} for details. 
\qed

\section{Construction of the invariant measure and flow}\label{sec:main}

In this section, we will construct an invariant probability measure and a flow associated to the $\Phi ^4_3$-measure.
We use the same notations as in Sections \ref{sec:Besov} and \ref{sec:OU}.

Let $\lambda _0 \in (0,\infty )$ and $\lambda \in (0,\lambda _0]$ be fixed.
Define a function $U_N$ on ${\cal D}'(\Lambda)$ by
\[
U_N(\phi ) = \int _{\Lambda} \left\{ \frac{\lambda}{4} (P_N^{(1)} \phi ) (x)^4 - \frac{3\lambda }{2}\left( C_1^{(N)} -3\lambda C_2^{(N)}\right) (P_N ^{(1)} \phi )(x)^2 \right\} dx ,
\]
and consider the probability measure $\mu _N$ on ${\cal D}'(\Lambda)$ given by
\[
\mu _N (d\phi ) = Z_N ^{-1} \exp \left( -U_N(\phi ) \right) \tilde \mu _0 (d\phi)
\]
where $Z_N$ is the normalizing constant.
We remark that $\{ \mu _N\}$ is an approximation sequence for the $\Phi ^4_3$-measure which will be constructed below as an invariant probability measure of the flow associated with the stochastic quantization equation.

Consider the stochastic partial differential equation on $\Lambda$
\begin{equation}\label{eq:SDEN1}\begin{array}{rl}
\displaystyle d Y_t^{N}(x)
&\displaystyle = d W_t(x) - (-\triangle +m_0^2)Y_t^N(x) dt \\[3mm]
&\displaystyle \quad - \lambda P_N ^{(1)} \left\{ (P_N ^{(1)} Y_t^N)^3 (x) -3 \left( C_1^{(N)} -3\lambda C_2^{(N)}\right) P_N ^{(1)} Y_t^N(x) \right\} dt
\end{array}\end{equation}
where $dW_t(x)$ is a white noise with parameter $(t,x)\in [0,\infty)\times \Lambda$.
First, we prove that this SPDE is associated to $\mu _N$, in the sense that $\mu _N$ is the invariant measure for $Y_t^N$.

\begin{thm}\label{thm:invN}
Let $\alpha \in (1/2,\infty)$.
For each $N$, (\ref{eq:SDEN1}) has a unique global solution as a stochastic process on $W^{-\alpha ,2}(\Lambda)$ almost surely for all initial values $Y^N_0\in W^{-\alpha ,2}(\Lambda)$.
Moreover, $\mu _N$ is the invariant measure with respect to $Y^N$.
\end{thm}

\begin{proof}
To simplify notation, we denote $\psi ^{(1)} (2^{-N}|\cdot |) ^{\otimes 3}$ by $\psi _N ^{(1)}$.
Denote $\langle Y^N_t, e_k \rangle$ and $\langle W_t, e_k \rangle$ by ${\widehat Y}_t^{N,k}$ and ${\widehat W}_t^k$, respectively, for $k\in {\mathbb Z}^3$, and consider the Fourier expansion of $Y^N_t$ as
\begin{equation}\label{eq:FourierXN}
Y^N_t(x)=\sum_{k\in {\mathbb Z}^3} {\widehat Y}_t^{N,k} e_k(x).
\end{equation}
Then, the stochastic differential equation associated to (\ref{eq:SDEN1}) is given by
\begin{equation}\label{eq:SDEN2}\begin{array}{rl}
\displaystyle d{\widehat Y}_t^{N,k}
&\displaystyle = d{\widehat W}_t^k - (k^2+m_0^2){\widehat Y}_t^{N,k}dt \\[3mm]
&\displaystyle \quad - \frac{\lambda}{(2\pi )^3} \sum_{\substack{l_1,l_2,l_3\in {\mathbb Z}^3;\\ l_1+l_2+l_3=-k}} \psi _N ^{(1)} (l_1) \psi _N ^{(1)} (l_2) \psi _N ^{(1)} (l_3) \psi _N ^{(1)} (k) {\widehat Y}_t^{N,l_1} {\widehat Y}_t^{N,l_2} {\widehat Y}_t^{N,l_3} dt \\
&\displaystyle \quad +3\lambda \left( C_1^{(N)} -3\lambda C_2^{(N)}\right) \psi _N ^{(1)} (k) ^2 {\widehat Y}_t^{N,k} dt 
\end{array}\end{equation}
where $k\in {\mathbb Z}^3$.
Once we prove the existence and the uniqueness of the solution to (\ref{eq:SDEN2}), we obtain the existence and the uniqueness of the solution to (\ref{eq:SDEN1}) by means of (\ref{eq:FourierXN}).

If $k\not \in {\mathbb Z}_{2N}^3$, (\ref{eq:SDEN2}) reduces to
\[
d{\widehat Y}_t^{N,k} = d{\widehat W}_t^k - (k^2+m_0^2){\widehat Y}_t^{N,k}dt.
\]
This implies that ${\widehat Y}^{N,k}$ with $k\not \in {\mathbb Z}_{2N}^3$ has no interaction with the other components, and the solution ${\widehat Y}^{N,k}$ exists and is unique almost surely for $k\not \in {\mathbb Z}_{2N}^3$.
In particular, similarly as for (\ref{eq:SDEZ2}),
\begin{equation}\label{eq:flow01}
{\widehat Y}_t^{N,k} = e^{-t(k^2 +m_0^2)} {\widehat Y}_0^{N,k} + \int _0^t e^{(s-t)(k^2+m_0^2)}d{\widehat W}_s^k
\end{equation}
for $k\not \in {\mathbb Z}_{2N}^3$.
In view of this fact, we can regard (\ref{eq:SDEN2}) with $k \in {\mathbb Z}_{2N}^3$ as a finite-dimensional stochastic differential equation, from now on.
The existence and the pathwise uniqueness of the local solution in time to (\ref{eq:SDEN2}) with $k \in {\mathbb Z}_{2N}^3$ immediately follow from the local Lipschitz continuity of the coefficients of (\ref{eq:SDEN2}).
Now we show that the global existence of the solution holds.
Let $T>0$.
Define a stopping time $\tau _M := \min\{ T, \inf \{t>0; \sum_{k\in {\mathbb Z}_{2N}^3}|{\widehat Y}_t^{N,k}|^2 > M\}\}$ for $M\in [0,\infty)$.
Then, by It\^o's formula we have for any $\tilde t \in [0,\infty)$
\begin{align*}
&E\left[ \sup_{t\in [0,\tilde t \wedge \tau _M]} \sum _{k\in {\mathbb Z}_{2N}^3} \left| {\widehat Y}_t^{N,k}\right| ^2 \right] - E\left[ \sum _{k\in {\mathbb Z}_{2N}^3} \left| {\widehat Y}_0^{N,k}\right| ^2 \right]\\
&\leq E\left[ \sup_{t\in [0,\tilde t \wedge \tau _M]} \sum _{k\in {\mathbb Z}_{2N}^3} \left( \int _0^t {\widehat Y}_s^{N,k} d{\widehat W}_s^{-k} + \int _0^t {\widehat Y}_s^{N,-k} d {\widehat W}_s^k\right) \right] + (2^{N+2}+1)^3 \tilde t \\
&\quad + 2E\left[ \sup_{t\in [0,\tilde t \wedge \tau _M]} \sum _{k\in {\mathbb Z}_{2N}^3} [-(k^2+m_0^2)] \int _0^t |{\widehat Y}_s^{N,k}|^2 ds \right] \\
&\quad + E\left[ \sup_{t\in [0,\tilde t \wedge \tau _M]} \left( -\frac{\lambda}{(2\pi )^3} \right) \sum _{k\in {\mathbb Z}_{2N}^3} \sum_{\substack{l_1,l_2,l_3\in {\mathbb Z}_{2N}^3;\\ l_1+l_2+l_3=-k}} \psi _N ^{(1)} (l_1) \psi _N ^{(1)} (l_2) \psi _N ^{(1)} (l_3) \psi _N ^{(1)} (k) \right. \\
&\quad \hspace{2cm} \left. \phantom{\sum _{k\in {\mathbb Z}_{2N}^3}} \times \int _0^t ({\widehat Y}_s^{N,l_1} {\widehat Y}_s^{N,l_2} {\widehat Y}_s^{N,l_3} {\widehat Y}_s^{N,-k} + {\widehat Y}_s^{N,-l_1} {\widehat Y}_s^{N,-l_2} {\widehat Y}_s^{N,-l_3} {\widehat Y}_s^{N,k})ds \right] \\
&\quad + 6\lambda E\left[ \sup_{t\in [0,\tilde t \wedge \tau _M]}\left( C_1^{(N)} - 3\lambda C_2^{(N)}\right) \sum _{k\in {\mathbb Z}_{2N}^3} \psi _N ^{(1)} (k) ^2 \int _0^t |{\widehat Y}_s^{N,k}|^2 ds \right] \\
&\leq  \sum _{k\in {\mathbb Z}_{2N}^3} E\left[  \left\langle \int _0^\cdot {\widehat Y}_s^{N,k} d{\widehat W}_s^{N,-k} + \int _0^\cdot {\widehat Y}_s^{N,-k} d {\widehat W}_s^{N,k}\right\rangle _{\tilde t \wedge \tau _M} \right] ^{1/2} + (2^{N+2}+1)^3 \tilde t\\
&\quad +E\left[ \sup_{t\in [0,\tilde t \wedge \tau _M]} \left( -\frac{2\lambda}{(2\pi )^{3/2}} \right) \int _0^t \left\langle ( P_N ^{(1)} Y^N_s) ^4 , 1 \right\rangle ds \right] +6 \lambda C_1^{(N)} E\left[ \sum _{k\in {\mathbb Z}_{2N}^3} \int _0^{\tilde t \wedge \tau _M} |{\widehat Y}_s^{N,k}|^2 ds \right] \\
&\leq \sum _{k\in {\mathbb Z}_{2N}^3} E\left[ \int _0^{\tilde t \wedge \tau _M} |{\widehat Y}_s^{N,k}|^2 ds \right] ^{1/2} + 6\lambda C_1^{(N)} E\left[ \sum _{k\in {\mathbb Z}_{2N}^3} \int _0^{\tilde t \wedge \tau _M} |{\widehat Y}_s^{N,k}|^2 ds \right] + (2^{N+2}+1)^3 \tilde t\\
&\leq (2^{N+2}+1)^3 (1+\tilde t) +\left( 1+ 6\lambda C_1^{(N)}\right) E\left[ \int _0^{\tilde t} \sup _{r\in [0,s \wedge \tau _M]} \sum _{k\in {\mathbb Z}_{2N}^3} |{\widehat Y}_r^{N,k}|^2 ds \right]
\end{align*}
where $\langle \cdot \rangle $ here means the quadratic variation.
Hence, by Gronwall's inequality we have for $\tilde t \in [0,\infty)$
\[
E\left[ \sup_{t\in [0,\tilde t \wedge \tau _M]} \sum _{k\in {\mathbb Z}_{2N}^3} \left| {\widehat Y}_t^{N,k}\right| ^2 \right] \leq (2^{N+2}+1)^3\left( 1 + \tilde t + E\left[ \sum _{k\in {\mathbb Z}_{2N}^3} \left| {\widehat Y}_0^{N,k}\right| ^2 \right] e^{ 6\lambda C_1^{(N)} \tilde t} \right) .
\]
By letting $\tilde t = T$ and $M\uparrow \infty$ in this inequality and combining it with (\ref{eq:flow01}) we get
\[
\sum _{k\in {\mathbb Z}^3}(1+ k^2)^{-\alpha} E\left[ \sup_{t\in [0,T]} \left| \widehat{Y}_t^{N,k}\right| ^2 \right]  \leq C \left( 1+ \sum _{k\in {\mathbb Z}^3}(1+ k^2)^{-\alpha} E \left[ \left| \widehat{Y}_0^{N,k}\right| ^2 \right] \right) e^{CT},
\]
where $C$ is a constant depending on $N$ and $m_0$.
Thus, we have the unique global solution as a stochastic process $Y^N$ on $W^{-\alpha ,2}(\Lambda)$ almost surely for all initial values $Y^N_0\in W^{-\alpha ,2}(\Lambda)$.

For the invariance of $\mu _N$ under the solution of (\ref{eq:SDEN1}), consider the differential operator
\[
A_N :=   \sum_{k\in {\mathbb Z}_{2N}^3} \exp \left( \frac{1}{2} \sum _{l\in {\mathbb Z}_{2N}^3} (l^2+m_0^2) |x_l|^2 + V(x) \right)\frac{\partial }{\partial \bar{x}_k} \exp \left( -  \frac{1}{2} \sum _{l\in {\mathbb Z}_{2N}^3} (l^2+m_0^2) |x_l|^2 -V(x) \right) \frac{\partial }{\partial x_k}
\]
for $x=(x_k) _{k\in {\mathbb Z}_{2N}^3}$ and $x_k \in {\mathbb C}$, where
\begin{align*}
V(x) &= \frac{\lambda}{4(2\pi )^3} \sum_{\substack{l_1,l_2,l_3,l_4\in {\mathbb Z}_{2N}^3;\\ l_1+l_2+l_3+l_4=0}} \psi _N ^{(1)} (l_1) \psi _N ^{(1)} (l_2) \psi _N ^{(1)} (l_3) \psi _N ^{(1)} (l_4) x_{l_1} x_{l_2} x_{l_3} x_{l_4} \\
&\quad \hspace{2cm} - \frac{3\lambda}{2} \left(  C_1^{(N)} - 3 \lambda C_2^{(N)}\right)  \sum _{\substack{l_1,l_2\in {\mathbb Z}_{2N}^3;\\ l_1+l_2 =0}} \psi _N ^{(1)} (l_1) \psi _N ^{(1)} (l_2) x_{l_1} x_{l_2}\\
\frac{\partial }{\partial x_k}&:= \frac 12 \left( \frac{\partial}{\partial {\rm Re}x_k} - \sqrt{-1}\frac{\partial}{\partial {\rm Im}x_k} \right) , \quad \frac{\partial }{\partial \bar{x}_k}:= \frac 12 \left( \frac{\partial}{\partial {\rm Re}x_k} + \sqrt{-1} \frac{\partial}{\partial {\rm Im}x_k} \right) .
\end{align*}
Then, by the standard argument by conformal martingales (see Section 6 of Chapter III in \cite{IW}) and Dirichlet forms (see \cite{FOT}) we see that $A_N$ is the generator of $({\widehat Y}^{N,k}; k \in {\mathbb Z}_{2N}^3)$ and the measure
\[
{\widehat \mu}_N^{(1)} (dx) := \left( {\widehat Z}_N^{(1)} \right) ^{-1} \exp \left( - \frac{1}{2} \sum _{l\in {\mathbb Z}_{2N}^3} (l^2+m_0^2) |x_l|^2 -V(x) \right) \prod _{k\in {\mathbb Z}_{2N}^3} dx_k,
\]
where $dx_k$ is the Lebesgue measure on ${\mathbb C}$ for $k\in {\mathbb Z}^3$ and ${\widehat Z}_N^{(1)}$ is a normalization constant, is the unique invariant measure associated to $({\widehat Y}^{N,k}; k \in {\mathbb Z}_{2N}^3)$.
For $k \in {\mathbb Z}^3 \setminus {\mathbb Z}_{2N}^3$, as we have seen above, ${\widehat Y}^{N,k}$ has no interaction with other components and satisfies (\ref{eq:flow01}).
Moreover, it is easy to see that for $k \in {\mathbb Z}^3 \setminus {\mathbb Z}_{2N}^3$
\[
{\widehat \mu}_N^{(2),k} (dx) := \left( {\widehat Z}_N^{(2)} \right) ^{-1} \exp \left( - \frac{k^2+m_0^2}{2}|x|^2\right) dx
\]
where ${\widehat Z}_N^{(2)}$ is a normalization constant, is the invariant measure associated to ${\widehat Y}^{N,k}$.
Hence,
\[
{\widehat \mu}_N^{(1)} \otimes \prod _{k \in {\mathbb Z}^3 \setminus {\mathbb Z}_{2N}^3} {\widehat \mu}_N^{(2),k}
\]
is the invariant measure associated to $({\widehat Y}^{N,k}; k \in {\mathbb Z}^3)$.
Therefore, $\mu _N$ is the invariant measure associated to $Y^N$.
\end{proof}

For each $N\in {\mathbb N}$, consider a stochastic process $Y_t^N$ given by (\ref{eq:SDEN1}) with initial law $\mu _N$.
We extend $W$ appeared in (\ref{eq:SDEN1}) to a white noise for $(t,x) \in (-\infty ,\infty ) \times \Lambda$, define $Z_t$ by (\ref{eq:SDEZ}) with $W$, and assume that $Y_0^N$ and $Z_0$ are independent.
Then, in view of Theorem \ref{thm:invN}, $Y_t^N$ and $Z_t$ are stationary processes.
In particular, each of the families of laws $\{ Y_t^N ; t\in [0,\infty )\}$ and $\{ Z_t; t\in [0,\infty )\}$ are tight.
Corollary \ref{cor:AT} implies the laws of the pair $\{(Y_t^N, Z_t); t\in [0,\infty )\}$ are also tight.
Hence, by Proposition \ref{prop:Ainv} we have an invariant probability measure for the system $(Y_t^N, Z_t)$.
Let $(\xi _N ,\zeta)$ be a pair of random variables whose law is the invariant probability measure.

Now we fix a pair of random variables $(\xi _N ,\zeta)$.
Consider the stochastic partial differential equation on $\Lambda$
\begin{equation}\label{eq:SDEN3}\left\{ \begin{array}{rl}
\displaystyle d \tilde X_t^{N}(x)
&\displaystyle = d W_t(x) - (-\triangle +m_0^2) \tilde X_t^N(x) dt \\[3mm]
&\displaystyle \quad - \lambda P_N ^{(1)} \left\{ (P_N ^{(1)} \tilde X_t^N)^3 (x) -3 \left( C_1^{(N)} -3\lambda C_2^{(N)}\right) P_N ^{(1)} \tilde X_t^N(x) \right\} dt\\
\displaystyle \tilde X_0^{N}(x)
&\displaystyle = \xi _N (x)
\end{array}\right. \end{equation}
where $W_t$ is a white noise independent of $(\xi _N ,\zeta)$.
Note that (\ref{eq:SDEN3}) is the equation with time evolution the same as (\ref{eq:SDEN1}) with initial law $\mu_N$.
Let $X^{N} := P_N ^{(2)} \tilde X^N$ for $N \in {\mathbb N}$.
Then, in view of the fact that $P_N^{(2)}P_N^{(1)} = P_N^{(1)}$, $X^{N}$ satisfies the stochastic partial differential equation
\begin{equation}\label{eq:SDEN4}\left\{ \begin{array}{rl}
\displaystyle d X_t^{N}(x)
&\displaystyle = P_N^{(2)} dW_t(x) - (-\triangle +m_0^2) X_t^{N}(x) dt \\[3mm]
&\displaystyle \quad - \lambda P_N^{(1)} \left\{ (P_N ^{(1)} X_t^N)^3 (x) -3 (C_1^{(N)}- 3\lambda C_2^{(N)}) P_N^{(1)} X_t^N(x) \right\} dt\\
\displaystyle X_0^{N}(x)
&\displaystyle = P_N^{(2)} \xi _N (x).
\end{array}\right. \end{equation}
By the definition, $X^{N} \in C([0,\infty); B_{p}^s)$ for $s\in {\mathbb R}$ and $p\in [1,\infty ]$.
Since for $k\in {\mathbb Z}_{N+2}^3$
\begin{align*}
&E[ | \langle P_N^{(2)} \xi _N , e_k\rangle | ^2] \\
&= \frac{1}{Z_N} \int _{{\mathcal S}'(\Lambda)} ( \psi ^{(2)} (2^{-N}|\cdot |) ^{\otimes 3}(k) )^2 | \langle \phi , e_k\rangle | ^2\\
&\quad \hspace{1cm} \times \exp \left( - \int _{\Lambda} \left(  \frac{\lambda}{4} (P_N ^{(1)} \phi ) ^4 - \frac{3\lambda }{2}(C_1^{(N)}-3\lambda C_2^{(N)}) (P_N^{(1)} \phi )^2 \right) dx\right) \mu _0(d\phi ) \\
&\leq \frac{1}{Z_N} \exp \left( \frac{9\lambda }{4} (C_1^{(N)}-3\lambda C_2^{(N)})^2 \right) \int _{{\mathcal S}'(\Lambda)} | \langle \phi , e_k\rangle | ^2 \mu _0(d\phi ) \\
&= \frac{1}{(k^2+m_0^2)Z_N} \exp \left( \frac{9\lambda }{4} (C_1^{(N)}-3\lambda C_2^{(N)})^2 \right) ,
\end{align*}
by the invariance of $\mu _N$ with respect to $\tilde X^N$ we have
\begin{equation}\label{eq:invXN}
E\left[ \| X_t^{N} \| _{L^2}^2 \right] = \sum _{k\in {\mathbb Z}_{N+2}^3} \left[ \psi ^{(2)}(2^{-N}|\cdot |) ^{\otimes 3}(k) \right] ^2 E\left[ | \langle \tilde X_t^N, e_k \rangle |^2 \right] = E\left[ \| X_0^{N} \| _{L^2}^2 \right] < \infty 
\end{equation}
for $t\in [0,\infty )$.

Now we shall investigate the tightness of the laws of  $\{ X^{N}\}$ (see Theorem \ref{thm:tight2} below).
To solve (\ref{eq:SDEN4}) we apply to our equation a method inspired by the one used by M. Hairer in his setting in \cite{Ha}, however we keep entirely in the paracontrolled decomposition setting.
We use the notation of paraproducts and of polynomials of the Ornstein-Uhlenbeck process as in Sections \ref{sec:Besov} and \ref{sec:OU}, respectively.
In particular, we extend $W$ in (\ref{eq:SDEN3}) to a white noise for $(t,x) \in (-\infty ,\infty ) \times \Lambda$ and define $Z_t$ by (\ref{eq:SDEZ}) with the extended $W$, where $\zeta$ is the random variable defined above.
We remark that the pair $(X_t^N, Z_t)$ is a stationary process by the construction of $(\xi _N, \zeta)$.
Similarly to (\ref{eq:invXN}) we have
\begin{equation}\label{eq:invZN}
E\left[ \|P_N^{(2)}Z_t \| _{L^2}^2 \right] = E\left[ \| P_N^{(2)}Z_0 \| _{L^2}^2 \right] < \infty 
\end{equation}
for $t\in [0,\infty )$.
Let $X^{N,(1)}_t:= X_t^{N} - P_N^{(2)}Z_t$ for $t\in [0,\infty )$.
From (\ref{eq:SDEN1}) and (\ref{eq:SDEZ}) we have
\begin{align*}
&d X^{N,(1)}_t + (- \triangle +m_0^2) X^{N,(1)}_t dt\\
&= -\lambda P_N^{(1)} \left[ (P_N^{(1)} X^{N,(1)}_t + {\mathcal Z}_t^{(1,N)}) ^3 \right] dt \\
&\quad + 3 \lambda \left( C_1^{(N)} - 3\lambda C_2^{(N)}\right) P_N^{(1)} (P_N^{(1)} X^{N,(1)}_t + {\mathcal Z}_t^{(1,N)}) dt\\
&= -\lambda P_N^{(1)} \left[ (P_N^{(1)} X^{N,(1)}_t)^3 \right] dt - 3\lambda P_N^{(1)} \left[ {\mathcal Z}_t^{(1,N)} (P_N^{(1)} X^{N,(1)}_t)^2 \right] dt\\
&\quad  - 3\lambda P_N^{(1)} \left[ {\mathcal Z}^{(2,N)}_t P_N^{(1)} X^{N,(1)}_t \right] dt -\lambda P_N^{(1)} {\mathcal Z}^{(3,N)}_t dt - 9 \lambda ^2 C_2^{(N)} P_N^{(1)}  (P_N^{(1)} X^{N,(1)}_t +{\mathcal Z}_t^{(1,N)}) dt.
\end{align*}
Let 
\[
X^{N,(2)}_t := X_t^{N} - P_N^{(2)}Z_t + \lambda {\mathcal Z}^{(0,3,N)}_t, \quad t\in [0,\infty ).
\]
Then, we have
\begin{align*}
&d X^{N,(2)}_t + (- \triangle +m_0^2) X^{N,(2)}_t dt\\
&= -\lambda P_N^{(1)} \left[ \left( P_N^{(1)} X^{N,(2)}_t - \lambda P_N^{(1)} {\mathcal Z}^{(0,3,N)}_t \right) ^3 \right] dt \\
&\quad - 3\lambda P_N^{(1)}  \left[ {\mathcal Z}_t^{(1,N)} \left( P_N^{(1)} X^{N,(2)}_t - \lambda P_N^{(1)} {\mathcal Z}^{(0,3,N)}_t \right) ^2 \right] dt \\
&\quad - 3\lambda P_N^{(1)}  \left[ {\mathcal Z}^{(2,N)}_t \left( P_N^{(1)} X^{N,(2)}_t - \lambda P_N^{(1)} {\mathcal Z}^{(0,3,N)}_t \right) \right] dt \\
&\quad - 9\lambda ^2 C_2^{(N)} P_N^{(1)}  \left( P_N^{(1)} X^{N,(2)}_t + {\mathcal Z}_t^{(1,N)} - \lambda P_N^{(1)} {\mathcal Z}^{(0,3,N)}_t \right) dt.
\end{align*}
Hence, the pair $(X^{N,(2),<}_t, X^{N,(2),\geqslant}_t)$ defined by 
\begin{align*}
X^{N,(2),<}_t &:= -3\lambda \int _0^t e^{(t-s)(\triangle -m_0^2)} P_N^{(1)} \left[ \left( P_N^{(1)} X^{N,(2)}_s - \lambda P_N^{(1)} {\mathcal Z}^{(0,3,N)}_s \right) \mbox{\textcircled{\scriptsize$<$}} {\mathcal Z}^{(2,N)}_s \right] ds \\
X^{N,(2),\geqslant}_t & := X^{N,(2)}_t - X^{N,(2),<}_t
\end{align*}
is the solution to the following partial differential equation
\begin{equation} \label{PDEpara}\left\{ \begin{array}{l}
\displaystyle (\partial _t - \triangle + m_0^2) X^{N,(2),<}_t \\
\displaystyle = -3\lambda P_N^{(1)} \left[  \left( P_N^{(1)} X^{N,(2),<}_t + P_N^{(1)} X^{N,(2),\geqslant}_t - \lambda P_N^{(1)} {\mathcal Z}^{(0,3,N)}_t \right) \mbox{\textcircled{\scriptsize$<$}} {\mathcal Z}^{(2,N)}_t \right] \\[5mm]
\displaystyle (\partial _t - \triangle + m_0^2) X^{N,(2),\geqslant}_t \\
\displaystyle  = - \lambda P_N^{(1)}  \left[ \left( P_N^{(1)} X^{N,(2),<}_t + P_N^{(1)} X^{N,(2),\geqslant}_t - \lambda P_N^{(1)} {\mathcal Z}^{(0,3,N)}_t \right) ^3 \right] \\
\displaystyle \quad - 3\lambda P_N^{(1)}  \left[ {\mathcal Z}_t^{(1,N)} \left( P_N^{(1)} X^{N,(2),<}_t + P_N^{(1)} X^{N,(2),\geqslant}_t - \lambda P_N^{(1)} {\mathcal Z}^{(0,3,N)}_t \right) ^2 \right] \\
\displaystyle \quad - 3\lambda P_N^{(1)}  \left[ \left( P_N^{(1)} X^{N,(2),<}_t + P_N^{(1)} X^{N,(2),\geqslant}_t - \lambda P_N^{(1)} {\mathcal Z}^{(0,3,N)}_t \right) \mbox{\textcircled{\scriptsize$\geqslant$}} {\mathcal Z}^{(2,N)}_t \right] \\
\displaystyle \quad - 9\lambda ^2 C_2^{(N)} P_N^{(1)}  \left( P_N^{(1)} X^{N,(2),<}_t + P_N^{(1)} X^{N,(2),\geqslant}_t + {\mathcal Z}_t^{(1,N)} - \lambda P_N^{(1)} {\mathcal Z}^{(0,3,N)}_t \right)
\end{array}\right. \end{equation}
with initial condition $(X^{N,(2),<}_0, X^{N,(2),\geqslant}_0) = (0,X^{N,(2)}_0)$.
Now, we change (\ref{PDEpara}) for another equivalent equation by using the calculus of paraproducts.
By denoting
\begin{align*}
\Psi _t^{(1)} (w) &:= \int _0^t e^{(t-s)(\triangle -m_0^2)} (P_N^{(1)} )^2  \left[ \left( w_s - \lambda P_N^{(1)} {\mathcal Z}^{(0,3,N)}_s \right) \mbox{\textcircled{\scriptsize$<$}} {\mathcal Z}^{(2,N)}_s \right] ds\\
&\qquad - \left( w_t - \lambda P_N^{(1)} {\mathcal Z}^{(0,3,N)}_t \right) \mbox{\textcircled{\scriptsize$<$}} \int _0^t e^{(t-s)(\triangle -m_0^2)} (P_N^{(1)} )^2 {\mathcal Z}^{(2,N)}_s ds , \\
\Psi _t^{(2)} (w) &:= \left[ \left( w_t - \lambda P_N^{(1)} {\mathcal Z}^{(0,3,N)}_t \right) \mbox{\textcircled{\scriptsize$<$}} \int _0^t e^{(t-s)(\triangle -m_0^2)} (P_N^{(1)} )^2 {\mathcal Z}^{(2,N)}_s ds \right] \mbox{\textcircled{\scriptsize$=$}} {\mathcal Z}^{(2,N)}_t\\
&\qquad - \left( w_t - \lambda P_N^{(1)} {\mathcal Z}^{(0,3,N)}_t \right) \left[ \int _0^t e^{(t-s)(\triangle -m_0^2)} (P_N^{(1)} )^2 {\mathcal Z}^{(2,N)}_s ds \mbox{\textcircled{\scriptsize$=$}} {\mathcal Z}^{(2,N)}_t \right] 
\end{align*}
for $w \in C([0,\infty ); L^\infty (\Lambda))$, we have
\begin{align*}
&(P_N^{(1)} X^{N,(2),<}_t) \mbox{\textcircled{\scriptsize$=$}} {\mathcal Z}^{(2,N)}_t\\
&= -3\lambda \left( P_N^{(1)} \int _0^t e^{(t-s)(\triangle -m_0^2)} P_N^{(1)} \left[ \left( P_N^{(1)} X^{N,(2)}_t - \lambda P_N^{(1)} {\mathcal Z}^{(0,3,N)}_s \right) \mbox{\textcircled{\scriptsize$<$}} {\mathcal Z}^{(2,N)}_s \right] ds \right) \mbox{\textcircled{\scriptsize$=$}} {\mathcal Z}^{(2,N)}_t\\
&= -3\lambda \left( P_N^{(1)} X^{N,(2)}_t - \lambda P_N^{(1)} {\mathcal Z}^{(0,3,N)}_t \right) \left[ \int _0^t e^{(t-s)(\triangle -m_0^2)} (P_N^{(1)} )^2 {\mathcal Z}^{(2,N)}_s ds \mbox{\textcircled{\scriptsize$=$}} {\mathcal Z}^{(2,N)}_t \right] \\
&\quad -3\lambda \Psi _t^{(1)} ( P_N^{(1)} X^{N,(2)} ) \mbox{\textcircled{\scriptsize$=$}} {\mathcal Z}^{(2,N)}_t -3\lambda \Psi _t^{(2)} ( P_N^{(1)} X^{N,(2)})
\end{align*}
for $t\in [0,\infty )$.
In view of this equality, (\ref{PDEpara}) is equivalent to
\begin{equation} \label{PDEpara2}\left\{ \begin{array}{l}
\displaystyle (\partial _t - \triangle + m_0^2) X^{N,(2),<}_t \\
\displaystyle = -3\lambda P_N^{(1)}  \left[ \left( P_N^{(1)} X^{N,(2),<}_t + P_N^{(1)} X^{N,(2),\geqslant}_t - \lambda P_N^{(1)} {\mathcal Z}^{(0,3,N)}_t \right) \mbox{\textcircled{\scriptsize$<$}} {\mathcal Z}^{(2,N)}_t \right] \\[4mm]
\displaystyle (\partial _t - \triangle + m_0^2) X^{N,(2),\geqslant}_t \\[1mm]
\displaystyle  = - \lambda P_N^{(1)} \left[ \left( P_N^{(1)} X^{N,(2),<}_t + P_N^{(1)} X^{N,(2),\geqslant}_t \right) ^3 \right] \\
\displaystyle \quad + \lambda P_N^{(1)} \Phi _t^{(1)}(P_N^{(1)} X^{N,(2),<} + P_N^{(1)} X^{N,(2),\geqslant}) \\
\displaystyle \quad + \lambda P_N^{(1)} \Phi _t^{(2)}(P_N^{(1)} X^{N,(2),<} + P_N^{(1)} X^{N,(2),\geqslant}) \\
\displaystyle \quad + \lambda P_N^{(1)}  \Phi _t ^{(3)} (P_N^{(1)} X^{N,(2),<} + P_N^{(1)} X^{N,(2),\geqslant}) - 3\lambda P_N^{(1)} \left[ ( P_N^{(1)} X^{N,(2),\geqslant}_t) \mbox{\textcircled{\scriptsize$=$}} {\mathcal Z}^{(2,N)}_t \right] \\[1mm]
\displaystyle \quad + 9\lambda ^2 P_N^{(1)} \left[ \Psi _t^{(1)} (P_N^{(1)} X^{N,(2),<} + P_N^{(1)} X^{N,(2),\geqslant} ) \mbox{\textcircled{\scriptsize$=$}} {\mathcal Z}^{(2,N)}_t \right] \\
\displaystyle \quad + 9\lambda ^2 P_N^{(1)} \Psi _t^{(2)} (P_N^{(1)} X^{N,(2),<} + P_N^{(1)} X^{N,(2),\geqslant})
\end{array}\right. \end{equation}
where for $w \in C([0,\infty ); L^\infty (\Lambda))$
\begin{align*}
\Phi _t^{(1)}(w)
&:= - 3 \left( {\mathcal Z}_t^{(1,N)} - \lambda P_N^{(1)} {\mathcal Z}^{(0,3,N)}_t\right) \mbox{\textcircled{\scriptsize$\leqslant$}} w_t^2 + 3 \lambda \left[ \left( 2{\mathcal Z}_t^{(1,N)} - \lambda P_N^{(1)} {\mathcal Z}^{(0,3,N)}_t \right) {\mathcal Z}^{(0,3,N)}_t \right] \mbox{\textcircled{\scriptsize$\leqslant$}} w_t ,\\
\Phi _t ^{(2)} (w)
&:= - 3 \left( w_t - \lambda P_N^{(1)} {\mathcal Z}^{(0,3,N)}_t \right) \mbox{\textcircled{\scriptsize$>$}} {\mathcal Z}^{(2,N)}_t + 3 \lambda {\mathcal Z}^{(2,3,N)}_t \\
&\quad + 9\lambda \left( w_t - \lambda P_N^{(1)} {\mathcal Z}^{(0,3,N)}_t \right) \\
&\quad \hspace{2cm} \times \left( {\mathcal Z}^{(2,2,N)}_t - {\mathcal Z}^{(2,N)}_t\mbox{\textcircled{\scriptsize$=$}}\int _{-\infty }^0 e^{(t-s)(\triangle -m_0^2)} \left( P_N^{(1)}\right) ^2 {\mathcal Z}^{(2,N)}_s ds\right) \\
&\quad - \lambda ^2 \left( 3{\mathcal Z}_t^{(1,N)} - \lambda P_N^{(1)} {\mathcal Z}^{(0,3,N)}_t \right) \left( P_N^{(1)} {\mathcal Z}^{(0,3,N)}_t \right) ^2, \\
\Phi _t ^{(3)} (w)
&:= -3 \left( {\mathcal Z}_t^{(1,N)} - \lambda P_N^{(1)} {\mathcal Z}^{(0,3,N)}_t\right) \mbox{\textcircled{\scriptsize$>$}} w_t^2 \\
&\quad \hspace{5cm} + 3 \lambda \left[ \left( 2{\mathcal Z}_t^{(1,N)} - \lambda P_N^{(1)} {\mathcal Z}^{(0,3,N)}_t \right) P_N^{(1)} {\mathcal Z}^{(0,3,N)}_t \right] \mbox{\textcircled{\scriptsize$>$}} w_t .
\end{align*}
For $\eta \in [0,1)$, $\gamma \in (0,1/4)$ and $\varepsilon \in (0,1]$ define ${\mathfrak X}_{\eta ,\gamma }^N (t)$ and  ${\mathfrak Y}_{\varepsilon}^N (t)$ by
\begin{align*}
{\mathfrak X}_{\lambda , \eta ,\gamma }^N (t) &:= \int _0^t \left( \left\| \nabla X_s^{N,(2),\geqslant}\right\| _{L^2}^2 + \left\| X_s^{N,(2)}\right\| _{L^2}^2 + \lambda \left\| P_N^{(1)} X_s^{N,(2)}\right\| _{L^4}^4 \right) ds \\
&\quad + \sup _{s',t' \in [0,t]; s'<t'} \frac{(s')^\eta \left\| X^{N,(2)}_{t'} - X^{N,(2)}_{s'} \right\| _{L^{4/3}}}{(t'-s')^{\gamma }}\\
{\mathfrak Y}_{\varepsilon}^N (t)&:= \int _0^t \left\| X^{N,(2),<}_s \right\| _{B_{4}^{1-(\varepsilon /12)}}^3 ds + \int _0^t \left\| X^{N,(2),\geqslant}_s \right\| _{B_{4/3}^{1+\varepsilon }} ds
\end{align*}
respectively.
We are going to estimate $E\left[ {\mathfrak X}_{\lambda , \eta ,\gamma}^N (T)\right]$ and $E\left[ {\mathfrak Y}_{\varepsilon}^N (T) ^q\right]$ for given $T \in (0,\infty )$ and $q\in (1,8/7)$.
To simplify the notation, we denote by $Q$ a positive polynomial built with the following quantities 
\begin{align*}
&\sup_{t\in [0,T]}\| {\mathcal Z}_t^{(1,N)}\| _{B_{\infty}^{-(1+\varepsilon )/2}}, \quad \sup_{t\in [0,T]}\| P_N^{(2)} Z_t \| _{B_{\infty}^{-(1+\varepsilon )/2}}, \quad \sup _{t\in [0,T]} \left\| {\mathcal Z}^{(2,N)}_t \right\| _{B_{\infty}^{-1-\varepsilon /24}},\\
&\sup _{t\in [0,T]}  \left\| {\mathcal Z}^{(2,2,N)}_t \right\| _{B_{\infty }^{-\varepsilon /4}}, \quad \sup _{t\in [0,T]} \left\| {\mathcal Z}^{(0,2,N)}_t \right\| _{B_{\infty }^{1-\varepsilon /2}}, \quad \sup _{t\in [0,T]} \left\| {\mathcal Z}^{(0,3,N)}_t \right\| _{B_{\infty }^{1/2-\varepsilon /4}}, \\
&\sup _{t\in [0,T]} \left\| {\mathcal Z}^{(2,3,N)}_t \right\| _{B_{\infty }^{-(1+\varepsilon )/2}}, \quad \sup _{t\in [0,T]} \left\| {\mathcal Z}_t^{(1,N)} \left( P_N^{(1)} {\mathcal Z}^{(0,3,N)}_t \right) \right\| _{B_{\infty }^{-(1+\varepsilon )/2}}, \\
&\sup _{t\in [0,T]} \left\| {\mathcal Z}_t^{(1,N)} \left( P_N^{(1)} {\mathcal Z}^{(0,3,N)}_t \right) ^2\right\| _{B_{\infty }^{-(1+\varepsilon )/2}} \quad \mbox{and} \quad \sup _{s,t \in [0,T]} \frac{\left\| {\mathcal Z}^{(0,3,N)}_t - {\mathcal Z}^{(0,3,N)}_s \right\| _{L^\infty }}{(t-s)^{\gamma }},
\end{align*}
with coefficients depending on $\lambda _0 $, $\varepsilon$, $\eta$, $\gamma$ and $T$, and we also denote by $C$ a positive constant depending on $\lambda _0$, $\varepsilon$, $\eta$, $\gamma$ and $T$.
A constant depending on an extra parameter $\delta$ is denoted by $C_\delta$.
We remark that $Q$, $C$ and $C_\delta$ can be different from line to line and that Proposition \ref{prop:Z} implies $E[Q] \leq C$ for some $C$.

\begin{lem}\label{lem:v3}
For $\varepsilon \in (0,1/4]$, $t\in [0,T]$, and $\delta \in (0,1]$, the following inequality holds for some positive $Q$ as above:
\[
\int _0^t \| X_s^{N,(2),<}\| _{B_4^{1-\varepsilon /12}} ^3 dt \leq \delta \left( \int _0^t \left\| P_N^{(1)} X^{N,(2)}_s \right\| _{L^4} ^4 ds \right) ^{7/8} + Q \delta ^{-6}
\]
almost surely.
\end{lem}

\begin{proof}
By (\ref{PDEpara2}) and Propositions \ref{prop:paraproduct} and \ref{prop:bddp} we have for $t\in [0,T]$
\begin{align*}
\| X_t^{N,(2),<}\| _{B_4^{1-\varepsilon /12}}
&\leq C \int _0^t (t-s)^{-1+\varepsilon /48} \left\| \left( P_N^{(1)} X^{N,(2)}_s - \lambda P_N^{(1)} {\mathcal Z}^{(0,3,N)}_s \right) \mbox{\textcircled{\scriptsize$<$}} {\mathcal Z}_s^{(2,N)} \right\| _{B_4^{-1-\varepsilon /24}} ds \\
&\leq Q \int _0^t (t-s)^{-1+\varepsilon /48} \left\| P_N^{(1)} X^{N,(2)}_s \right\| _{L^4} ds + Q.
\end{align*}
Hence, by Young's inequality we obtain for $t\in [0,T]$
\[
\int _0^t \| X_s^{N,(2),<}\| _{B_4^{1-\varepsilon /12}} ^3 ds \leq Q \int _0^t \left\| P_N^{(1)} X^{N,(2)}_s \right\| _{L^4} ^3 ds + Q.
\]
This yields the desired inequality through H\"older's inequality.
\end{proof}

\begin{lem}\label{lem:com1}
For $\gamma \in (0,1/8)$, $\varepsilon \in (0, \gamma /2)$, $p\in [1,2]$, $t\in [0,T]$ and $\theta \in (0,1]$
\begin{align*}
&\left\| \Psi _t^{(1)} (P_N^{(1)} X^{N,(2)}) \right\| _{B_p^{1+2\varepsilon }} \\
&\leq Q \int _0^t (t-s)^{-21/32} \left\| P_N^{(1)} X^{N,(2)}_s \right\| _{B_{p}^{15/16}} ds\\
&\quad + Q \left( \sup _{s\in [0,t]} \frac{s^\eta \left\| P_N^{(1)} X^{N,(2)}_t - P_N^{(1)} X^{N,(2)}_s \right\| _{L^p} }{(t-s)^{\gamma }} \right)^{\theta}\\
&\quad \hspace{1cm} \times \left(\| P_N^{(1)} X^{N,(2)}_t \| _{L^p}^{1-\theta} + \int _0^t s^{-\eta \theta }(t-s)^{\theta \gamma - 1 -3\varepsilon /2} \left\| P_N^{(1)} X^{N,(2)}_s \right\| _{L^p} ^{1-\theta} ds \right) + Q
\end{align*}
almost surely.
\end{lem}

\begin{proof}
Let $s\in [0,t)$.
Then, Propositions \ref{prop:paraproduct}, \ref{prop:bddp} and \ref{prop:com3} imply that
\begin{align*}
&\left\| e^{(t-s)(\triangle - m_0)} (P_N^{(1)})^2 \left[ \left( P_N^{(1)} X^{N,(2)}_s - \lambda P_N^{(1)} {\mathcal Z}^{(0,3,N)}_s \right) \mbox{\textcircled{\scriptsize$<$}} {\mathcal Z}_s^{(2,N)} \right] \phantom{\left[ \left( X^{N,(2)}_t - \lambda P_N^{(1)} {\mathcal Z}^{(0,3,N)}_t \right) \right] }\right. \\
&\quad \hspace{2cm} \left. - \left( P_N^{(1)} X^{N,(2)}_t - \lambda P_N^{(1)} {\mathcal Z}^{(0,3,N)}_t \right) \mbox{\textcircled{\scriptsize$<$}} \left( e^{(t-s)(\triangle - m_0)}  (P_N^{(1)})^2 {\mathcal Z}_s^{(2,N)} \right) \right\| _{B_{p}^{1+2\varepsilon }} \\
&\leq \left\| e^{(t-s)(\triangle - m_0)}  (P_N^{(1)})^2 \left[P_N^{(1)} X^{N,(2)}_s \mbox{\textcircled{\scriptsize$<$}} {\mathcal Z}_s^{(2,N)} \right] \right. \\
&\quad \left. \hspace{3cm} - \left( P_N^{(1)} X^{N,(2)}_s \right) \mbox{\textcircled{\scriptsize$<$}} \left( e^{(t-s)(\triangle - m_0)}  (P_N^{(1)})^2 {\mathcal Z}_s^{(2,N)} \right) \right\| _{B_{p}^{1+2\varepsilon }} \\
&\quad + \lambda \left\| e^{(t-s)(\triangle - m_0)}  (P_N^{(1)})^2 \left[ \left( P_N {\mathcal Z}^{(0,3,N)}_s\right) \mbox{\textcircled{\scriptsize$<$}} {\mathcal Z}_s^{(2,N)} \right] \right. \\
&\quad \hspace{4cm} \left. - \left( P_N {\mathcal Z}^{(0,3,N)}_s \right) \mbox{\textcircled{\scriptsize$<$}} \left( e^{(t-s)(\triangle - m_0)}  (P_N^{(1)})^2 {\mathcal Z}_s^{(2,N)} \right) \right\| _{B_{p}^{1+2\varepsilon }} \\
&\quad + \left\| \left( P_N^{(1)} X^{N,(2)}_t - P_N^{(1)} X^{N,(2)}_s - \lambda P_N^{(1)}{\mathcal Z}^{(0,3,N)}_t + \lambda P_N^{(1)}{\mathcal Z}^{(0,3,N)}_s \right) \right. \\
&\quad \hspace{7cm} \left. \mbox{\textcircled{\scriptsize$<$}} \left( e^{(t-s)(\triangle - m_0)}  (P_N^{(1)})^2 {\mathcal Z}_s^{(2,N)} \right) \right\| _{B_{p}^{1+2\varepsilon }} \\
&\leq Q (t-s)^{-17/32-2\varepsilon } \left\| P_N^{(1)} X^{N,(2)}_s \right\| _{B_{p}^{15/16} } + Q(t-s)^{-3/4-3\varepsilon } \\
&\quad + Q (t-s)^{-1-3\varepsilon /2} \left( \left\| P_N^{(1)} X^{N,(2)}_t - P_N^{(1)} X^{N,(2)}_s \right\| _{L^p} + \left\| {\mathcal Z}^{(0,3,N)}_t - {\mathcal Z}^{(0,3,N)}_s \right\| _{L^p} \right).
\end{align*}
Hence, we have
\begin{align*}
& \left\| \int _0^t e^{(t-s)(\triangle - m_0)}  (P_N^{(1)})^2 \left[ \left( P_N^{(1)} X^{N,(2)}_s - \lambda P_N^{(1)} {\mathcal Z}^{(0,3,N)}_s \right) \mbox{\textcircled{\scriptsize$<$}} {\mathcal Z}_s^{(2,N)} \right] ds \right. \\
&\quad \hspace{2cm} \left. - \left( P_N^{(1)} X^{N,(2)}_t - \lambda P_N^{(1)} {\mathcal Z}^{(0,3,N)}_t \right) \mbox{\textcircled{\scriptsize$<$}} \int _0^{t} e^{(t-s)(\triangle - m_0)}  (P_N^{(1)})^2 {\mathcal Z}_s^{(2,N)} ds \right\| _{B_{p}^{1+2\varepsilon }}\\
&\leq \int _0^t \left\| e^{(t-s)(\triangle - m_0)}  (P_N^{(1)})^2 \left[ \left( P_N^{(1)} X^{N,(2)}_s - \lambda P_N^{(1)} {\mathcal Z}^{(0,3,N)}_s \right) \mbox{\textcircled{\scriptsize$<$}} {\mathcal Z}_s^{(2,N)} \right] \phantom{\left[ \left( X^{N,(2)}_t - {\mathcal Z}^{(0,3,N)}_t \right) \right] } \right. \\
&\quad \hspace{2cm} \left. - \left( P_N^{(1)} X^{N,(2)}_t - \lambda P_N^{(1)} {\mathcal Z}^{(0,3,N)}_t \right) \mbox{\textcircled{\scriptsize$<$}} \left( e^{(t-s)(\triangle - m_0)}  (P_N^{(1)})^2 {\mathcal Z}_s^{(2,N)} \right) \right\| _{B_{p}^{1+2\varepsilon }} ds \\
&\leq Q \int _0^t (t-s)^{-17/32-2\varepsilon } \left\| P_N^{(1)} X^{N,(2)}_s \right\| _{B_{p}^{15/16}} ds + Q\\
&\quad + Q \int _0^t s^{-\eta \theta}(t-s)^{\theta \gamma -1-3\varepsilon /2} \left( \left\| P_N^{(1)} X^{N,(2)}_t \right\| _{L^p} + \left\| P_N^{(1)} X^{N,(2)}_s \right\| _{L^p} \right) ^{1-\theta} \\
&\quad \hspace{6cm} \times \left( \frac{s^\eta \left\| P_N^{(1)} X^{N,(2)}_t - P_N^{(1)} X^{N,(2)}_s \right\| _{L^p} }{(t-s)^{\gamma}} \right)^{\theta} ds\\
&\quad + Q \int _0^t (t-s)^{-(1-\gamma )-3\varepsilon /2}\frac{\left\| {\mathcal Z}^{(0,3,N)}_t - {\mathcal Z}^{(0,3,N)}_s \right\| _{L^p}}{(t-s)^{\gamma }} ds\\
&\leq Q \int _0^t (t-s)^{-21/32} \left\| P_N^{(1)} X^{N,(2)}_s \right\| _{B_{p}^{15/16}} ds + Q\\
&\quad + Q \left( \sup _{s\in [0,t]} \frac{s^\eta \left\| P_N^{(1)} X^{N,(2)}_t - P_N^{(1)} X^{N,(2)}_s \right\| _{L^p} }{(t-s)^{\gamma }} \right)^{\theta}\\
&\quad \hspace{1cm} \times \left( C \| P_N^{(1)} X^{N,(2)}_t \| _{L^p}^{1-\theta} + \int _0^t s^{-\eta \theta}(t-s)^{\theta \gamma - 1 -3\varepsilon /2} \left\| P_N^{(1)} X^{N,(2)}_s \right\| _{L^p} ^{1-\theta} ds \right) .
\end{align*}
This proves the assertion in Lemma \ref{lem:com1}.
\end{proof}

\begin{lem}\label{lem:estcom}
For $\gamma \in (0,1/8)$, $\varepsilon \in (0,\gamma /2)$, $p\in [1,2]$, $t\in [0,T]$, $\theta \in (0,1]$ and $\delta \in (0,1]$
\begin{align*}
&\left\| \Psi _t^{(1)} (P_N^{(1)} X^{N,(2)}) \mbox{\textcircled{\scriptsize$=$}} {\mathcal Z}^{(2,N)}_t \right\| _{B_p^{\varepsilon}} \\
&\leq Q \int _0^t (t-s)^{-21/32} \left\| P_N^{(1)} X^{N,(2)}_s \right\| _{B_{p}^{15/16}} ds \\
&\quad + Q \left( \sup _{s\in [0,t]} \frac{s^\eta \left\| P_N^{(1)} X^{N,(2)}_t - P_N^{(1)} X^{N,(2)}_s \right\| _{L^p} }{(t-s)^{\gamma }} \right)^{\theta}\\
&\quad \hspace{0.5cm} \times \left(\| P_N^{(1)} X^{N,(2)}_t \| _{L^p}^{1-\theta} + \int _0^t s^{-\eta \theta} (t-s)^{\theta \gamma - 1 -3\varepsilon /2} \left\| P_N^{(1)} X^{N,(2)}_s \right\| _{L^p} ^{1-\theta} ds \right) + Q,\\
&\left\| \Psi _t^{(2)} (P_N^{(1)} X^{N,(2)}) \right\| _{B_{p}^{\varepsilon}} \leq \delta \left( \| P_N^{(1)} X^{N,(2)}_t\| _{L^4}^4 + \| X^{N,(2)}_t\| _{B_2^{15/16}}^2 \right) ^{7/8}  + (1+t^{-\varepsilon })Q\delta ^{-16/19}
\end{align*}
almost surely.
\end{lem}

\begin{proof}
By Proposition \ref{prop:paraproduct} we have, for a positive constant $C$,
\[
\left\| \Psi _t^{(1)} (P_N^{(1)} X^{N,(2)}) \mbox{\textcircled{\scriptsize$=$}} {\mathcal Z}^{(2,N)}_t \right\| _{B_{p}^{\varepsilon}} \leq C\left\| {\mathcal Z}^{(2,N)}_t \right\| _{B_{\infty}^{-1-\varepsilon}} \left\| \Psi _t^{(1)} (P_N^{(1)} X^{N,(2)})  \right\| _{B_{p}^{1+2\varepsilon}} .
\]
Hence, by applying Lemma \ref{lem:com1} we have 
\begin{align*}
&\left\| \Psi _t^{(1)} (P_N^{(1)} X^{N,(2)}) \mbox{\textcircled{\scriptsize$=$}} {\mathcal Z}^{(2,N)}_t \right\| _{B_p^{\varepsilon}} \\
&\leq Q \int _0^t (t-s)^{-21/32} \left\| P_N^{(1)} X^{N,(2)}_s \right\| _{B_{p}^{15/16}} ds \\
&\quad + C \left( \sup _{s\in [0,t]} \left\| {\mathcal Z}_s^{(2,N)} \right\| _{B_{\infty}^{-1-\varepsilon }}\right) ^2 \left( \sup _{s\in [0,t]} \frac{s^\eta \left\| P_N^{(1)} X^{N,(2)}_t - P_N^{(1)} X^{N,(2)}_s \right\| _{L^p} }{(t-s)^{\gamma }} \right)^{\theta}\\
&\quad \hspace{1cm} \times \left(\| P_N^{(1)} X^{N,(2)}_t \| _{L^p}^{1-\theta} + \int _0^t s^{- \eta \theta} (t-s)^{\theta \gamma - 1 -3\varepsilon /2} \left\| P_N^{(1)} X^{N,(2)}_s \right\| _{L^p} ^{1-\theta} ds \right) + Q
\end{align*}
almost surely.
Thus, the first estimate is proven.

By Proposition \ref{prop:com2} we have
\begin{align*}
&\left\| \Psi _t^{(2)} (P_N^{(1)} X^{N,(2)})  \right\| _{B_{p}^{\varepsilon}} \\
&\leq C \left\| {\mathcal Z}_t^{(2,N)} \right\| _{B_{\infty}^{-1-\varepsilon}} \left\| \int _0^t e^{(t-s)(\triangle -m_0^2)} \left( P_N^{(1)}\right) ^2 {\mathcal Z}_s^{(2,N)} ds \right\| _{B_{\infty}^{1-\varepsilon}} \\
&\quad \hspace{5cm} \times \left\| P_N^{(1)} X^{N,(2)}_t - \lambda P_N^{(1)} {\mathcal Z}^{(0,3,N)}_t \right\| _{B_p^{3\varepsilon }} \\
&\leq Q \left\| {\mathcal Z}_t^{(0,2,N)} - {\mathcal Z}_0^{(0,2,N)} \right\| _{B_{\infty}^{1-\varepsilon}} \left\| P_N^{(1)} X^{N,(2)}_t - \lambda P_N^{(1)} {\mathcal Z}^{(0,3,N)}_t \right\| _{B_p^{3\varepsilon }}.
\end{align*}
Hence, by Proposition \ref{prop:paraproduct} and Lemma \ref{lem:Lpestimates1} we also have proven the second estimate.
\end{proof}

\begin{lem}\label{lem:Phi1}
For $\varepsilon \in (0,1/16)$, $t\in [0,T]$ and $\delta \in (0,1]$, the following bound holds almost surely:
\[
\left\| \Phi _t ^{(1)} (P_N^{(1)} X^{N,(2)}) \right\| _{L^{4/3}} \leq \delta \left( \| P_N^{(1)} X^{N,(2)}_t\| _{L^4}^4 + \| P_N^{(1)} X^{N,(2)}_t\| _{B_2^{15/16}}^2 \right) ^{7/8} + \delta ^{-26/9}Q .
\]
\end{lem}

\begin{proof}
Proposition \ref{prop:paraproduct} implies
\begin{align*}
&\left\| \Phi _t ^{(1)} (P_N^{(1)} X^{N,(2)}) \right\| _{L^{4/3}} \\
&\leq C \left( \left\| {\mathcal Z}_t^{(1,N)} \right\|  _{B_{\infty}^{-1/2-\varepsilon /2}} + \left\| {\mathcal Z}^{(0,3,N)}_t \right\| _{B_{\infty}^{-1/2-\varepsilon/2}} \right) \left\| (P_N^{(1)} X^{N,(2)}_t ) ^2\right\| _{B_{4/3}^{1/2+\varepsilon}} \\
&\quad + C \left( \left\| {\mathcal Z}_t^{(1,N)} {\mathcal Z}^{(0,3,N)}_t \right\| _{B_{\infty}^{-1/2-\varepsilon/2}} + \left\| {\mathcal Z}^{(0,3,N)}_t\right\| _{L^{\infty }} ^2 \right) \left\| P_N^{(1)} X^{N,(2)}_t \right\| _{B_{4/3}^{1/2+\varepsilon}} .
\end{align*}
Hence, by Lemma \ref{lem:Lpestimates1} we have the assertion.
\end{proof}

\begin{lem}\label{lem:Phi1-2}
For $t\in [0,T]$ and $\delta \in (0,1]$, the following bound holds almost surely:
\[
\left| \int _\Lambda (P_N^{(1)} X_t^{N,(2)}) \Phi _t ^{(1)} (P_N^{(1)} X^{N,(2)}) dx \right| \leq \delta \left( \| P_N^{(1)} X^{N,(2)}_t\| _{L^4}^4 + \| P_N^{(1)} X^{N,(2)}_t\| _{B_2^{15/16}}^2 \right) + C_\delta Q .
\]
\end{lem}

\begin{proof}
H\"older's inequality implies
\[
\left\| (P_N^{(1)} X_t^{N,(2)}) \Phi _t ^{(1)} (P_N^{(1)} X^{N,(2)}) \right\| _{L^1} \leq \left\| P_N^{(1)} X_t^{N,(2)}\right\| _{L^4} \left\| \Phi _t ^{(1)} (P_N^{(1)} X^{N,(2)}) \right\| _{L^{4/3}}
\]
Applying Lemma \ref{lem:Phi1} with replacing $\delta$ by 
\[
\min \left\{ \left\| P_N^{(1)} X_t^{N,(2)}\right\| _{L^4}^{-1}, 1\right\} ,
\]
we obtain
\begin{align*}
& \left\| (P_N^{(1)} X_t^{N,(2)}) \Phi _t ^{(1)} (P_N^{(1)} X^{N,(2)}) \right\| _{L^1} \\
& \leq \left( \| P_N^{(1)}  X^{N,(2)}_t\| _{L^4}^4 + \|P_N^{(1)}  X^{N,(2)}_t\| _{B_2^{15/16}}^2 \right) ^{7/8} + Q \left\|P_N^{(1)}  X_t^{N,(2)}\right\| _{L^4}^{26/9}
\end{align*}
almost surely.
This inequality and (\ref{eq:lemLpestimates1-1}) imply the assertion.
\end{proof}

\begin{lem}\label{lem:Phi2}
For $p\in [1,2]$, $\varepsilon \in (0,1/16)$, $t\in ( 0,T]$ and $\delta \in (0,1]$, the following bounds hold almost surely:
\begin{align*}
&\left\| \Phi _t ^{(2)} (P_N^{(1)} X^{N,(2)}) \right\| _{B_{p}^{-1/2-\varepsilon}} \\
&\leq \delta \left( \| P_N^{(1)} X^{N,(2)}_t\| _{L^4}^4 + \| P_N^{(1)} X^{N,(2)}_t\| _{B_2^{15/16}}^2 \right) ^{7/8} + \delta ^{-16/19} t^{-2\varepsilon} Q, \\
&\left| \int _\Lambda (P_N^{(1)} X_t^{N,(2)}) \Phi _t ^{(2)} (P_N^{(1)} X^{N,(2)}) dx \right| \\
&\leq \delta \left( \| P_N^{(1)} X^{N,(2)}_t\| _{L^4}^4 + \| P_N^{(1)} X^{N,(2)}_t\| _{B_2^{15/16}}^2 \right) ^{7/8}+ C_\delta t^{-12\varepsilon} Q,
\end{align*}
for a positive constant $C_\delta$ and a positive polynomial $Q$.
\end{lem}

\begin{proof}
By Proposition \ref{prop:paraproduct} we have
\begin{align*}
&\left\| \left( P_N^{(1)} X^{N,(2)}_t - \lambda P_N^{(1)}{\mathcal Z}^{(0,3,N)}_t \right) \mbox{\textcircled{\scriptsize$>$}} {\mathcal Z}^{(2,N)}_t \right\| _{B_{p}^{-(1+\varepsilon )/2}} \\
&\leq C \left\| {\mathcal Z}^{(2,N)}_t \right\| _{B_{\infty}^{-1-\varepsilon /4}} \left\| P_N^{(1)} X^{N,(2)}_t - \lambda P_N^{(1)} {\mathcal Z}^{(0,3,N)}_t \right\|  _{B_{p}^{1/2-\varepsilon /4}} ,\\
%
%
&\left\| \left( P_N^{(1)} X^{N,(2)}_t - \lambda P_N^{(1)} {\mathcal Z}^{(0,3,N)}_t\right) \left( {\mathcal Z}^{(2,2,N)}_t - {\mathcal Z}_t^{(2,N)}\mbox{\textcircled{\scriptsize$=$}} \int _{-\infty}^0 e^{(t-s)(\triangle -m_0^2)} P_N {\mathcal Z}_s^{(2,N)} ds\right)   \right\| _{B_{p}^{-(1+\varepsilon )/2}} \\
& \leq C \left( \left\| {\mathcal Z}^{(2,2,N)}_t \right\| _{B_{\infty}^{-\varepsilon /2}} + t^{-\varepsilon} \left\| {\mathcal Z}^{(2,N)}_t \right\| _{B_{\infty}^{-1-\varepsilon /2}} \left\| {\mathcal Z}^{(0,2,N)}_0 \right\| _{B_{\infty}^{1-\varepsilon /2}} \right) \\
&\quad \hspace{7cm} \times \left\| P_N^{(1)} X^{N,(2)}_t - \lambda P_N^{(1)} {\mathcal Z}^{(0,3,N)}_t \right\| _{B_p^{1/2-\varepsilon }} .
\end{align*}
Hence, we have
\begin{equation}\label{eq:lemPhi2-1}
\left\| \Phi _t ^{(2)} (P_N^{(1)} X^{N,(2)}) \right\| _{B_{p}^{-(1+\varepsilon )/2}} \leq Q t^{-\varepsilon} \left\| P_N^{(1)} X^{N,(2)}_t \right\|  _{B_{p}^{1/2}} + t^{-\varepsilon} Q.
\end{equation}
This inequality and Lemma \ref{lem:Lpestimates1} imply the first inequality.

For the second inequality, by Proposition \ref{prop:paraproduct}, and (\ref{eq:lemPhi2-1}) we have
\begin{align*}
\left| \int _\Lambda (P_N^{(1)} X_t^{N,(2)}) \Phi _t ^{(2)} (P_N^{(1)} X^{N,(2)}) dx \right| 
&\leq t^{-\varepsilon} Q + t^{-\varepsilon} Q \left\| P_N^{(1)} X^{N,(2)}_t \right\|  _{B_{2}^{9/16}} ^2\\
&\leq t^{-\varepsilon} Q + t^{-\varepsilon} Q \left\| P_N^{(1)} X^{N,(2)}_t \right\|  _{L^2} ^{4/5} \left\| P_N^{(1)} X^{N,(2)}_t \right\|  _{B_{2}^{15/16}}^{6/5} .
\end{align*}
Hence, by (\ref{eq:lemLpestimates1-1}) we have
\begin{align*}
&\left| \int _\Lambda (P_N^{(1)} X_t^{N,(2)}) \Phi _t ^{(2)} (P_N^{(1)} X^{N,(2)}) dx \right| \\
&\leq t^{-\varepsilon} Q + Q t^{-\varepsilon} \left( \left\| P_N^{(1)} X^{N,(2)}_t \right\|  _{L^2} ^{16/5} + \left\| P_N^{(1)} X^{N,(2)}_t \right\| _{B_{2}^{15/16}}^{8/5} \right) \\
&\leq Q t^{-\varepsilon} \left( \| P_N^{(1)} X^{N,(2)}_t\| _{L^4}^4 + \| P_N^{(1)} X^{N,(2)}_t\| _{B_2^{15/16}}^2 \right) ^{4/5} + t^{-\varepsilon} Q .
\end{align*}
Thus, we have the second inequality.
\end{proof}

\begin{lem}\label{lem:Phi3}
For $p\in [1,2]$, $\varepsilon \in (0,1/16)$, $t\in [0,T]$ and $\delta \in (0,1]$, the following bounds hold almost surely:
\begin{align*}
\left\| \Phi _t ^{(3)} (P_N^{(1)} X^{N,(2)}) \right\| _{B_{p}^{-(1+\varepsilon)/2}} &\leq \delta \| P_N^{(1)} X^{N,(2)}_t\| _{L^4}^{7/2} + \delta ^{-4/3} Q, \\
\left| \int _\Lambda (P_N^{(1)} X_t^{N,(2)}) \Phi _t ^{(3)} (P_N^{(1)} X^{N,(2)}) dx \right| &\leq \delta \left( \| P_N^{(1)} X^{N,(2)}_t\| _{L^4}^4 + \| P_N^{(1)} X^{N,(2)}_t\| _{B_2^{15/16}}^2 \right) ^{7/8}+ C_\delta Q .
\end{align*}
\end{lem}

\begin{proof}
By Proposition \ref{prop:paraproduct} we have
\begin{equation}\label{eq:lemPhi3-1}\begin{array}{rl}
\left\| \Phi _t ^{(3)} (P_N^{(1)} X^{N,(2)}) \right\| _{B_{p}^{-(1+\varepsilon)/2}} &\leq Q\left( \left\| \left(P_N^{(1)} X_t^{N,(2)} \right) ^2 \right\| _{L^p} + \| P_N^{(1)} X_t^{N,(2)}\| _{L^p}\right) \\
&\leq Q \| P_N^{(1)} X_t^{N,(2)}\| _{L^4}^2 + Q.
\end{array}\end{equation}
This and (\ref{eq:lemLpestimates1-1}) imply the first inequality.

By Proposition \ref{prop:paraproduct} and (\ref{eq:lemPhi3-1}) we have
\[
\left| \int _\Lambda (P_N^{(1)} X_t^{N,(2)}) \Phi _t ^{(3)} (P_N^{(1)} X^{N,(2)}) dx \right| \leq Q \| P_N^{(1)} X_t^{N,(2)}\| _{L^4}^2 \| P_N^{(1)} X_t^{N,(2)}\| _{B_2^{1/2+\varepsilon}} .
\]
Applying Lemma \ref{lem:Lpestimates1} with replacing $\delta$ by $(1+ Q \| P_N^{(1)} X_t^{N,(2)}\| _{L^4}^2)^{-1} \delta$, we obtain the second inequality.
\end{proof}

\begin{lem}\label{lem:=Z2}
For $\varepsilon \in (0,1/16)$, $p\in [1,2]$, $t\in [0,T]$ and $\delta \in (0,1]$,
\begin{align*}
&\left\| (P_N^{(1)} X^{N,(2),\geqslant}_t) \mbox{\textcircled{\scriptsize$=$}} {\mathcal Z}^{(2,N)}_t \right\| _{B_{p}^{\varepsilon /8}} \\
&\leq \delta \left( \left\| \nabla X^{N,(2),\geqslant}_t \right\| _{L^2}^2 + \left\| P_N^{(1)} X^{N,(2)}_t \right\| _{L^4} ^4 \right) ^{7/8} + \delta \left\| X^{N,(2),<}_t \right\| _{L^p}^{7/4} + \delta \left\| X^{N,(2),\geqslant}_t \right\| _{B_{p}^{1+\varepsilon }}^{5/6} + \delta ^{-82/23}Q.
\end{align*}
\end{lem}

\begin{proof}
By Proposition \ref{prop:paraproduct}\ref{prop:paraproduct4} we have
\begin{equation}\label{eq:lem=Z2-1}
\left\| (P_N^{(1)} X^{N,(2),\geqslant}_t) \mbox{\textcircled{\scriptsize$=$}} {\mathcal Z}^{(2,N)}_t \right\| _{B_{p}^{\varepsilon /8}} \leq C\left\| {\mathcal Z}^{(2,N)}_t \right\| _{B_{\infty}^{-1-\varepsilon /8}} \left\| P_N^{(1)} X^{N,(2),\geqslant}_t \right\| _{B_{p}^{1+\varepsilon /4}} .
\end{equation}
Proposition \ref{prop:paraproduct}\ref{prop:paraproduct6} and (\ref{eq:lemLpestimates1-1}) imply
\begin{align*}
\left\| P_N^{(1)} X^{N,(2),\geqslant}_t \right\| _{B_{p}^{1+\varepsilon /4}}
&\leq C \left\| P_N^{(1)} X^{N,(2),\geqslant}_t \right\| _{B_{p}^{1-\varepsilon /8}}^{2/3} \left\| P_N^{(1)} X^{N,(2),\geqslant}_t \right\| _{B_{p}^{1+\varepsilon }}^{1/3} \\
&\leq \delta \left\| P_N^{(1)}X^{N,(2),\geqslant}_t \right\| _{B_{p}^{1-\varepsilon /8}}^{7/4} + C \delta ^{-8/13} \left\| P_N^{(1)} X^{N,(2),\geqslant}_t \right\| _{B_{p}^{1+\varepsilon }}^{7/13} .
\end{align*}
Since
\begin{align*}
&\left\| P_N^{(1)} X^{N,(2),\geqslant}_t \right\| _{B_{p}^{1-\varepsilon /8}}^{7/4}
\leq C \left\| P_N^{(1)} X^{N,(2),\geqslant}_t \right\| _{W^{1-\varepsilon /8,p}}^{7/4}\\
&\leq C \left(  \left\| P_N^{(1)} X^{N,(2)}_t \right\| _{L^p} + \left\| \nabla P_N^{(1)} X^{N,(2),\geqslant}_t \right\| _{L^p} \right) ^{7/4} + C \left\| P_N^{(1)} X^{N,(2),<}_t \right\| _{L^p}^{7/4} +C,
\end{align*}
from (\ref{eq:lem=Z2-1}) we have
\begin{align*}
&\left\| (P_N^{(1)} X^{N,(2),\geqslant}_t) \mbox{\textcircled{\scriptsize$=$}} {\mathcal Z}^{(2,N)}_t \right\| _{B_{p}^{\varepsilon /8}} \\
&\leq C \delta \left\| {\mathcal Z}^{(2,N)}_t \right\| _{B_{\infty}^{-1-\varepsilon /8}} \left[ \left(  \left\| P_N^{(1)} X^{N,(2)}_t \right\| _{L^p} + \left\| \nabla P_N^{(1)} X^{N,(2),\geqslant}_t \right\| _{L^p} \right) ^{7/4} + \left\| P_N^{(1)} X^{N,(2),<}_t \right\| _{L^p}^{7/4} \right] \\
&\quad + C \delta ^{-8/13} \left\| {\mathcal Z}^{(2,N)}_t \right\| _{B_{\infty}^{-1-\varepsilon /8}} \left\| P_N^{(1)} X^{N,(2),\geqslant}_t \right\| _{B_{p}^{1+\varepsilon }}^{7/13} + Q.
\end{align*}
On the other hand, (\ref{eq:lemLpestimates1-1}) implies
\[
\delta ^{-8/13} \left\| P_N^{(1)} X^{N,(2),\geqslant}_t \right\| _{B_{p}^{1+\varepsilon }}^{7/13} \leq \delta \left\| P_N^{(1)} X^{N,(2),\geqslant}_t \right\| _{B_{p}^{1+\varepsilon }}^{5/6} + \delta ^{-82/23}.
\]
Hence, by replacing $\delta$ by
\[
\delta \min \left\{ C\left\| {\mathcal Z}^{(2,N)}_t \right\| _{B_{\infty}^{-1-\varepsilon /8}} ^{-1} ,1 \right\}
\]
and applying Proposition \ref{prop:bddp} we obtain the assertion.
\end{proof}

\begin{lem}\label{lem:vw}
For $t\in [0,T]$ and $\delta \in (0,1]$ it holds that, for some positive $C$ and $Q$,
\begin{align*}
&\left| \int _\Lambda (P_N^{(1)} X^{N,(2),\geqslant}_t) \left[ ( P_N^{(1)} X^{N,(2)}_t - \lambda P_N^{(1)} {\mathcal Z}^{(0,3,N)}_t) \mbox{\textcircled{\scriptsize$<$}} {\mathcal Z}^{(2,N)}_t \right] dx \right| \\
&\leq \delta \left( \| \nabla X_t^{N,(2),\geqslant}\| _{L^2}^2 + \| P_N^{(1)} X_t^{N,(2)}\| _{L^4}^4 \right) + \delta \| X_t^{N,(2),\geqslant}\| _{B_{4/3}^{1+\varepsilon}} + C \| X_t^{N,(2),<}\| _{L^{4/3}} ^{5/3} + C_ \delta Q
\end{align*}
almost surely.
\end{lem}

\begin{proof}
By Proposition \ref{prop:paraproduct} and (\ref{eq:lemLpestimates1-1}) we have
\begin{align*}
&\left| \int _\Lambda (P_N^{(1)} X^{N,(2),\geqslant}_t) \left[ ( P_N^{(1)} X^{N,(2)}_t - \lambda P_N^{(1)} {\mathcal Z}^{(0,3,N)}_t) \mbox{\textcircled{\scriptsize$<$}} {\mathcal Z}^{(2,N)}_t \right] dx \right| \\
&\leq C \left\| P_N^{(1)} X^{N,(2),\geqslant}_t \right\| _{B_{4/3}^{1+\varepsilon /8}} \left\| (P_N^{(1)} X^{N,(2)}_t - \lambda P_N^{(1)} {\mathcal Z}^{(0,3,N)}_t) \mbox{\textcircled{\scriptsize$<$}} {\mathcal Z}^{(2,N)}_t \right\| _{B_4^{-1-\varepsilon /12}} \\
&\leq C \left\| P_N^{(1)} X^{N,(2),\geqslant}_t \right\| _{B_{4/3}^{1+\varepsilon /8}}^{3/2} + \left\| (P_N^{(1)} X^{N,(2)}_t - \lambda P_N^{(1)} {\mathcal Z}^{(0,3,N)}_t) \mbox{\textcircled{\scriptsize$<$}} {\mathcal Z}^{(2,N)}_t \right\| _{B_4^{-1-\varepsilon /12}}^3 \\
&\leq C \left\| P_N^{(1)} X^{N,(2),\geqslant}_t \right\| _{B_{4/3}^{1+\varepsilon /8}}^{3/2} + \left( \sup _{t\in [0,T]}\| {\mathcal Z}^{(2,N)}_t\| _{B_\infty ^{-1-\varepsilon /12}}\right) ^3 \left\| P_N^{(1)} X^{N,(2)}_t \right\| _{L^4}^3 + Q.
\end{align*}
Hence, we have for $\delta \in (0,1]$
\begin{equation}\label{eq:lemvw1}\begin{array}{l}
\displaystyle \left| \int _\Lambda (P_N^{(1)} X^{N,(2),\geqslant}_t) \left[ ( P_N^{(1)} X^{N,(2)}_t - \lambda P_N^{(1)} {\mathcal Z}^{(0,3,N)}_t) \mbox{\textcircled{\scriptsize$<$}} {\mathcal Z}^{(2,N)}_t \right] dx \right| \\
\displaystyle \leq \delta \left\| P_N^{(1)} X^{N,(2)}_t \right\| _{L^4}^4 + C \left\| P_N^{(1)} X^{N,(2),\geqslant}_t \right\| _{B_{4/3}^{1+\varepsilon /8}}^{3/2} + C_ \delta Q.
\end{array}\end{equation}
On the other hand, by Proposition \ref{prop:paraproduct} and (\ref{eq:lemLpestimates1-1}) we have
\begin{align*}
\| P_N^{(1)} X_t^{N,(2),\geqslant}\| _{B_{4/3}^{1+\varepsilon /8}}^{3/2} &\leq C \| P_N^{(1)} X_t^{N,(2),\geqslant}\| _{B_{4/3}^{1-\varepsilon /20}}^{5/4} \| P_N^{(1)} X_t^{N,(2),\geqslant}\| _{B_{4/3}^{1+\varepsilon }}^{1/4} \\
&\leq \delta \| P_N^{(1)} X_t^{N,(2),\geqslant}\| _{B_{4/3}^{1+\varepsilon }} + C \delta ^{-1/3} \| P_N^{(1)} X_t^{N,(2),\geqslant}\| _{B_{4/3}^{1-\varepsilon /20}}^{5/3} .
\end{align*}
Since
\begin{align*}
&\| P_N^{(1)} X_t^{N,(2),\geqslant}\| _{B_{4/3}^{1-\varepsilon /20}}^{5/3} \leq C \| P_N^{(1)} X_t^{N,(2),\geqslant}\| _{W^{1,4/3}}^{5/3} \\
&\quad \leq \delta ^{4/3}\left( \| \nabla P_N^{(1)} X_t^{N,(2),\geqslant}\| _{L^{4/3}}^2 + \| P_N^{(1)} X_t^{N,(2)}\| _{L^{4/3}}^2 \right) + C \| P_N^{(1)} X_t^{N,(2),<}\| _{L^{4/3}} ^{5/3} + C \delta ^{-\alpha }
\end{align*}
for some $\alpha \in (0,\infty )$, by Proposition \ref{prop:bddp} we have
\begin{align*}
&\| P_N^{(1)} X_t^{N,(2),\geqslant}\| _{B_{4/3}^{1+\varepsilon /8}}^{3/2} \\
&\leq \delta \| X_t^{N,(2),\geqslant}\| _{B_{4/3}^{1+\varepsilon }} + \delta \left( \| \nabla X_t^{N,(2),\geqslant}\| _{L^{4/3}}^2 + \| P_N^{(1)} X_t^{N,(2)}\| _{L^{4/3}}^2 \right) + C \| X_t^{N,(2),<}\| _{L^{4/3}} ^{5/3} + C \delta ^{-\alpha } .
\end{align*}
This inequality and (\ref{eq:lemvw1}) yield the assertion.
\end{proof}

Now we prepare a pathwise estimate of the energy functional on the left-hand side.

\begin{prop}\label{prop:global2-2}
For $\gamma \in (0,1/8)$, $\varepsilon \in (0,\gamma /2)$, $\eta \in [0,1)$ and $t\in [0,T]$
\begin{align*}
&\| X_t^{N,(2)}\| _{L^2} ^2 - \| X_0^{N,(2)}\| _{L^2} ^2 + \int _0^t \left( \left\| \nabla X_s^{N,(2),\geqslant}\right\| _{L^2}^2 + \left\| X_s^{N,(2)}\right\| _{L^2}^2 + \lambda \left\| P_N^{(1)} X_s^{N,(2)}\right\| _{L^4}^4 \right) ds\\
&\leq \| X_t^{N,(2),<}\| _{L^2}^2 + {\mathfrak Y}_{\varepsilon}^N (t) + Q\left( \sup _{s',t' \in [0,t]; s'<t'} \frac{(s')^\eta \left\| P_N^{(1)} X^{N,(2)}_{t'} - P_N^{(1)} X^{N,(2)}_{s'} \right\| _{L^{4/3}}}{(t'-s')^{\gamma }} \right) ^{4/5} +Q.
\end{align*}
\end{prop}

\begin{proof}
Proposition \ref{prop:paraproduct} implies that for $\delta \in (0,1]$
\begin{align*}
\left| \int _{\Lambda} f (\triangle -m_0^2) g dx \right| & \leq \left\| f \right\| _{B_{4}^{1-\varepsilon /12}} \left\| g\right\| _{B_{4/3}^{1+\varepsilon /6}}\\
&\leq C \delta ^{-2} \left\| f \right\| _{B_{4}^{1-\varepsilon /12}}^3 + \delta \left\| g\right\| _{B_{4/3}^{1+\varepsilon /6}}^{3/2} .
\end{align*}
Hence, by the integration by parts formula, the facts that $X_0^{N,(2),<} =0$ and that $\int _{\Lambda} f (P_N^{(1)}g) dx = \int _{\Lambda} (P_N^{(1)}f) g dx$ for $f,g \in L^2$, and Proposition \ref{prop:bddp}, we have
\begin{align*}
&\| X_t^{N,(2)}\| _{L^2} ^2 - \| X_0^{N,(2)}\| _{L^2} ^2\\
&= 2 \int _0^t \int _{\Lambda} \left( X_s^{N,(2),<} \partial _s X_s^{N,(2),<} + X_s^{N,(2), \geqslant} \partial _s X_s^{N,(2),<} + X_s^{N,(2)} \partial _s X_s^{N,(2), \geqslant} \right) dx ds\\
&= \| X_t^{N,(2),<}\| _{L^2}^2 + 4 \int _0^t \int _\Lambda X_s^{N,(2),<} (\triangle -m_0^2) X_s^{N,(2), \geqslant} dx ds \\
&\quad - 6 \lambda \int _0^t \int _{\Lambda} X_s^{N,(2), \geqslant} P_N^{(1)} \left[ \left( P_N^{(1)} X_s^{N,(2)} - \lambda P_N^{(1)} {\mathcal Z}^{(0,3,N)}_s \right) \mbox{\textcircled{\scriptsize$<$}} {\mathcal Z}^{(2,N)}_s \right] dx ds \\
&\quad - 2 \int _0^t \int _{\Lambda} \left( |\nabla X_s^{N,(2),\geqslant}|^2 +m_0^2 |X_s^{N,(2),\geqslant}|^2 + \lambda |P_N^{(1)} X_s^{N,(2)}|^4 \right) dx ds \\
&\quad - 6 \lambda \int _0^t  \int _{\Lambda} X_s^{N,(2)} P_N^{(1)} \left[ \left( {\mathcal Z}_s^{(1,N)} - \lambda P_N^{(1)} {\mathcal Z}^{(0,3,N)}_s\right) \mbox{\textcircled{\scriptsize$>$}} \left[ \left( P_N^{(1)} X^{N,(2)}_s \right) ^2 \right] \right] dx ds \\
&\quad + 2 \lambda \int _0^t  \int _{\Lambda} X_s^{N,(2)} P_N^{(1)} \Phi _s^{(1)} (P_N^{(1)} X^{N,(2)}) dx ds +2 \lambda \int _0^t \int _{\Lambda} X_s^{N,(2)} P_N^{(1)} \Phi ^{(2)}_s (P_N^{(1)} X^{N,(2)}) dx ds\\
&\quad + 6 \lambda \int _0^t \int _{\Lambda} X_s^{N,(2)} P_N^{(1)} \left[ (P_N^{(1)} X_s^{N,(2), \geqslant}) \mbox{\textcircled{\scriptsize$=$}} {\mathcal Z}^{(2,N)}_s \right] dx ds \\
&\quad -18 \lambda ^2 \int _0^t \int _{\Lambda} X_s^{N,(2)} P_N^{(1)} \left[ \Psi _{s}^{(1)} (P_N^{(1)} X^{N,(2)}) \mbox{\textcircled{\scriptsize$=$}} {\mathcal Z}^{(2,N)}_s \right] dx ds\\
&\quad -18 \lambda ^2 \int _0^t \int _{\Lambda} X_s^{N,(2)} P_N^{(1)} \Psi _s^{(2)} (P_N^{(1)} X^{N,(2)}) dx ds\\
&\leq - 2 \int _0^t \left( \left\| \nabla X_s^{N,(2), \geqslant}\right\| _{L^2}^2 + m_0^2 \left\| X_s^{N,(2)} \right\| _{L^2}^2 + \lambda \left\| P_N^{(1)} X_s^{N,(2)}\right\| _{L^4}^4 \right) ds \\
&\quad + \| X_t^{N,(2),<}\| _{L^2}^2 + C \delta ^{-2} \int _0^t \left\| X_s^{N,(2),<} \right\| _{B_{4}^{1-\varepsilon /12}}^3 ds + \delta \int _0^t \left\| X_s^{N,(2), \geqslant} \right\| _{B_{4/3}^{1+\varepsilon /6}}^{3/2} ds \\
&\quad + C \int _0^t \left| \int _\Lambda \left( P_N^{(1)} X_s^{N,(2), \geqslant} \right) \left[ \left( P_N^{(1)} X_s^{N,(2)} - \lambda P_N^{(1)} {\mathcal Z}^{(0,3,N)}_s \right) \mbox{\textcircled{\scriptsize$<$}} {\mathcal Z}^{(2,N)}_s \right] dx \right| ds \\
&\quad + C \lambda \int _0^t \left| \int _\Lambda (P_N^{(1)} X_s^{N,(2)} ) \Phi _s^{(1)} (P_N^{(1)} X^{N,(2)} ) dx \right| ds \\
&\quad + C \lambda \int _0^t \left| \int _\Lambda (P_N^{(1)} X_s^{N,(2)}) \Phi ^{(2)}_s (P_N^{(1)} X^{N,(2)}) dx \right|   ds \\
&\quad + C \lambda \int _0^t \left| \int _\Lambda (P_N^{(1)} X_s^{N,(2)}) \Phi ^{(3)}_s (P_N^{(1)} X^{N,(2)}) dx \right|   ds \\
&\quad + C \lambda \int _0^t \left\| P_N^{(1)} X_s^{N,(2)}\right\| _{L^4} \left\| (P_N^{(1)} X_s^{N,(2), \geqslant}) \mbox{\textcircled{\scriptsize$=$}} {\mathcal Z}^{(2,N)}_s \right\| _{L^{4/3}} ds \\
&\quad + C \lambda \int _0^t \left\| P_N^{(1)} X_s^{N,(2)} \right\| _{L^4} \left\| \Psi_s^{(1)} (P_N^{(1)} X^{N,(2)}) \mbox{\textcircled{\scriptsize$=$}} {\mathcal Z}^{(2,N)}_s \right\| _{L^{4/3}} ds \\
&\quad + C \lambda \int _0^t \left\| P_N^{(1)} X_s^{N,(2)} \right\| _{L^4} \left\| \Psi _s^{(2)} (P_N^{(1)} X^{N,(2)}) \right\| _{L^{4/3}} du .
\end{align*}
By a similar way to the proof of Lemma \ref{lem:vw} it is obtained that
\begin{align*}
\| X_t^{N,(2),\geqslant}\| _{B_{4/3}^{1+\varepsilon /6}}^{3/2} & \leq \delta \| X_t^{N,(2),\geqslant}\| _{B_{4/3}^{1+\varepsilon }} + \delta \left( \| \nabla X_t^{N,(2),\geqslant}\| _{L^{4/3}}^2 + \| X_t^{N,(2)}\| _{L^{4/3}}^2 \right) \\
& \quad + C \| X_t^{N,(2),<}\| _{L^{4/3}} ^{5/3} + C \delta ^{-\alpha }
\end{align*}
for some $\alpha \in (0,\infty )$, and by (\ref{eq:lemLpestimates1-1}) and Lemma \ref{lem:timeint} it holds that
\begin{align*}
&\int _0^t \left( \left\| P_N^{(1)} X_s^{N,(2)}\right\| _{L^4} \int _0^s (s-u)^{-21/32} \left\| P_N^{(1)} X^{N,(2)}_u \right\| _{B_{p}^{15/16}} du \right) ds \\
&\leq \delta \int _0^t \left\| P_N^{(1)} X_s^{N,(2)}\right\| _{L^4}^4 ds + C_\delta \int _0^t \left( \int _0^s (s-u)^{-21/32} \left\| P_N^{(1)} X^{N,(2)}_u \right\| _{B_{p}^{15/16}} du \right) ^{4/3} ds \\
&\leq \delta \int _0^t \left\| P_N^{(1)} X_s^{N,(2)}\right\| _{L^4}^4 ds + C_\delta \int _0^t \| P_N^{(1)} X_s^{N,(2)}\| _{B_2^{15/16}}^{4/3} ds.
\end{align*}
Thus, applying Lemmas \ref{lem:v3}, \ref{lem:Phi1-2}, \ref{lem:Phi2}, \ref{lem:Phi3} and \ref{lem:vw}, and Lemmas \ref{lem:estcom} and \ref{lem:=Z2} with replacing $\delta$ by $\delta \min \left\{ \left\| P_N^{(1)} X_s^{N,(2)} \right\| _{L^4} ^{-1},1\right\}$, we have
\begin{align*}
&\| X_t^{N,(2)}\| _{L^2} ^2 - \| X_0^{N,(2)}\| _{L^2} ^2\\
&\leq - 2 \int _0^t \left( \left\| \nabla X_s^{N,(2),\geqslant}\right\| _{L^2}^2 + m_0^2 \left\| X_s^{N,(2)} \right \| _{L^2}^2 + \lambda \left\| P_N^{(1)} X_s^{N,(2)}\right\| _{L^4}^4 \right) ds + \| X_t^{N,(2),<}\| _{L^2}^2\\
&\quad + C \delta \int _0^t \left( \left\| \nabla X_s^{N,(2),\geqslant}\right\| _{L^2}^2 + \left\| X_s^{N,(2)} \right \| _{L^2}^2 + \lambda \left\| P_N^{(1)} X_s^{N,(2)}\right\| _{L^4}^4 \right) ds + \delta {\mathfrak Y}_{\varepsilon}^N (t) \\
&\quad + \lambda Q\left( \sup _{s',t'\in [0,t]; s'<t'} \frac{(s')^\eta \left\| P_N^{(1)} X^{N,(2)}_{t'} - P_N^{(1)} X^{N,(2)}_{s'} \right\| _{L^{4/3}}}{(t'-s')^{\gamma }} \right) ^{1/2} \\
&\quad \hspace{1cm} \times \left( \int _0^t \left\| P_N^{(1)} X_u^{N,(2)} \right\| _{L^{4/3}} ^{1/2} \left\|P_N^{(1)} X_u^{N,(2)} \right\| _{L^{4}} du \phantom{\int _0^t} \right.\\
&\quad \hspace{2cm} \left . + \int _0^t \left\| P_N^{(1)} X_u^{N,(2)} \right\| _{L^{4}} \left( \int _0^u r^{-\eta /2} (u-r)^{\gamma /2 -1 -3\varepsilon /2} \left\| P_N^{(1)} X^{N,(2)}_r \right\| _{L^{4/3}} ^{1/2} dr \right) du \right) \\
&\quad + \delta ^{-\alpha} Q
\end{align*}
almost surely, for some $\alpha \in (0,\infty )$.
Since for $\tilde \delta \in (0,1]$
\begin{align*}
&\int _0^t \left\| P_N^{(1)} X_u^{N,(2)} \right\| _{L^{4/3}} ^{1/2} \left\| P_N^{(1)} X_u^{N,(2)} \right\| _{L^{4}} du \\
&\leq \int _0^t \left\| P_N^{(1)} X_u^{N,(2)} \right\| _{L^4} ^{3/2} du \leq \tilde \delta \int _0^t \left\| P_N^{(1)} X_u^{N,(2)} \right\| _{L^4} ^{4} du + C\tilde \delta ^{-3/5} ,
\end{align*}
and 
\begin{align*}
&\int _0^t \left\| P_N^{(1)} X_u^{N,(2)} \right\| _{L^4} \left( \int _0^u r^{-\eta /2} (u-r)^{\gamma /2 -1 -3\varepsilon /2} \left\| P_N^{(1)} X^{N,(2)}_r \right\| _{L^{4/3}} ^{1/2} dr \right) du\\
&\leq C \int _0^t \int _0^t r^{-\eta /2} |u-r|^{\gamma /2 -1 -3\varepsilon /2} \left\| P_N^{(1)} X^{N,(2)}_r \right\| _{L^4} ^{1/2} \left\| P_N^{(1)} X_u^{N,(2)} \right\| _{L^4} dr du \\
&\leq C \int _0^t \int _0^t r^{-\eta /2} |u-r|^{\gamma /2 -1 -3\varepsilon /2} \left( \tilde \delta ^{-1/5}\left\| P_N^{(1)} X^{N,(2)}_r \right\| _{L^4} + \tilde \delta ^{1/5}\left\| P_N^{(1)} X_u^{N,(2)} \right\| _{L^4}^2 \right) dr du \\
&\leq C \tilde \delta ^{-1/5} \int _0^t \left( \int _0^t |u-r|^{\gamma /2 -1 -3\varepsilon /2} du \right) r^{-\eta /2} \left\| P_N^{(1)} X^{N,(2)}_r \right\| _{L^4} dr \\
&\quad + C \tilde \delta ^{1/5} \int _0^t \left( \int _0^t r^{-\eta /2} |u-r|^{\gamma /2 -1 -3\varepsilon /2} dr \right) \left\| P_N^{(1)} X^{N,(2)}_u \right\| _{L^4}^2 du \\
&\leq C \left[ \tilde \delta ^{-1/5} \left( \int _0^t r^{-2 \eta /3} dr \right) ^{3/4} \left( \int _0^t \left\| P_N^{(1)} X^{N,(2)}_u \right\| _{L^4} ^4 du \right) ^{1/4} + \tilde \delta ^{1/5} \int _0^t\left\| P_N^{(1)} X^{N,(2)}_u \right\| _{L^4}^2 du\right] \\
&\leq \tilde \delta \int _0^t\left\| P_N^{(1)} X^{N,(2)}_u \right\| _{L^4}^4 du + C \tilde \delta ^{-3/5},
\end{align*}
by applying these inequalities with letting 
\[
\tilde \delta := \max\left\{ \delta \left( \sup _{s',t'\in [0,t]; s'<t'} \frac{(s')^\eta \left\| P_N^{(1)} X^{N,(2)}_{t'} - P_N^{(1)} X^{N,(2)}_{s'} \right\| _{L^{4/3}}}{(t'-s')^{\gamma }} \right) ^{-1/2} , \frac 12 \right\}
\]
we obtain
\begin{align*}
&\| X_t^{N,(2)}\| _{L^2} ^2 - \| X_0^{N,(2)}\| _{L^2} ^2\\
&\leq - 2 \int _0^t \left( \left\| \nabla X_u^{N,(2),\geqslant}\right\| _{L^2}^2 + m_0^2 \left\| X_u^{N,(2)} \right \| _{L^2}^2 + \lambda \left\| P_N^{(1)} X_u^{N,(2)}\right\| _{L^4}^4 \right) du + \| X_t^{N,(2),<}\| _{L^2}^2\\
&\quad + \delta Q \int _0^t \left( \left\| \nabla X_u^{N,(2),\geqslant}\right\| _{L^2}^2 + \left\| X_u^{N,(2)} \right \| _{L^2}^2 + \lambda \left\| P_N^{(1)} X_u^{N,(2)}\right\| _{L^4}^4 \right) du + \delta {\mathfrak Y}_{\varepsilon}^N (t)\\
&\quad + \delta ^{-1} Q \left( \sup _{s',t'\in [0,t]; s'<t'} \frac{(s')^\eta \left\| P_N^{(1)} X^{N,(2)}_{t'} - P_N^{(1)} X^{N,(2)}_{s'} \right\| _{L^{4/3}}}{(t'-s')^{\gamma }} \right) ^{4/5} + \delta ^{-\alpha} Q
\end{align*}
almost surely.
Now by taking sufficiently small $\delta $ so that $Q\delta \leq 1/2$, we obtain
\begin{align*}
&\| X_t^{N,(2)}\| _{L^2} ^2 - \| X_0^{N,(2)}\| _{L^2} ^2 + \int _0^t \left( \left\| \nabla X_u^{N,(2)}\right\| _{L^2}^2 + \left\| X_u^{N,(2)}\right\| _{L^2}^2 + \lambda \left\| P_N^{(1)} X_u^{N,(2)}\right\| _{L^4}^4 \right) du\\
&\leq \| X_t^{N,(2),<}\| _{L^2}^2 + {\mathfrak Y}_{\varepsilon}^N (t) + Q\left( \sup _{s',t'\in [0,t]; s'<t'} \frac{(s')^\eta \left\| P_N^{(1)} X^{N,(2)}_{t'} - P_N^{(1)} X^{N,(2)}_{s'} \right\| _{L^{4/3}}}{(t'-s')^{\gamma }} \right) ^{4/5} +Q.
\end{align*}
Therefore, the assertion of Proposition \ref{prop:global2-2} holds.
\end{proof}

The expectation of the energy functional is estimated as follows.
We remark that in the proof of the following proposition we apply the stationarity of the process $X^N$.

\begin{prop}\label{prop:global2-2exp}
For $\gamma \in (0,1/8)$, $\varepsilon \in (0,\gamma /2)$, $\eta \in [0,1)$, $t\in [0,T]$ and $\delta \in (0,1]$
\begin{align*}
&E\left[ \int _0^t \left( \left\| \nabla X_u^{N,(2),\geqslant}\right\| _{L^2}^2 + \left\| X_u^{N,(2)}\right\| _{L^2}^2 + \lambda \left\| P_N^{(1)} X_u^{N,(2)}\right\| _{L^4}^4 \right) du \right] \\
&\leq \delta E\left[ \sup _{s',t'\in [0,t]; s'<t'} \frac{(s')^\eta \left\| P_N^{(1)} X^{N,(2)}_{t'} - P_N^{(1)} X^{N,(2)}_{s'} \right\| _{L^{4/3}}}{(t'-s')^{\gamma }} \right] + CE\left[ \| X^{N,(2),<}_t \| _{L^2}^2 \right] \\
&\quad + C E\left[ {\mathfrak Y}_{\varepsilon}^N (t) \right] + C\sup _{s\in [0,t]} E\left[ \left\| X_s^{N,(2)} \right\| _{B_1^{-1/2+\varepsilon}} ^q\right]  + C_\delta .
\end{align*}
\end{prop}

\begin{proof}
It holds that for $t \in [0,T]$
\begin{align*}
&\left( X_t^{N,(2)} \right) ^2 - \left( X_0^{N,(2)} \right) ^2\\
&= \left( X_t^N - P_N^{(2)} Z_t\right) ^2 - \left( X_0^N - P_N^{(2)} Z_0\right) ^2 + 2\lambda \left( P_N^{(1)} {\mathcal Z}^{(0,3,N)}_t \right) X_t^{N,(2)} \\
&\quad \hspace{2cm} - \left( \lambda P_N^{(1)} {\mathcal Z}^{(0,3,N)}_t \right) ^2 - 2\lambda \left( P_N^{(1)} {\mathcal Z}^{(0,3,N)}_0 \right) X_0^{N,(2)} + \left( \lambda P_N^{(1)} {\mathcal Z}^{(0,3,N)}_0 \right) ^2.
\end{align*}
By the stationarity of the pair $(X_t^N, Z_t)$, (\ref{eq:invXN}) and (\ref{eq:invZN}) we have for $t \in [0,T]$
\begin{align*}
&\left| E\left[ \left\| X_t^{N,(2)} \right\| _{L^2} ^2 \right] - E\left[ \left\| X_0^{N,(2)} \right\| _{L^2} ^2 \right] \right| \\
&\leq 2 \lambda \left( E\left[ \left| \int _{\Lambda} \left( P_N^{(1)} {\mathcal Z}^{(0,3,N)}_t \right) X_t^{N,(2)} dx \right| \right] + E\left[ \left| \int _{\Lambda} \left( P_N^{(1)} {\mathcal Z}^{(0,3,N)}_0 \right) X_0^{N,(2)} dx \right| \right] \right) + C .
\end{align*}
Hence, Proposition \ref{prop:paraproduct} implies
\begin{align*}
&\left| E\left[ \left\| X_t^{N,(2)} \right\| _{L^2} ^2 \right] - E\left[ \left\| X_0^{N,(2)} \right\| _{L^2} ^2 \right] \right| \\
&\leq 4 \lambda \sup _{s\in [0,t]} E\left[ \left\| {\mathcal Z}^{(0,3,N)}_s \right\| _{B_{\infty}^{(1-\varepsilon)/2}} \left\| X_s^{N,(2)} \right\| _{B_1^{-1/2+\varepsilon}} \right] + C .
\end{align*}
From this inequality and Proposition \ref{prop:global2-2} we obtain that for $t \in [0,T]$
\begin{align*}
&E\left[ \int _0^t \left( \left\| \nabla X_u^{N,(2)}\right\| _{L^2}^2 +  \left\| X_u^{N,(2)}\right\| _{L^2}^2 + \lambda \left\| P_N^{(1)} X_u^{N,(2)}\right\| _{L^4}^4 \right) du \right] \\
&\leq E\left[ Q\left( \sup _{s', t'\in [0,t];s'<t'} \frac{(s')^\eta \left\| P_N^{(1)} X^{N,(2)}_{t'} - P_N^{(1)} X^{N,(2)}_{s'} \right\| _{L^{4/3}}}{(t'-s')^{\gamma }} \right) ^{4/5} \right] + CE\left[ \| X^{N,(2),<}_t \| _{L^2}^2 \right] \\
&\quad + C E\left[ {\mathfrak Y}_{\varepsilon}^N (t) \right] + C\sup _{s\in [0,t]} E\left[ \left\| X_s^{N,(2)} \right\| _{B_1^{-1/2+\varepsilon}} ^q\right] + C .
\end{align*}
Therefore, the assertion of Proposition \ref{prop:global2-2exp} holds.
\end{proof}

Next we consider the estimate of the H\"older continuity that appeared in Proposition \ref{prop:global2-2exp}. 

\begin{prop}\label{prop:global2-3}
For $\gamma \in (0,1/8)$, $\eta \in [0,1)$, $\varepsilon \in (0,\gamma /2)$ and $t\in [0,T]$, 
\begin{align*}
&E\left[ \sup _{s',t'\in [0,t]; s'<t'} \frac{(s')^\eta \left\| X^{N,(2)}_{t'} - X^{N,(2)}_{s'} \right\| _{L^{4/3}}}{(t'-s')^{\gamma }} \right] \\
&\leq C E\left[ \sup _{r\in [0,t]} r^\eta \left\| X^{N,(2),\geqslant}_{r} \right\| _{B_{4/3}^{2\gamma }} \right]  + CE\left[ \sup _{r\in [0,t]} r^\eta \left\| X^{N,(2),<}_r \right\| _{B_{4/3}^{2\gamma }} \right] + C E\left[ \| X_t^{N,(2),<}\| _{L^2}^2 \right] \\
&\quad + C E\left[ {\mathfrak Y}_{\varepsilon}^N (t) \right] + C\sup _{s\in [0,t]} E\left[ \left\| X_s^{N,(2)} \right\| _{B_1^{-1/2+\varepsilon}} ^q\right] + C.
\end{align*}
\end{prop}

\begin{proof}
From (\ref{PDEpara2}) it follows
\begin{align*}
&X^{N,(2),<}_t - e^{(t-s)(\triangle - m_0^2)} X^{N,(2),<}_s\\
& = -3 \lambda \int _s^t e^{(t-u)(\triangle - m_0^2)} P_N^{(1)} \left[ \left( P_N^{(1)} X^{N,(2)}_u - \lambda P_N^{(1)} {\mathcal Z}^{(0,3,N)}_u \right) \mbox{\textcircled{\scriptsize$<$}} {\mathcal Z}^{(2,N)}_u \right] du
\end{align*}
for $s,t \in [0,T]$ such that $s<t$.
Hence, for $s',t,t' \in [0,T]$ such that $s'<t' \leq t$, the mild form representation of the solutions implies
\begin{align*}
&\left\| X^{N,(2),<}_{t'} - X^{N,(2),<}_{s'} \right\| _{L_{4/3}} \\
&\leq \left\| e^{(t'-s')(\triangle - m_0^2)} -I \right\| _{B_{4/3}^{2\gamma } \rightarrow L^{4/3}} \left\| X^{N,(2),<}_{s'} \right\| _{B_{4/3}^{2\gamma }} \\
&\quad + C \lambda \int _{s'}^{t'} \left\| e^{(t'-u)(\triangle - m_0^2)} P_N^{(1)} \left[ \left( P_N^{(1)} X^{N,(2)}_u - \lambda P_N^{(1)} {\mathcal Z}^{(0,3,N)}_u \right) \mbox{\textcircled{\scriptsize$<$}} {\mathcal Z}^{(2,N)}_u \right] \right\| _{L^{4/3}}du \\
&\leq C \lambda (t'-s')^{\gamma } \left\| X^{N,(2),<}_{s'} \right\| _{B_{4/3}^{2\gamma }} \\
&\quad + C \int _{s'}^{t'} (t'-u)^{-1/2-\varepsilon /2} \left\| \left( P_N^{(1)} X^{N,(2)}_u - \lambda P_N^{(1)} {\mathcal Z}^{(0,3,N)}_u \right) \mbox{\textcircled{\scriptsize$<$}} {\mathcal Z}^{(2,N)}_u \right\| _{B_{4/3}^{-1-\varepsilon}}du.
\end{align*}
Applying Proposition \ref{prop:paraproduct} we have
\begin{align*}
&\left\| X^{N,(2),<}_{t'} - X^{N,(2),<}_{s'} \right\| _{L^{4/3}} \\
&\leq C(t'-s')^{\gamma } \left\| X^{N,(2),<}_{s'} \right\| _{B_{4/3}^{2\gamma }} + \lambda Q \int _{s'}^{t'} (t'-u)^{-1/2-\varepsilon /2} \left\| P_N^{(1)} X^{N,(2)}_u - \lambda P_N^{(1)} {\mathcal Z}^{(0,3,N)}_u \right\| _{L^{4/3}}du .
\end{align*}
Thus, by applying (\ref{eq:lemLpestimates1-1}) we have for $s,t \in [0,T]$ such that $s<t$
\begin{equation}\label{eq:propglobal2-3-01}
\begin{array}{l}
\displaystyle \sup _{s',t' \in [0,t]; s'<t'} \frac{(s')^\eta \left\| X^{N,(2),<}_{t'} - X^{N,(2),<}_{s'} \right\| _{L^{4/3}}}{(t'-s')^{\gamma }} \\
\displaystyle \leq C \sup _{r\in [0,t]} \left( r^\eta \left\| X^{N,(2),<}_r \right\| _{B_{4/3}^{2\gamma }} \right) + \lambda \int _0^t \left\| P_N^{(1)} X^{N,(2)}_u \right\| _{L^{4/3}}^{7/2} du + Q.
\end{array}
\end{equation}

Similarly, from (\ref{PDEpara2}) and Proposition \ref{prop:bddp} for $s',t' \in [0,T]$ such that $s'<t'$ we have the estimate
\begin{align*}
&\left\| X^{N,(2),\geqslant}_{t'} - X^{N,(2),\geqslant}_{s'} \right\| _{L^{4/3}} \\
&\leq C(t'-s')^{\gamma } \left\| X^{N,(2),\geqslant}_{s'} \right\| _{B_{4/3}^{2\gamma }} + C \lambda (t'-s')^{\gamma } \int _{s'}^{t'} (t'-u)^{-\gamma } \left\| P_N^{(1)} X^{N,(2)}_u \right\| _{L^4}^3 du \\
&\quad + C \lambda (t'-s')^{\gamma } \int _{s'}^{t'} (t'-u)^{-\gamma } \left\| \Phi _u^{(1)}(P_N^{(1)} X^{N,(2)} ) \right\| _{L^{4/3}} du \\
&\quad + C \lambda (t'-s')^{\gamma } \int _{s'}^{t'} (t'-u)^{-1/4-\gamma -\varepsilon /2} \left\| \Phi _u ^{(2)} (P_N^{(1)} X^{N,(2)}) \right\| _{B_{4/3}^{-1/2-\varepsilon}} du \\
&\quad + C \lambda (t'-s')^{\gamma } \int _{s'}^{t'} (t'-u)^{-1/4-\gamma -\varepsilon /2} \left\| \Phi _u ^{(3)} (P_N^{(1)} X^{N,(2)}) \right\| _{B_{4/3}^{-1/2-\varepsilon}} du \\
&\quad + C \lambda (t'-s')^{\gamma } \int _{s'}^{t'} (t'-u)^{-\gamma }  \left\| (P_N^{(1)} X^{N,(2),\geqslant}) \mbox{\textcircled{\scriptsize$=$}} {\mathcal Z}^{(2,N)}_u \right\| _{L^{4/3}} du \\
&\quad + C \lambda (t'-s')^{\gamma } \int _{s'}^{t'} (t'-u)^{-\gamma }  \left\| \Psi _u^{(1)} (P_N^{(1)} X^{N,(2)}) \mbox{\textcircled{\scriptsize$=$}} {\mathcal Z}^{(2,N)}_u \right\| _{B_{4/3}^{\varepsilon}} du \\
&\quad + C \lambda (t'-s')^{\gamma } \int _{s'}^{t'} (t'-u)^{-\gamma }  \left\| \Psi _u^{(2)} (P_N^{(1)} X^{N,(2)}) \right\| _{B_{4/3}^{\varepsilon}} du .
\end{align*}
For $\delta \in (0,1]$, applying Lemmas \ref{lem:estcom}, \ref{lem:Phi1}, \ref{lem:Phi2}, \ref{lem:Phi3} and \ref{lem:=Z2} with replacing $\delta$ by $\delta (t-u) ^\alpha$ with suitable $\alpha $ for each lemmas, and applying Lemma \ref{lem:timeint}, we have for $\delta \in (0,1)$ $s',t,t' \in [0,T]$ such that $s'<t' \leq t$
\begin{align*}
&\left\| X^{N,(2),\geqslant}_{t'} - X^{N,(2),\geqslant}_{s'} \right\| _{L^{4/3}} \\
&\leq C(t'-s')^{\gamma } \left\| X^{N,(2),\geqslant}_{s'} \right\| _{B_{4/3}^{2\gamma }} \\
&\quad + C (t'-s')^{\gamma }\int _0^t \left( \left\| \nabla X_u^{N,(2),\geqslant}\right\| _{L^2}^2 + \left\| X_u^{N,(2)}\right\| _{L^2}^2  + \lambda \left\| P_N^{(1)} X_u^{N,(2)}\right\| _{L^4}^4 \right) ^{7/8} du \\
&\quad + C (t'-s')^{\gamma } \int _{s'}^{t'} \left\| P_N^{(1)} X^{N,(2),<}_u \right\| _{B_{2}^{1-\varepsilon }}^{7/4} du + Q (t'-s') ^{\gamma } \int _{s'}^{t'} (t'-u)^{-\gamma /3} \left\| P_N^{(1)} X^{N,(2),\geqslant}_u \right\| _{B_{4/3}^{1+\varepsilon }}^{5/6} du\\
&\quad + \delta Q (t'-s')^{\gamma } \int _{s'}^{t'} \sup _{r\in [s',u)} \frac{r^\eta \left\| P_N^{(1)} X^{N,(2)}_u - P_N^{(1)} X^{N,(2)}_r \right\| _{L^{4/3}}}{(u-r)^{\gamma }} du + C_\delta Q (t'-s')^{\gamma } .
\end{align*}
This inequality and (\ref{eq:propglobal2-3-01}) imply
\begin{align*}
&\sup _{s',t'\in [0,t]; s'<t'} \frac{(s' )^\eta \left\| X^{N,(2)}_{t'} - X^{N,(2)}_{s'} \right\| _{L^{4/3}}}{(t'-s')^{\gamma}} \\
&\leq C \sup _{r\in [0,t]}\left( r^\eta \left\| X^{N,(2),\geqslant}_{r} \right\| _{B_{4/3}^{2\gamma }} \right) + C \sup _{r\in [0,t]} \left( r^\eta \left\| X^{N,(2),<}_r \right\| _{B_{4/3}^{2\gamma }} \right)\\
&\quad + C\int _0^t \left( \left\| \nabla X_u^{N,(2),\geqslant}\right\| _{L^2}^2 + \left\| X_u^{N,(2)}\right\| _{L^2}^2 + \lambda \left\| P_N^{(1)} X_u^{N,(2)}\right\| _{L^4}^4 \right) ^{7/8} du\\
&\quad + \delta Q \sup _{s',t'\in [0,t]; s'<t'} \frac{(s') ^\eta \left\| P_N^{(1)} X^{N,(2)}_{t'} - P_N^{(1)} X^{N,(2)}_{s'} \right\| _{L^{4/3}}}{(t'-s')^{\gamma}} + C \int _0^t \left\| P_N^{(1)} X^{N,(2),<}_u \right\| _{B_{2}^{1-\varepsilon }}^{7/4} du \\
&\quad + Q \sup _{s',t' \in [0,t]; s'<t'} \int _{s'}^{t'} (t'-u)^{-1/2 -\gamma -\varepsilon /2} \left\| P_N^{(1)} X^{N,(2)}_u \right\| _{L^{4/3}}du \\
&\quad + Q\sup _{s',t'\in [0,t]; s'<t'} \int _{s'}^{t'} (t'-u)^{-\gamma /3} \left\| P_N^{(1)} X^{N,(2),\geqslant}_u \right\| _{B_{4/3}^{1+\varepsilon }}^{5/6} du + C_\delta Q .
\end{align*}
Hence, by taking $\delta = \frac 12 \left( Q [1+T] \right) ^{-1}$ and applying Proposition \ref{prop:bddp} and (\ref{eq:lemLpestimates1-1}) we obtain
\begin{equation}\label{eq:propglobal2-3-11}\begin{array}{l}
\displaystyle \sup _{s',t'\in [0,t]; s'<t'} \frac{(s')^\eta \left\| X^{N,(2),\geqslant}_{t'} - X^{N,(2),\geqslant}_{s'} \right\| _{L^{4/3}}}{(t'-s')^{\gamma }} \\
\displaystyle \leq C \sup _{r\in [0,t]}\left( r^\eta \left\| X^{N,(2),\geqslant}_{r} \right\| _{B_{4/3}^{2\gamma }} \right) + C \sup _{r\in [0,t]} \left( r^\eta \left\| X^{N,(2),<}_r \right\| _{B_{4/3}^{2\gamma }} \right)\\
\displaystyle \quad + C\int _0^t \left( \left\| \nabla X_u^{N,(2),\geqslant}\right\| _{L^2}^2 + \left\| X_u^{N,(2)}\right\| _{L^2}^2 + \lambda \left\| P_N^{(1)} X_u^{N,(2)}\right\| _{L^4}^4 \right) ^{7/8} du\\
\displaystyle \quad + C \int _0^t \left\| P_N^{(1)} X^{N,(2)}_u \right\| _{B_2^{1-\varepsilon }}^{7/4} du + C \int _0^t \left\| P_N^{(1)} X^{N,(2),\geqslant}_u \right\| _{B_{4/3}^{1+\varepsilon }} du + Q  .
\end{array}\end{equation}
By (\ref{eq:propglobal2-3-11}) and Proposition \ref{prop:global2-2exp}, for $\delta ' \in (0,1]$ we have
\begin{align*}
&E\left[ \sup _{s',t'\in [0,t]; s'<t'} \frac{(s')^\eta \left\| X^{N,(2)}_{t'} - X^{N,(2)}_{s'} \right\| _{L^{4/3}}}{(t'-s')^{\gamma }} \right] \\
&\leq C \delta ' E\left[ \sup _{s',t'\in [0,t]; s'<t'} \frac{(s')^\eta \left\| P_N^{(1)} X^{N,(2)}_{t'} - P_N^{(1)} X^{N,(2)}_{s'} \right\| _{L^{4/3}}}{(t'-s')^{\gamma }} \right] + C E\left[ \| X_t^{N,(2),<}\| _{L^2}^2 \right] \\
&\quad + C E\left[ \sup _{r\in [0,t]} r^\eta \left\| X^{N,(2),\geqslant}_{r} \right\| _{B_{4/3}^{2\gamma }} \right] + CE\left[ \sup _{r\in [0,t]} r^\eta \left\| X^{N,(2),<}_r \right\| _{B_{4/3}^{2\gamma }} \right] \\
&\quad + C E\left[ {\mathfrak Y}_{\varepsilon}^N (t) \right] + C E\left[ \left\| X_t^{N,(2)} \right\| _{B_{1}^{-1/2+\varepsilon}}^q \right] + C _{\delta'} .
\end{align*}
Therefore, by taking $\delta'$ sufficiently small we have the conclusion.
\end{proof}

We have to estimate the terms that appeared in Propositions \ref{prop:global2-2exp} and \ref{prop:global2-3}.
For convenience in the proofs of the estimates we give the following lemma. 

\begin{lem}\label{lem:global2-1}
\begin{enumerate}
\item \label{lem:global2-1-1} For $p\in [1,2]$, $\alpha , \beta \in {\mathbb R}$ and $s,t\in [0,T]$ such that $s<t$,
\begin{align*}
\left\| X^{N,(2),<}_t \right\| _{B_{p}^{\alpha}}
&\leq C (t-s)^{-\beta} \left\| X^{N,(2),<}_s \right\|  _{B_{p}^{\alpha - 2\beta }} \\
&\quad + Q \int _s^t (t-u)^{-(\alpha +1)/2-\varepsilon /4} \left\| P_N^{(1)} X^{N,(2)}_u - \lambda P_N^{(1)} {\mathcal Z}^{(0,3,N)}_u \right\| _{L^p}du 
\end{align*}

\item \label{lem:global2-1-2} For $\gamma \in (0,1/4)$, $\eta \in [0,1)$, $p\in [1,4/3]$, $\varepsilon \in (0,1/16)$, $\alpha \in [0,2]$, $\beta \in {\mathbb R}$ and $\theta \in (0,1]$
\begin{align*}
& \left\| X^{N,(2),\geqslant}_t \right\| _{B_{p}^{\alpha}} \\
& \leq C (t-s)^{-\beta} \left\| X^{N,(2),\geqslant}_s \right\| _{B_{p}^{\alpha -2\beta}} \\
& \quad + \delta \int _0^t (t-u)^{-(2\alpha +1 + 2\varepsilon )/4}  \left( \left\| \nabla X_u^{N,(2),\geqslant}\right\| _{L^2}^2 + \left\| X_u^{N,(2)}\right\| _{L^2}^2 + \lambda \left\| P_N^{(1)} X_u^{N,(2)}\right\| _{L^4}^4 \right) ^{7/8} du \\
& \quad + \delta \int _s^t (t-u)^{-(2\alpha +1 + 2\varepsilon )/4}  \left\| P_N^{(1)} X^{N,(2),<}_u \right\| _{B_{2}^{1-\varepsilon }}^{7/4} du \\
& \quad + \delta \int _s^t (t-u)^{-\alpha /2} \left\| P_N^{(1)} X^{N,(2),\geqslant}_u \right\| _{B_{p}^{1+\varepsilon }}^{5/6} du \\
& \quad + \lambda Q \int _s^t (t-u)^{-(\alpha -\varepsilon )/2} \left( \sup _{r\in [s,u)} \frac{r^\eta \left\| P_N^{(1)} X^{N,(2)}_u - P_N^{(1)} X^{N,(2)}_r \right\| _{L^{p}}}{(u-r)^{\gamma }} \right) ^{\theta} \\
& \quad \hspace{0.5cm} \times \left(\| P_N^{(1)} X^{N,(2)}_u \| _{L^{p}}^{1-\theta} + \int _0^u r^{-\eta \theta}(u-r)^{\gamma \theta -1 -3\varepsilon /2} \left\| P_N^{(1)} X^{N,(2)}_r \right\| _{L^{p}} ^{1-\theta} dr \right) du + C_\delta Q.
\end{align*}

\item \label{lem:global2-1-3} For $\gamma \in (0,1/4)$, $\eta \in [0,1)$, $\varepsilon \in (0,1/16)$, $p\in [1,4/3]$, $\alpha \in [0,2/3]$, $\beta \in {\mathbb R}$ and $\theta \in (0,1]$
\begin{align*}
& \left\| X^{N,(2),\geqslant}_t \right\| _{B_{p}^{\alpha}} \\
& \leq C (t-s)^{-\beta} \left\| X^{N,(2),\geqslant}_s \right\| _{B_{p}^{\alpha -2\beta}} \\
& \quad + \delta \int _0^t (t-u)^{-\alpha /2} \left( \left\| \nabla X_u^{N,(2),\geqslant}\right\| _{L^2}^2+ \left\| X_u^{N,(2)}\right\| _{L^2}^2  + \lambda \left\| P_N^{(1)} X_u^{N,(2)}\right\| _{L^4}^4 \right) ^{7/8} du \\
& \quad + \delta \int _s^t (t-u)^{-\alpha /2} \left\| P_N^{(1)} X^{N,(2),<}_u \right\| _{B_{2}^{1-\varepsilon }}^{7/4} du + \delta \int _s^t (t-u)^{-\alpha /2} \left\| P_N^{(1)} X^{N,(2),\geqslant}_u \right\| _{B_{p}^{1+\varepsilon }}^{5/6} du \\
& \quad + \lambda Q \int _s^t (t-u)^{-(\alpha -\varepsilon )/2} \left( \sup _{r\in [s,u)} \frac{r^\eta \left\| P_N^{(1)} X^{N,(2)}_u - P_N^{(1)} X^{N,(2)}_r \right\| _{L^{p}}}{(u-r)^{\gamma }} \right) ^{\theta} \\
& \quad \hspace{0.5cm} \times \left(\| P_N^{(1)} X^{N,(2)}_u \| _{L^{p}}^{1-\theta} + \int _0^u r^{-\eta \theta}(u-r)^{\gamma \theta -1 -3\varepsilon /2} \left\| P_N^{(1)} X^{N,(2)}_r \right\| _{L^{p}} ^{1-\theta} dr \right) du + C_\delta Q.
\end{align*}

\end{enumerate}
\end{lem}

\begin{proof}
Similarly to the beginning of the proof of Proposition \ref{prop:global2-3} we have for $s,t\in [0,T]$ such that $s<t$
\begin{align*}
\left\| X^{N,(2),<}_t \right\| _{B_{p}^{\alpha}}
&\leq C (t-s)^{-\beta} \left\| X^{N,(2),<}_s \right\| _{B_{p}^{\alpha - 2\beta }} \\
&\quad + C \int _s^t (t-u)^{-(\alpha +1 )/2 -\varepsilon /4} \left\| \left( P_N^{(1)} X^{N,(2)}_u - \lambda P_N^{(1)} {\mathcal Z}^{(0,3,N)}_u \right) \mbox{\textcircled{\scriptsize$<$}} {\mathcal Z}^{(2,N)}_u \right\| _{B_{p}^{-1-\varepsilon /2}} du .
\end{align*}
Therefore Proposition \ref{prop:paraproduct} yields \ref{lem:global2-1-1}. 

Similarly to above, from (\ref{PDEpara2}) and Propositions \ref{prop:paraproduct} and \ref{prop:bddp} we have
\begin{align*}
\left\| X^{N,(2),\geqslant}_t \right\| _{B_{p}^{\alpha}}
&\leq C (t-s)^{-\beta} \left\| X^{N,(2),\geqslant}_s \right\| _{B_{p}^{\alpha - 2\beta }} + C\int _s^t (t-u)^{-\alpha /2} \left\| \left( P_N^{(1)} X^{N,(2)}_u \right) ^3 \right\| _{L^{p}} du \\
&\quad + C\int _s^t (t-u)^{-\alpha /2} \left\| \Phi _u^{(1)}( P_N^{(1)} X^{N,(2)} ) \right\| _{L^p} du \\
&\quad + C\int _s^t (t-u)^{-(2\alpha +1 + 2\varepsilon )/4} \left\| \Phi _u ^{(2)} ( P_N^{(1)} X^{N,(2)}) \right\| _{B_{p}^{-1/2-\varepsilon}} du \\
&\quad + C\int _s^t (t-u)^{-(2\alpha +1 + 2\varepsilon )/4} \left\| \Phi _u ^{(3)} ( P_N^{(1)} X^{N,(2)}) \right\| _{B_{p}^{-1/2-\varepsilon}} du \\
&\quad + C\int _s^t (t-u)^{-\alpha /2}\left\| ( P_N^{(1)} X^{N,(2)}_u) \mbox{\textcircled{\scriptsize$=$}} {\mathcal Z}^{(2,N)}_u \right\| _{L^p} du \\
&\quad + C\int _s^t (t-u)^{-(\alpha -\varepsilon )/2} \left\| \Psi _u^{(1)} ( P_N^{(1)} X^{N,(2)} ) \mbox{\textcircled{\scriptsize$=$}} {\mathcal Z}^{(2,N)}_u \right\| _{B_{p}^{\varepsilon}} du \\
&\quad + C\int _s^t (t-u)^{-(\alpha -\varepsilon )/2} \left\| \Psi_u^{(2)} ( P_N^{(1)} X^{N,(2)} ) \right\| _{B_{p}^{\varepsilon}} du .
\end{align*}
Hence, by the fact that 
\[
\left\| \left( P_N^{(1)} X^{N,(2)}_u \right) ^k \right\| _{L^{p}} = \left\| P_N^{(1)} X^{N,(2)}_u \right\| _{L^{kp}}^k \leq C \left\| P_N^{(1)} X^{N,(2)}_u \right\| _{L^{4}}^k
\]
for $k=2,3$, and Lemmas \ref{lem:timeint}, \ref{lem:estcom}, \ref{lem:Phi1} and \ref{lem:=Z2}, and Lemmas \ref{lem:Phi2} and \ref{lem:Phi3} with and without replacing $\delta$ by $\delta (t-u)^{(1+2\varepsilon )/4}$, we obtain \ref{lem:global2-1-2} and \ref{lem:global2-1-3}.
\end{proof}

\begin{prop}\label{prop:global2-1exp}
For $\gamma \in (0,1/8)$, $\eta \in [0,1)$, $\varepsilon \in (0,\gamma /2)$, $q\in (1, 8/7) $, $t\in [0,T]$ and $\delta \in (0,1]$, 
\[
E\left[ {\mathfrak Y}_{\varepsilon}^N (t) ^{q} \right] \leq C E\left[ \left\| X^{N,(2)}_0 \right\| _{B_{4/3}^{-1+2\gamma + 3\varepsilon }} ^{q} \right] + C \delta E\left[ {\mathfrak X}_{\lambda , \eta ,\gamma}^N (t) \right] + C_\delta .
\]
\end{prop}

\begin{proof}
By Lemmas \ref{lem:timeint} and \ref{lem:global2-1}\ref{lem:global2-1-2} we have for $t\in [0,T]$
\begin{equation}\label{eq:propglobal2-1exp-01}\begin{array}{l}
\displaystyle \int_0^t \left\| X^{N,(2),\geqslant}_u \right\| _{B_{4/3}^{1+\varepsilon}} du \\
\displaystyle \leq C \left\| X^{N,(2),\geqslant}_0 \right\| _{B_{4/3}^{-1+3\varepsilon}} + C \delta \int _0^t \left( \left\| \nabla X_u^{N,(2),\geqslant}\right\| _{L^2}^2 + \left\| X_u^{N,(2)}\right\| _{L^2}^2 + \lambda \left\| P_N^{(1)} X_u^{N,(2)}\right\| _{L^4}^4 \right) ^{7/8} du \\
\displaystyle \quad + C \delta \int _0^t \left\| P_N^{(1)} X^{N,(2),<}_u \right\| _{B_{2}^{1-\varepsilon }}^{7/4} du + C \delta \int _0^t \left\| P_N^{(1)} X^{N,(2),\geqslant}_u \right\| _{B_{4/3}^{1+\varepsilon }}^{5/6} du \\
\displaystyle \quad + \lambda Q \int _0^t \left( \sup _{r\in [0,u)} \frac{r^\eta \left\| P_N^{(1)} X^{N,(2)}_u - P_N^{(1)} X^{N,(2)}_r \right\| _{L^{4/3}}}{(u-r)^{\gamma }} \right) ^{1/2} \\
\displaystyle \quad \hspace{1cm} \times \left(\| P_N^{(1)} X^{N,(2)}_u \| _{L^{4/3}}^{1/2} + \int _0^u r^{-\eta /2}(u-r)^{\gamma /2 - 1 -3\varepsilon /2} \left\| P_N^{(1)} X^{N,(2)}_r \right\| _{L^{4/3}} ^{1/2} dr \right) du + C_\delta Q.
\end{array}\end{equation}
On the other hand, by Lemma \ref{lem:timeint}
\begin{align*}
&\int _0^t\left( \int _0^u r^{-\eta /2} (u-r)^{\gamma /2 - 1 -3\varepsilon /2} \left\| P_N^{(1)} X^{N,(2)}_r \right\| _{L^{4/3}} ^{1/2} dr \right) du \\
&\leq C \int _0^t r^{-\eta /2} (t-r)^{\gamma /2 -3\varepsilon /2} \left\| P_N^{(1)} X^{N,(2)}_r \right\| _{L^{4/3}} ^{1/2} dr \leq C \left( \int _0^t \left\| P_N^{(1)} X^{N,(2)}_r \right\| _{L^4}^4 dr \right) ^{1/8} .
\end{align*}
This inequality and (\ref{eq:lemLpestimates1-1}) imply
\begin{align*}
&Q \int _0^t \left( \sup _{r\in [0,u)} \frac{r^\eta \left\| P_N^{(1)} X^{N,(2)}_u - P_N^{(1)} X^{N,(2)}_r \right\| _{L^{4/3}}}{(u-r)^{\gamma }} \right) ^{1/2} \\
&\quad \hspace{1cm} \times \left(\| P_N^{(1)} X^{N,(2)}_u \| _{L^{4/3}}^{1/2} + \int _0^u r^{-\eta /2} (u-r)^{\gamma /2 - 1 -3\varepsilon /2} \left\| P_N^{(1)} X^{N,(2)}_r \right\| _{L^{4/3}} ^{1/2} dr \right) du \\
&\leq C Q \left( \sup _{s',t'\in [0,t]; s'<t'} \frac{(s')^\eta \left\| P_N^{(1)} X^{N,(2)}_{t'} - P_N^{(1)} X^{N,(2)}_{s'} \right\| _{L^{4/3}}}{(t'-s')^{\gamma }} \right) ^{1/2} \left( \int _0^t \left\| P_N^{(1)} X^{N,(2)}_r \right\| _{L^4}^4 dr \right) ^{1/8} \\
&\leq \delta \left( \sup _{s',t'\in [0,t]; s'<t'} \frac{(s')^\eta \left\| P_N^{(1)} X^{N,(2)}_{t'} - P_N^{(1)} X^{N,(2)}_{s'} \right\| _{L^{4/3}}}{(t'-s')^{\gamma }} \right) ^{7/8} + \delta \left( \int _0^t \left\| P_N^{(1)} X^{N,(2)}_r \right\| _{L^4}^4 dr \right) ^{7/8} \\
&\quad + C_\delta Q .
\end{align*}
Hence, (\ref{eq:propglobal2-1exp-01}) yields
\begin{equation}\label{eq:propglobal2-1exp-02}\begin{array}{l}
\displaystyle \int_0^t \left\| X^{N,(2),\geqslant}_u \right\| _{B_{4/3}^{1+\varepsilon}} du \\
\displaystyle \leq C \left\| X^{N,(2),\geqslant}_0 \right\| _{B_{4/3}^{-1+3\varepsilon}} \\
\displaystyle + C \delta \int _0^t \left( \left\| \nabla X_u^{N,(2),\geqslant}\right\| _{L^2}^2 + \left\| X_u^{N,(2)}\right\| _{L^2}^2 + \lambda \left\| P_N^{(1)} X_u^{N,(2)}\right\| _{L^4}^4 \right) ^{7/8} du \\
\displaystyle \quad + \delta \left( \sup _{s',t'\in [0,t]; s'<t'} \frac{(s')^\eta \left\| P_N^{(1)} X^{N,(2)}_{t'} - P_N^{(1)} X^{N,(2)}_{s'} \right\| _{L^{4/3}}}{(t'-s')^{\gamma }} \right) ^{7/8} \\
\displaystyle \quad + \delta \int _0^t \left( \left\| P_N^{(1)} X^{N,(2),<}_u \right\| _{B_{2}^{1-\varepsilon }}^{7/4} + \left\| P_N^{(1)} X^{N,(2),\geqslant}_u \right\| _{B_{4/3}^{1+\varepsilon }} \right) du + C_\delta Q.
\end{array}\end{equation}
Thus, from this inequality and Lemma \ref{lem:v3} we get
\[
E\left[ {\mathfrak Y}_{\varepsilon}^N (t) ^{q} \right] \leq C E\left[ \left\| X^{N,(2),\geqslant}_0 \right\| _{B_{4/3}^{-1 +3\varepsilon +2\gamma /3}} ^q \right] + C \delta E\left[ {\mathfrak X}_{\lambda ,\eta ,\gamma}^N (t) \right] + \delta E\left[ {\mathfrak Y}_{\varepsilon}^N (t) ^{q} \right] + C_\delta .
\]
By taking $\delta$ sufficiently small, we obtain the desired inequality.
\end{proof}

\begin{prop}\label{prop:estimate10}
For $\gamma \in (0,1/8)$, $\eta \in [0,1)$, $\varepsilon \in (0, \gamma /2]$, $q\in (1, 8/7) $, $t\in ( 0,T]$ and $\delta \in (0,1]$, 
\[
\sup _{s\in [0,t]} E\left[ \left\| X_s^{N,(2)} \right\| _{B_1^{-1/2+\varepsilon}} ^q\right] \leq \delta E\left[{\mathfrak X}_{\lambda ,\eta ,\gamma}^N (T) \right] + C_\delta .
\]
\end{prop}

\begin{proof}
By the stationarity of the pair $(X_t^N, Z_t)$, (\ref{eq:invXN}) and (\ref{eq:invZN}) we have for $t \in [0,T]$
\begin{align*}
E\left[ \left\| X_t^{N,(2)} \right\| _{B_1^{-1/2+\varepsilon}} ^q\right]
& \leq C E\left[ \left\| X_t^N - P_N^{(2)} Z_t  \right\| _{B_1^{-1/2+\varepsilon}} ^q\right] + C E\left[ \left\| {\mathcal Z}^{(0,3,N)}_t \right\| _{B_1^{-1/2+\varepsilon}} ^q\right] \\
&=  \frac{C}{T} \int _0^T E\left[ \left\| X_s^N - P_N^{(2)} Z_s \right\| _{B_1^{-1/2+\varepsilon}} ^q\right] ds + C \\
&\leq  C \int _0^T E\left[ \left\| X_s^{N,(2)} \right\| _{B_1^{-1/2+\varepsilon}} ^q\right] ds + C .
\end{align*}
Hence, by (\ref{eq:lemLpestimates1-1}) we have the assertion.
\end{proof}

\begin{prop}\label{prop:estsup}
For $\gamma \in (0,1/8)$, $\eta \in [0,1)$, $\varepsilon \in (0,\gamma /2)$, $q\in (1, 8/7) $, $t\in [0,T]$ and $\delta \in (0,1]$, we have
\begin{align*}
&E\left[ \sup _{r\in [0,t]} r^\eta \left\| X^{N,(2),<}_{r} \right\| _{B_{4}^{2\gamma }}^3 \right] + E\left[ \sup _{r\in [0,t]} r^\eta \left\| X^{N,(2),\geqslant}_{r} \right\| _{B_{4/3}^{2\gamma }} \right]\\
&\leq C E\left[ \left\| X^{N,(2)}_0 \right\| _{B_{4/3}^{2(\gamma - \eta )}} \right] + C \delta E\left[{\mathfrak X}_{\lambda , \eta ,\gamma}^N (t) \right] + C \delta E\left[ {\mathfrak Y}_{\varepsilon}^N (t) ^q \right] + C_\delta ,
\end{align*}
for some constants $C$ and $C_\delta$.
\end{prop}

\begin{proof}
By Lemma \ref{lem:global2-1}\ref{lem:global2-1-1} we have
\begin{align*}
&E\left[ \sup _{r\in [0,t]} r^\eta \left\| X^{N,(2),<}_{r} \right\| _{B_{4}^{2\gamma }}^3 \right] \\
&\leq \lambda E\left[ Q \sup _{r\in [0,t]} \left( \int _0^r (r-u)^{-\gamma -1/2 -\varepsilon /4} \left\| P_N^{(1)} X^{N,(2)}_u - \lambda P_N^{(1)} {\mathcal Z}^{(0,3,N)}_u \right\| _{L^{4}} du \right) ^3\right] \\
&\leq \lambda E\left[ Q \sup _{r\in [0,t]} \left( \int _0^r (r-u)^{-(4 \gamma +2 +\varepsilon )/3} du \right) ^4 \left( \int _0^t \left\| P_N^{(1)} X^{N,(2)}_u - \lambda P_N^{(1)} {\mathcal Z}^{(0,3,N)}_u \right\| _{L^4}^4 du \right) ^{3/4}\right] .
\end{align*}
Hence, we have for $\delta \in (0,1]$
\begin{equation}\label{eq:propestsup01}
E\left[ \sup _{r\in [0,t]} r^\eta \left\| X^{N,(2),<}_{r} \right\| _{B_{4}^{2\gamma }}^3 \right] \leq \delta E\left[ \int _0^t  \left\| P_N^{(1)} X^{N,(2)}_u\right\| _{L^4}^4 du \right] + C_\delta.
\end{equation}

By Lemma \ref{lem:global2-1}\ref{lem:global2-1-3} and H\"older's inequality, for $\delta \in (0,1]$ we have
\begin{align*}
&E\left[ \sup _{r\in [0,t]} r^\eta \left\| X^{N,(2),\geqslant}_{r} \right\| _{B_{4/3}^{2\gamma }} \right] \\
&\leq C E\left[ \left\| X^{N,(2)}_0 \right\| _{B_{4/3}^{2(\gamma - \eta )}} \right] + \delta E\left[ {\mathfrak Y}_{\varepsilon}^N (t) ^q \right]\\
&\quad + \delta E\left[ \int _0^t \left( \left\| \nabla X_u^{N,(2),\geqslant}\right\| _{L^2}^2 + \left\| X_u^{N,(2)}\right\| _{L^2}^2 + \lambda \left\| P_N^{(1)} X_u^{N,(2)}\right\| _{L^4}^4 \right) du \right]  \\
& \quad + \lambda E\left[ Q \sup _{r\in [0,t]} \int _0^r  (r-u)^{-\gamma +\varepsilon /2} \left( \sup _{v\in [0,u)} \frac{v^\eta \left\| P_N^{(1)} X^{N,(2)}_u - P_N^{(1)} X^{N,(2)}_v \right\| _{L^{4/3}}}{(u-v)^{\gamma }} \right) ^{1/2} \right. \\
& \quad \hspace{1cm} \left. \times \left(\| P_N^{(1)} X^{N,(2)}_u \| _{L^{4/3}}^{1/2} + \int _0^u v^{-\eta /2}(u-v)^{\gamma /2 -1 -3\varepsilon /2} \left\| P_N^{(1)} X^{N,(2)}_v \right\| _{L^{4/3}} ^{1/2} dv \right) du \right] + C_\delta.
\end{align*}
Since in view of Lemma \ref{lem:timeint} and H\"older's inequality it holds that
\begin{align*}
&E\left[ Q \sup _{r\in [0,t]} \int _0^r (r-u)^{-\gamma +\varepsilon /2} \left( \sup _{v\in [0,u)} \frac{v^\eta \left\| P_N^{(1)} X^{N,(2)}_u - P_N^{(1)} X^{N,(2)}_v \right\| _{L^{4/3}}}{(u-v)^{\gamma }} \right) ^{1/2} \right. \\
& \quad \left. \phantom{\left( \frac{\left\| X^{N,(2)}_{t'} \right\|}{(t'-s')^{\gamma }} \right) }\times \left(\| P_N^{(1)} X^{N,(2)}_u \| _{L^{4/3}}^{1/2} + \int _0^u v^{-\eta /2}(u-v)^{\gamma /2 -1 -3\varepsilon /2} \left\| P_N^{(1)} X^{N,(2)}_v \right\| _{L^{4/3}} ^{1/2} dv \right) du \right] \\
&\leq E\left[ Q \left( \sup _{s',t'\in [0,t]; s' <t'} \frac{(s')^\eta \left\| P_N^{(1)} X^{N,(2)}_{t'} - P_N^{(1)} X^{N,(2)}_{s'} \right\| _{L^{4/3}}}{(t'-s')^{\gamma }} \right) ^{1/2} \right. \\
&\quad \hspace{1cm} \times \left( \sup _{r\in [0,t]} \int _0^r (r-u)^{-\gamma +\varepsilon /2} \| P_N^{(1)} X^{N,(2)}_u \| _{L^{4/3}}^{1/2} du \right. \\
&\quad \left. \phantom{\left( \frac{\left\| X^{N,(2)}_{t'} \right\| _{L^{4/3}}}{(t'-s')^{\gamma }} \right) ^{1/2}} \left. + \sup _{r\in [0,t]} \int _0^r u^{-\eta /2} (r-u)^{-\gamma /2 -\varepsilon } \left\| P_N^{(1)} X^{N,(2)}_u \right\| _{L^{4/3}} ^{1/2} du \right) \right] \\
&\leq E\left[ Q \left( \sup _{s',t'\in [0,t]; s' <t'} \frac{(s')^\eta \left\| P_N^{(1)} X^{N,(2)}_{t'} - P_N^{(1)} X^{N,(2)}_{s'} \right\| _{L^{4/3}}}{(t'-s')^{\gamma }} \right) ^{1/2} \left( 1+ \int _0^t \| P_N^{(1)} X^{N,(2)}_u \| _{L^{4/3}} du \right) \right] \\
&\leq C_\delta + \delta E\left[ \sup _{s',t'\in [0,t]; s' <t'} \frac{(s')^\eta \left\| P_N^{(1)} X^{N,(2)}_{t'} - P_N^{(1)} X^{N,(2)}_{s'} \right\| _{L^{4/3}}}{(t'-s')^{\gamma }} \right] + C \delta E\left[ \int _0^t \| P_N^{(1)} X^{N,(2)}_u \| _{L^4}^3 du \right] ,
\end{align*}
we obtain for $\delta \in (0,1]$
\[
E\left[ \sup _{r\in [0,t]} r^\eta \left\| X^{N,(2),\geqslant}_{r} \right\| _{B_{4/3}^{2\gamma }} \right] \\
\leq C E\left[ \left\| X^{N,(2)}_0 \right\| _{B_{4/3}^{2(\gamma - \eta )}} \right]  + C \delta E\left[{\mathfrak X}_{\lambda ,\eta ,\gamma}^N (t) \right] + C \delta E\left[ {\mathfrak Y}_{\varepsilon}^N (t) ^q \right] + C_\delta.
\]
Therefore, by this inequality and (\ref{eq:propestsup01}) we have the assertion.
\end{proof}

We have finished estimating the terms.
So, now we obtain the following uniform estimate.

\begin{thm}\label{thm:tight1}
Let $\gamma \in (0,1/8)$, $\eta \in [0,1)$, $\varepsilon \in (0, \gamma /2)$ and $q\in (1, 8/7)$.
Assume that 
\[
\eta > \gamma + \frac 14.
\]
Then, 
\[
E\left[{\mathfrak X}_{\lambda , \eta ,\gamma}^N (T) \right] + E\left[ {\mathfrak Y}_{\varepsilon}^N (T) ^q \right]+E\left[ \sup _{r\in [0,T]} r^\eta \left\| X^{N,(2),<}_{r} \right\| _{B_4^{2\gamma }}^3 \right] + E\left[ \sup _{r\in [0,T]} r^\eta \left\| X^{N,(2),\geqslant}_{r} \right\| _{B_{4/3}^{2\gamma }} \right] \leq C .
\]
\end{thm}

\begin{proof}
Propositions \ref{prop:global2-2exp}, \ref{prop:global2-3}, \ref{prop:global2-1exp}, \ref{prop:estimate10} and \ref{prop:estsup} imply that for $\delta \in (0,1]$
\begin{align*}
&E\left[{\mathfrak X}_{\lambda , \eta ,\gamma}^N (T) \right] + E\left[ {\mathfrak Y}_{\varepsilon}^N (T) ^q \right]+ E\left[ \sup _{r\in [0,T]} r^\eta \left\| X^{N,(2),<}_{r} \right\| _{B_4^{2\gamma }}^3 \right] + E\left[ \sup _{r\in [0,T]} r^\eta \left\| X^{N,(2),\geqslant}_{r} \right\| _{B_{4/3}^{2\gamma }} \right] \\
&\leq C\delta E\left[{\mathfrak X}_{\lambda , \eta ,\gamma}^N (T) \right] + C \delta E\left[ {\mathfrak Y}_{\varepsilon}^N (T) ^q \right] + C E\left[ \left\| X^{N,(2),<}_{T} \right\| _{L^2}^2 \right] \\
&\quad + C E\left[ \left\| X^{N,(2)}_0 \right\| _{B_{4/3}^{2(\gamma - \eta )}} \right] + C E\left[ \left\| X^{N}_0 \right\| _{B_{4/3}^{-1+2\gamma + 3\varepsilon}} ^{q} \right] + C_\delta.
\end{align*}
On the other hand,
\[
E\left[ \left\| X^{N,(2),<}_{T} \right\| _{L^2}^2 \right] \leq \delta E\left[ \sup _{r\in [0,T]} r^\eta \left\| X^{N,(2),<}_{r} \right\| _{B_4^{2\gamma }}^3 \right] + C_\delta.
\]
Hence, by taking $\delta$ sufficiently small we have
\begin{equation}\label{eq:thmtight1-1}\begin{array}{l}
\displaystyle E\left[{\mathfrak X}_{\lambda , \eta ,\gamma}^N (T) \right] + E\left[ {\mathfrak Y}_{\varepsilon}^N (T) ^q \right]\\
\displaystyle + E\left[ \sup _{r\in [0,T]} r^\eta \left\| X^{N,(2),<}_{r} \right\| _{B_{4/3}^{2\gamma }}^3 \right] + E\left[ \sup _{r\in [0,T]} r^\eta \left\| X^{N,(2),\geqslant}_{r} \right\| _{B_{4/3}^{2\gamma }} \right] \\
\displaystyle \leq C E\left[ \left\| X^{N,(2)}_0 \right\| _{B_{4/3}^{2(\gamma - \eta )}} \right] + C E\left[ \left\| X^{N,(2)}_0 \right\| _{B_{4/3}^{-1+2\gamma + 3\varepsilon}} ^{q} \right] + C.
\end{array}\end{equation}
The invariance of  the law of $X^{N}_0$ with respect to $X^N$ implies that
\begin{align*}
&E\left[ \left\| X^{N,(2)}_0 \right\| _{B_{4/3}^{2(\gamma - \eta )}} \right] \\
&\leq \frac{1}{T} \int _0^T E\left[ \left\| X^{N}_t \right\| _{B_{4/3}^{2(\gamma - \eta )}} \right] dt + E\left[ \left\| P_N^{(2)}Z_0 \right\| _{B_{4/3}^{2(\gamma - \eta )}} + \left\| {\mathcal Z}^{(0,3,N)}_0 \right\| _{B_{4/3}^{2(\gamma - \eta )}}\right] \\
&\leq \frac{1}{T} \int _0^T E\left[ \left\| X^{N,(2)}_t \right\| _{B_{4/3}^{2(\gamma - \eta )}} \right] dt + 2 \sup _{t\in [0,T]} E\left[ \left\| P_N^{(2)}Z_t \right\| _{B_{4/3}^{2(\gamma - \eta )}} + \left\| {\mathcal Z}^{(0,3,N)}_t \right\| _{B_{4/3}^{2(\gamma - \eta )}} \right] \\
&\leq \frac 13 E\left[{\mathfrak X}_{\lambda, \eta ,\gamma}^N (T) \right] + C.
\end{align*}
Similarly it holds that
\[
E\left[ \left\| X^{N,(2)}_0 \right\| _{B_{4/3}^{-1+2\gamma + 3\varepsilon }} ^{q} \right] \leq \frac 13 E\left[{\mathfrak X}_{\lambda , \eta ,\gamma}^N (T) \right] + C.
\]
By these inequalities and (\ref{eq:thmtight1-1}) we obtain the assertion.
\end{proof}

Theorem \ref{thm:tight1} yields the tightness of the laws of $\{ X^N\}$, which is the target in the present paper.

\begin{thm}\label{thm:tight2}
For $\varepsilon \in (0, 1/16]$, the laws of $\{ X^N\}$ are tight on $C([0,\infty ); {B_{4/3}^{-1/2-\varepsilon }})$.
Moreover, if $X$ is a limit of a subsequence $\{ X^{N(k)}\}$ of $\{ X^N\}$ on $C([0,\infty ); {B_{4/3}^{-1/2-\varepsilon }})$, then $X$ is a continuous process on $B_{4/3}^{-1/2-\varepsilon }$, the limit measure $\mu$ (in the weak convergence sense) of the associated subsequence $\{ \mu _{N(k)}\}$ is an invariant measure with respect to $X$ and it holds that
\[
\int \| \phi \| _{B_2^{-1/2-\varepsilon}}^2 \mu (d\phi ) < \infty .
\]
\end{thm}

\begin{proof}
Let $T\in (0,\infty )$ and $t_0 \in (0,T)$.
Take $\gamma$, $\eta$ and $\varepsilon$ as in Theorem \ref{thm:tight1}.
For $h\in (0,1]$ and $\varepsilon' \in (0,1]$, Chebyshev's inequality implies that
\begin{align*}
&\sup _{N\in {\mathbb N}} P \left( \sup _{s,t\in [t_0,T]; |s-t|<h} \left\| X^{N,(2)}_t - X^{N,(2)}_s \right\| _{L^{4/3}} > \varepsilon' \right) \\
&\leq \frac{h^\gamma }{\varepsilon' t_0^\eta } E\left[ \sup _{s,t\in [t_0,T]; s<t, t-s<h} \frac{s^\eta \left\| X^{N,(2)}_{t} - X^{N,(2)}_{s} \right\| _{L^{4/3}}}{(t-s)^{\gamma }} \right] 
\end{align*}
Hence, from Theorem \ref{thm:tight1} we obtain
\begin{equation}\label{eq:thmtight2-1}
\lim _{h\downarrow 0} \sup _{N\in {\mathbb N}} P \left( \sup _{s,t\in [t_0,T]; |s-t|<h} \left\| X^{N,(2)}_t - X^{N,(2)}_s \right\| _{L^{4/3}} > \varepsilon' \right) =0
\end{equation}
for $\varepsilon' \in (0,1]$.
On the other hand, Chebyshev's inequality implies that, for any $R>0$,
\[
\sup _{N\in {\mathbb N}} P \left( \left\| X^{N,(2)}_{t_0} \right\|  _{B_{4/3}^{2\gamma }} >R \right) 
\leq \frac{1}{R t_0^\eta} \sup _{N\in {\mathbb N}} E\left[ \sup _{r\in [0,T]} r^\eta \left\| X^{N,(2)}_{r} \right\| _{B_{4/3}^{2\gamma }}\right].
\]
Hence, by Theorem \ref{thm:tight1} we obtain
\begin{equation}\label{eq:thmtight2-2}
\lim _{R\rightarrow \infty} \sup _{N\in {\mathbb N}} P \left( \left\| X^{N,(2)}_{t_0} \right\|  _{B_{4/3}^{2\gamma }} >R \right) =0 .
\end{equation}
In view of the fact that the unit ball in $B_{4/3}^{2\gamma }$ is compactly embedded in $L^{4/3}$ (see Theorem 2.94 \cite{BCD}), the tightness of the laws of $\{ X^{N,(2)}\}$ on $C([t_0,T]; {L^{4/3}})$ follows from (\ref{eq:thmtight2-1}) and (\ref{eq:thmtight2-2}).
From this fact, the tightness of the laws of $\{ P_N^{(2)}Z \}$ and $\{ {\mathcal Z}^{(0,3,N)} \}$ on $C([t_0,T]; {B_{\infty}^{-1/2-\varepsilon }})$ for sufficiently small $\varepsilon \in (0,1]$, and Corollary \ref{cor:AT}, we have the tightness of $\{ X^N\}$ on $C([t_0,T]; {B_{4/3}^{-1/2-\varepsilon }})$.
For $N\in {\mathbb N}$, in view of the Markov property of $X^N$ and the invariance of $\mu ^N$ with respect to $X^N$, the law of $X^N$ on $C([t_0,T]; {B_{4/3}^{-1/2-\varepsilon }})$ coincides with the law of $X^N$ on $C([0,T-t_0]; {B_{4/3}^{-1/2-\varepsilon }})$.
Hence, we have the tightness of the laws of $\{ X^N\}$ on $C([0,T-t_0]; {B_{4/3}^{-1/2-\varepsilon }})$.
Since $T\in (0,\infty )$ and $t_0 \in (0,T)$ are arbitrary and the topology of $C([0,\infty ); {B_{4/3}^{-1/2-\varepsilon }})$ is given by uniform convergence on compact sets, we obtain the tightness of the laws of $\{ X^N\}$ on $C([0,\infty ); {B_{4/3}^{-1/2-\varepsilon }})$.
By construction there is then a continuous limit process $X$ (which might depend on the subsequence).

Let $f$ be a bounded continuous function on $B_{4/3}^{-1/2-\varepsilon }$.
Then, by the invariance of $\mu ^N$ with respect to $X^N$ for any $N\in {\mathbb N}$, we have
\[
E\left[f(X_t)\right] = \lim _{N\rightarrow \infty} E\left[  f(X^N_t) \right] = \lim _{N\rightarrow \infty} \int f d \mu _N = \int f d\mu , \quad t\in [0,\infty ).
\]
Therefore, $\mu$ is invariant with respect to $X$.
Moreover, by the invariance of $\mu ^N$ with respect to $X^N$, for $t\in (0,\infty )$ we have
\begin{align*}
&E\left[ \| X_0 \| _{B_2^{-1/2-\varepsilon}}^2 \right] \leq \liminf _{N\rightarrow \infty} E\left[ \| X_0^N \| _{B_2^{-1/2-\varepsilon}}^2 \right] \leq C \liminf _{N\rightarrow \infty} E\left[ \| X_0^{N,(2)} \| _{B_2^{-1/2-\varepsilon}}^2 \right] +C \\
&\quad = \frac{C}{t} \liminf _{N\rightarrow \infty} \int _0^t E\left[ \| X_t^{N,(2)} \| _{B_2^{-1/2-\varepsilon}}^2 \right] dt +C \leq \frac{C}{t} \liminf _{N\rightarrow \infty} {\mathfrak X}_{\lambda , \eta ,\gamma}^N (t) +C .
\end{align*}
From this also the last assertion in Theorem \ref{thm:tight2} is proven.
\end{proof}

\begin{rem}
The existence of the continuous process $X$ obtained in Theorem \ref{thm:tight2} is only for almost all initial point $X_0$ with respect to the probability measure $\mu$ which we obtained as a limit measure of the $\{ \mu _N\}$.
The exceptional set appears, because we give the initial distribution of $X^N$ by the specific measure $\mu _N$. 
\end{rem}

\begin{rem}
The state space of $X$ obtained in Theorem \ref{thm:tight2} is $B_{4/3}^{-1/2-\varepsilon}$.
The index $-1/2-\varepsilon$ for the differentiability seems to be optimal.
However, the index $4/3$ for integrability is not expected to be optimal, in fact higher integrability for the process is obtained in \cite{MW3}.
By following the argument in \cite{MW3} we may improve also in our case the integrability index of the state space.
\end{rem}

\begin{rem}
In the present paper, we proved only the existence of a continuous limit process and of an associated invariant measure.
However, the uniqueness of the limit process in some classes of approximations is expected to hold, because in Theorem 1.15 in \cite{Ha} and Theorem 3.1 in \cite{CaCh} a contractive map from the polynomials of the Ornstein-Uhlenbeck process to the unique local solution has been obtained.
It seems thus possible to show this kind of uniqueness in our approach by adapting the arguments in \cite{Ha} and \cite{CaCh} to our setting.
\end{rem}

\begin{rem}
In the present paper, we only considered the approximation of the $\Phi ^4_3$-measure by finite sums in a Fourier series expansion.
However, a small modification of the proof yields the same result with other spatial regularization as for the process discussed in \cite{CaCh}.
\end{rem}

\begin{cor}\label{cor:tight2}
The limit process $X$ that appeared in Theorem \ref{thm:tight2} can be regarded as a $B_2^{-3/4}$-valued continuous process.
\end{cor}

\begin{proof}
As in the proof of Theorem \ref{thm:tight2}, for $T\in (0,\infty )$ and $t_0 \in (0,T)$, the laws of $\{ X^{N,(2)}\}$ on $C([t_0,T]; {L^{4/3}})$ are tight.
Hence, by the Besov embedding theorem, the laws of $\{ X^{N,(2)}\}$ are also tight as the probability measures on $C([t_0,T]; {B_2^{-3/4}})$.
The rest of the proof follows similarly as in the proof of Theorem \ref{thm:tight2}.
\end{proof}

\section*{Acknowledgment}
The authors are grateful to Professors Massimiliano Gubinelli, Hiroshi Kawabi and Minoru Yoshida for helpful discussions and valuable comments.
They also thank Professors Alexei Daletskii, Yuzuru Inahama, Hendrik Weber, Masato Hoshino, Rongchan Zhu and Xiangchan Zhu for pointing out mistakes in previous versions of this work.
This work was initiated during a common stay at the Centre Interfacultaire Bernoulli (CIB) of the Ecole Polytechnique F\'ed\'erale in Lausanne in  2015 during the semester program (organized by Professors Ana Bela Cruzeiro and Darryl Holm, with the first named author) on Geometric Stochastic Mechanics.
We are grateful to the Centre and its Director, Professor Nicolas Monod for their great hospitality.
The first named author gratefully acknowledges financial support by CIB at EPFL, and the second named author support by Professor Masayoshi Takeda for the CIB stay, HCM for a stay in Bonn in March 2016, and the JSPS grants KAKENHI Grant numbers 25800054 and 17K14204 for financial support for the research activity.

\appendix
\section{Appendix}

\subsection{Almost sure convergence of continuous stochastic processes}

\begin{prop}\label{prop:AC}
Let $(S,d)$ be a separable metric space, and let $X^n, X$ be $S$-valued continuous stochastic processes on a probability space $(\Omega, {\mathcal F}, P)$.
Assume that the family of the laws of $\{ X^n\}$ is tight as a family of probability measures on $C([0,\infty);S)$, and that $X_t^n$ converges to $X_t$ almost surely for $t\in [0,T]$.
Then, $X^n$ converges to $X$ almost surely in $C([0,\infty);S)$ with the topology of uniform convergence on finite intervals.
\end{prop}

\begin{proof}
Let $T>0$ and $\varepsilon >0$.
For $m,p \in {\mathbb N}$ define $\Omega _{m,p}$ by the total set of all $\omega \in {\Omega }$ satisfying
\[
\sup _{s,t \in [0,T]; |s-t|<1/m} d(X_s(\omega ), X_t(\omega ))< \frac 1p, \ \mbox{and}\ \sup _{n\in {\mathbb N}} \sup _{s,t \in [0,T]; |s-t|<1/m} d(X^n_s(\omega ), X^n_t(\omega )) < \frac 1p.
\]
Because of the tightness of $\{ P \circ X^{-1} \} \cup \{ P \circ (X^n)^{-1}; n\in {\mathbb N}\}$ on $C([0,\infty);S)$, for $P$-almost every $\omega \in \Omega$, $\{ X(\omega )\} \cup \{ X^n(\omega ); n\in {\mathbb N}\}$ is equi-continuous on $[0,T]$.
Hence, we have 
\begin{equation}\label{eq:propA1-1}
P \left( \bigcup _{m=1}^\infty \Omega _{m,p} \right) =1, \quad p\in {\mathbb N}.
\end{equation}
Let $K_m:= \min \{ k\in {\mathbb N}; k > m T\}$ for any $m\in {\mathbb N}$.
Since by assumption $X_t^n$ converges to $X_t$ almost surely for $t\in [0,T]$, for each $m\in {\mathbb N}$ there exists a $P$-null set $N_m$ such that
\begin{equation}\label{eq:propA1-2}
\lim _{n\rightarrow \infty} \max _{k =1,2,\dots ,K_m} d(X^n_{k/m}(\omega ), X_{k/m}(\omega )) =0, \quad \omega \in \Omega \setminus N_m.
\end{equation}
On the other hand, for $m\in {\mathbb N}$ and $\omega \in \Omega _{m,p}$ we have
\begin{align*}
&\sup _{t\in [0,T]}d(X^n_t(\omega ), X_t(\omega ))\\
&= \max _{k =1,2,\dots ,K_m} \sup _{t\in [(k-1)/m, k/m)} d(X^n_t(\omega ), X_t(\omega ))\\
&\leq \max _{k =1,2,\dots ,K_m} \left( d(X^n_{k/m}(\omega ), X_{k/m}(\omega )) + \sup _{t\in [(k-1)/m, k/m)} d(X^n_t(\omega ), X^n_{k/m}(\omega )) \right. \\
&\quad \hspace{6cm} \left. + \sup _{t\in [(k-1)/m, k/m)} d(X_t(\omega ), X_{t/m}(\omega )) \right)\\
&< \max _{k =1,2,\dots ,K_m} d(X^n_{k/m}(\omega ), X_{k/m}(\omega )) + \frac 2p.
\end{align*}
Hence, by (\ref{eq:propA1-2}), for $p \in {\mathbb N}$ and $\omega \in \cup _{m=1}^\infty (\Omega _{m,p} \setminus N_m)$,
\[
\limsup _{n\rightarrow \infty} \sup _{t\in [0,T]}d(X^n_t(\omega ), X_t(\omega )) \leq \frac 2p.
\]
Therefore, by (\ref{eq:propA1-1}) we obtain
\begin{align*}
P \left( \left\{ \lim _{n\rightarrow \infty} \sup _{t\in [0,T]}d(X^n_t, X_t) =0\right\} ^c \right)
&= P \left( \bigcup _{p=1}^\infty \left\{ \limsup _{n\rightarrow \infty} \sup _{t\in [0,T]}d(X^n_t(\omega ), X_t(\omega )) > \frac 2p\right\} \right) \\
&\leq \sum _{p=1}^\infty P \left(\limsup _{n\rightarrow \infty} \sup _{t\in [0,T]}d(X^n_t(\omega ), X_t(\omega )) > \frac 2p \right) \\
&\leq \sum _{p=1}^\infty \left[ 1- P \left( \bigcup _{m=1}^\infty (\Omega _{m,p} \setminus N_m ) \right) \right]\\
& =0.
\end{align*}
\end{proof}

\subsection{Moments of multidimensional Gaussian random variables}\label{secAG}

\begin{prop}\label{prop:AG1}
Let $n\in {\mathbb N}$ and let $(X_1, X_2, \dots , X_{2n})$ be a $2n$-dimensional Gaussian random vector with real-valued components.
Then, we have
\begin{align*}
E\left[ \prod _{i=1}^{2n} X_i \right]
&= \sum_{i=0}^n \frac{1}{(2i)!\, (n-i)!\, 2^{n-i}} \\
&\quad \times \sum _{\sigma \in {\mathfrak S}_{2n}} \left( \prod _{j=1}^{2i} E\left[ X_{\sigma(j)} \right]\right) \left( \prod _{j=i+1}^n {\rm Cov }(  X_{\sigma(2j-1)}, X_{\sigma(2j)}) \right)
\end{align*}
\end{prop}

\begin{proof}
It is well-known that for $m\in {\mathbb N}$ and $m$-dimensional real Gaussian vector $(Y_1,Y_2, \dots Y_m)$ it holds that
\[
E\left[ \prod _{i=1}^{m} (Y_i -E[Y_i])\right] =
\left\{ \begin{array}{cl}
\displaystyle \frac{1}{(m/2)!\, 2^{m/2}} \sum _{\sigma \in {\mathfrak S}_{m}} \prod _{i=1}^{m/2} {\rm Cov}(Y_{\sigma (2i-1)}, Y_{\sigma (2i)}), & m\ \mbox{:even} \\
\displaystyle 0, &m\ \mbox{:odd}
\end{array} \right.
\]
(see Proposition I.2 in \cite{Si}).
Applying this formula, we have
\begin{align*}
&E\left[ \prod _{i=1}^{2n} X_i \right] =E\left[ \prod _{i=1}^{2n} (X_i -E[X_i] + E[X_i]) \right]\\
&= \sum _{i=0}^{2n} \frac{1}{i!\, (2n-i)!} \sum _{\sigma \in {\mathfrak S}_{2n}} \left( \prod _{j=i+1}^{2n} E[X_{\sigma (j)}] \right) E\left[ \prod _{j=1}^{i} (X_{\sigma (j)} -E[X_{\sigma (j)}]) \right] \\
&= \sum _{i=0}^{n} \frac{1}{(2i)!\, (2n-2i)!} \sum _{\sigma \in {\mathfrak S}_{2n}} \left(\prod _{j=2i+1}^{2n} E[X_{\sigma (j)}] \right) \frac{1}{i!\, 2^i} \sum_{ \tau \in {\mathfrak S}_{2i}} \prod _{j=1}^i {\rm Cov }(  X_{\tau \circ \sigma(2j-1)}, X_{\tau \circ \sigma (2j)})\\
&= \sum _{i=0}^{n} \frac{1}{i!\, (2n-2i)!\, 2^i} \sum _{\sigma \in {\mathfrak S}_{2n}} \left(\prod _{j=2i+1}^{2n} E[X_{\sigma (j)}] \right) \left( \prod _{j=1}^i {\rm Cov }( X_{\sigma(2j-1)}, X_{\sigma (2j)}) \right)  .
\end{align*}
By changing $i$ for $n-i$ in the sum, we obtain the assertion.
\end{proof}

Now we consider a complex-valued version of Proposition \ref{prop:AG1}.
For square-integrable complex-valued random variables $Z_1$, $Z_2$ we define ${\rm Cov}(Z_1, Z_2)$ by
\[
{\rm Cov}(Z_1, Z_2) := E[(Z_1-E[Z_1])(Z_2-E[Z_2])].
\]

\begin{thm}\label{thm:AG}
Let $n\in {\mathbb N}$ and let $(X_1, Y_1, X_2, Y_2, \dots , X_{2n}, Y_{2n})$ be a $4n$-dimensional Gaussian random vector.
Then, we have
\begin{align*}
E\left[ \prod _{i=1}^{2n} (X_i + \sqrt{-1} Y_i) \right] &= \sum_{i=0}^n \frac{1}{(2i)!\, (n-i)!\, 2^{n-i}} \sum _{\sigma \in {\mathfrak S}_{2n}} \left( \prod _{j=1}^{2i} E\left[ X_{\sigma(j)} + \sqrt{-1} Y_{\sigma (j)}\right]\right) \\
&\quad \times \left( \prod _{j=i+1}^n {\rm Cov }(  X_{\sigma(2j-1)} + \sqrt{-1} Y_{\sigma (2j-1)}, X_{\sigma(2j)} + \sqrt{-1} Y_{\sigma (2j)}) \right)
\end{align*}
\end{thm}

\begin{proof}
Define a $4n$-dimensional real-valued Gaussian random vector\\
$(Z_1,Z_2,\dots , Z_{2n}, Z_{-1}, Z_{-2},\dots , Z_{-2n})$ and a $4n$-dimensional complex-valued Gaussian random vector $(\tilde Z_1, \tilde Z_2,\dots , \tilde Z_{2n}, \tilde Z_{-1}, \tilde Z_{-2},\dots , \tilde Z_{-2n})$ by
\begin{align*}
Z_i &:= \left\{ \begin{array}{rl}
X_i ,& i=1,2,\dots , 2n,\\
Y_{-i} ,& i=-1, -2, \dots , -2n,
\end{array}\right.
\\
\tilde Z_i &:= \left\{ \begin{array}{rl}
X_i ,& i=1,2,\dots , 2n,\\
\sqrt{-1} Y_{-i} ,& i=-1, -2, \dots , -2n,
\end{array}\right.
\end{align*}
respectively.
Then, by Proposition \ref{prop:AG1} we have
\begin{align*}
&E\left[ \prod _{i=1}^{2n} (X_i + \sqrt{-1} Y_i) \right] \\
&= \sum _{\epsilon = (\epsilon _k; k=1,2,\dots ,2n)\in \{ \pm 1\} ^{2n}} (\sqrt{-1})^{\#\{ j=1,2,\dots, 2n; \epsilon _j <0\}} E\left[ \prod _{i=1}^{2n} Z_{\epsilon _i i}\right] \\
&= \sum _{\epsilon = (\epsilon _k; k=1,2,\dots ,2n)\in \{ \pm 1\} ^{2n}} (\sqrt{-1})^{\#\{ j=1,2,\dots, 2n; \epsilon _j <0\}} \sum_{i=0}^n \frac{1}{(2i)!\, (n-i)!\, 2^{n-i}} \\
&\hspace{2cm} \quad \times \sum_{\sigma \in {\mathfrak S}_{2n}} \left( \prod _{l=1}^{2i} E\left[ Z_{\epsilon _l \sigma (l)}\right] \right) \left( \prod _{l=i+1}^n{\rm Cov}( Z_{\epsilon _{2l-1} \sigma (2l-1)}, Z_{\epsilon _{2l} \sigma (2l)})\right) \\
&= \sum_{i=0}^n \frac{1}{(2i)!\, (n-i)!\, 2^{n-i}} \\
&\quad \times \sum _{\sigma \in {\mathfrak S}_{2n}} \sum _{\epsilon = (\epsilon _k; k=1,2,\dots ,2n)\in \{ \pm 1\} ^{2n}} \left( \prod _{l=1}^{2i} E\left[ \tilde Z_{\epsilon _l \sigma (l)}\right] \right) \left( \prod _{l=i+1}^n{\rm Cov}( \tilde Z_{\epsilon _{2l-1} \sigma (2l-1)}, \tilde Z_{\epsilon _{2l} \sigma (2l)})\right) \\
&= \sum_{i=0}^n \frac{1}{(2i)!\, (n-i)!\, 2^{n-i}} \sum _{\sigma \in {\mathfrak S}_{2n}} \left( \prod _{j=1}^{2i} E\left[ X_{\sigma(j)} + \sqrt{-1} Y_{\sigma (j)}\right]\right) \\
&\hspace{3cm} \quad \times \left( \prod _{j=i+1}^n {\rm Cov }(  X_{\sigma(2j-1)} + \sqrt{-1} Y_{\sigma (2j-1)}, X_{\sigma(2j)} + \sqrt{-1} Y_{\sigma (2j)}) \right) .
\end{align*}
\end{proof}

\subsection{Tightness of the direct product of tight families}

Let $S_1$ and $S_2$ be metric spaces, $S_1\times S_2$ be the product space of $S_1$ and $S_2$ and $\pi _i$ be a projection on $S_1\times S_2$ to $S_i$ for $i=1,2$.
We remark that $S_1\times S_2$ is a metrizable topological space.

\begin{prop}\label{prop:AT}
Let $\{ P_\lambda \}$ be a family of probability measures on $S_1\times S_2$.
If the family $\{ P_{\lambda } \circ \pi _i ^{-1} \}$ is tight as probability measures on $S_i$ for $i=1,2$, then $\{ P_\lambda \}$ is tight on $S_1\times S_2$.
\end{prop}

\begin{proof}
For $\varepsilon \in (0,1]$ there exists compact sets $K_1$ and $K_2$ in $S_1$ and $S_2$ such that for $\lambda \in \Lambda$
\[
P_{\lambda } \circ \pi _1 ^{-1}(K_1) >1- \frac{\varepsilon}{2}, \quad P_{\lambda } \circ \pi _2 ^{-1}(K_2) >1-\frac{\varepsilon}{2}
\]
respectively.
Hence, for $\lambda \in \Lambda$
\begin{align*}
P_{\lambda }(K_1\times K_2) &= P((K_1 \times S_2) \cap (S_1\times K_2))\\
&\geq 1- P(K_1 \times S_2) - P(S_1 \times K_2) \\
&>1-\varepsilon .
\end{align*}
Since the compactness is equivalent to the sequential compactness on metric spaces and the product of sequentially compact sets is also sequentially compact, $K_1 \times K_2$ is a compact set in $S_1\times S_2$.
Therefore, the assertion holds.
\end{proof}

\begin{cor}\label{cor:AT}
Let $B$ be a Banach space.
Let $\{ X_\lambda^{(1)} \}$ and $\{ X_\lambda^{(2)} \}$ be families of $B$-valued random variables on a probability space.
If the laws of $\{ X_\lambda^{(1)} \}$ and $\{ X_\lambda^{(2)} \}$ are tight, then the laws of the pairs $\{ (X_\lambda^{(1)}, X_\lambda^{(2)} )\}$ are also tight as probability measures on $B\times B$.
In particular, the laws of $\{ X_\lambda^{(1)} + X_\lambda^{(2)} \}$ are tight as probability measures on $B$.
\end{cor}

\begin{proof}
The assertions follow from Proposition \ref{prop:AT} and the continuity of the mapping
\[
f: B\times B \rightarrow B, \quad f(x,y) = x+y ,\ (x,y\in B).
\]
\end{proof}

\subsection{Existence of invariant measures for stationary Markov processes}

\begin{prop}\label{prop:Ainv}
Consider a Markov process $(X_t^x; t\in [0,\infty ) )$ on a topological space $S$ and denote the process $X_\cdot$ with initial distribution $\nu$ by $X^\nu _\cdot$.
If the family of probability measures
\[
\left\{ P\circ (X_t^\nu )^{-1} ; t\in [0,\infty ) \right\}
\]
is tight for a probability measure $\nu$, then $X$ has an invariant probability measure.
\end{prop}

\begin{proof}
Since $\{ P\circ (X_t^\nu )^{-1}; t\in [0,\infty )\}$ is tight, the family $\{ \mu _t ; t\in (0,\infty )\}$ of probability measures on $(S,{\mathscr B}(S))$ defined by
\[
\mu _t (A) := \frac 1t \int _0^t P\circ (X_s^\nu )^{-1} (A) ds, \quad A\in {\mathscr B}(S)
\]
is also tight.
Hence, there exists a sequence $\{ t_n\} \subset (0,\infty )$ such that $\lim _{n\rightarrow \infty} t_n = \infty$ and $\mu _{t_n}$ converges to a probability measure $\mu$.
For $f\in C_b(S)$
\begin{align*}
E[f(X_t^\mu )] & = \lim _{n\rightarrow \infty} \frac{1}{t_n} \int _0^{t_n} E[f(X_{t+s}^\nu )] ds \\
&= \lim _{n\rightarrow \infty} \frac{1}{t_n} \left( \int _0^{t_n} E[f(X_s^\nu )] ds + \int _{t_n}^{t_n +t} E[f(X_s^\nu )] ds - \int _{0}^{t} E[f(X_s^\nu )] ds \right) \\
&= \int _S f d\mu .
\end{align*}
Therefore, $\mu$ is an invariant probability measure for $X$.
\end{proof}


\def\polhk#1{\setbox0=\hbox{#1}{\ooalign{\hidewidth
  \lower1.5ex\hbox{`}\hidewidth\crcr\unhbox0}}}

\end{document}